\documentclass[smallextended]{svjour3}       
\usepackage{longtable}

\usepackage[title]{appendix}

\smartqed  
\usepackage{mathrsfs}
\usepackage{amssymb,amsmath,amsfonts}
\usepackage{graphicx,subfigure,epstopdf}
\usepackage[usenames]{color}
\usepackage{url}
\usepackage{algorithm,algorithmic}
\usepackage{dsfont,verbatim}
\usepackage[colorlinks,linktocpage,linkcolor=blue]{hyperref}
\usepackage{enumerate,enumitem}
\usepackage[normalem]{ulem}
\usepackage{cite}
\usepackage{bm}
\allowdisplaybreaks

\numberwithin{theorem}{section}
\numberwithin{lemma}{section}
\numberwithin{remark}{section}
\numberwithin{equation}{section}

\def\bfx{{\bf x}}
\def\bfe{{\bf e}}

\def\nu{n}

\def\vb{{\bf v}}

\def\d{{\mathrm d}}

\def\R{{\mathbb R}}

\def\ud{\underline{D}}
\def\md{\partial_{t}^\bullet}

\def\ehm{{\hat{e}_h^m}}
\def\eM{{e_h^{m+1}}}
\def\ehM{{\hat{e}_h^{m+1}}}

\def\Gm{{\Gamma^m}}

\def\Ghm{{\Gamma^m_h}}
\def\GhM{{\Gamma^{m+1}_h}}
\def\Ghsm{{\hat\Gamma^m_{h, *}}}
\def\GhsM{{\hat\Gamma^{m+1}_{h, *}}}

\def\Ghso{{\Gamma_{h,\rm f}^0}}

\def\nhm{{n^m_h}}

\def\nhsm{{\hat n^m_{h, *}}}

\def\nbhsm{{\bar n^m_{h, *}}}
\def\nbhsM{{\bar n^{m+1}_{h, *}}}
\def\nsm{{n^m_{*}}}
\def\nsM{{n^{m+1}_{*}}}

\def\Thsm{{\hat T^m_{h, *}}}

\def\Tbhsm{{\bar T^m_{h, *}}}

\def\Tsm{{T^m_{*}}}
\def\TsM{{T^{m+1}_{*}}}

\def\Nhsm{{\hat N^m_{h, *}}}

\def\Nsm{{N^m_{*}}}

\def\Nbhsm{{\bar N^m_{h, *}}}

\def\Tbhm{{\bar{T}^m_h}}
\def\Nbhm{{\bar{N}^m_h}}
\def\nbhm{{\bar{n}^m_h}}

\def\km{{\kappa_m}}
\def\ksm{{\kappa_{*,m}}}

\def\kl{{\kappa_l}}
\def\ksl{{\kappa_{*,l}}}

\numberwithin{equation}{section}

\begin{document}
\title{Convergence analysis for the Barrett--Garcke--N\"urnberg method of transport type for evolving curves}
\titlerunning{\,}        

\author{Genming Bai \and Harald Garcke \and \\ Shravan Veerapaneni 
}

\authorrunning{\,} 

\institute{G. Bai and S. Veerapaneni \at
	Department of Mathematics, University of Michigan, Ann Arbor, USA. 
	Email address: gbai@umich.edu, shravan@umich.edu\\
	\and
	H. Garcke \at
	Fakultät für Mathematik, Universit\"at Regensburg, Regensburg, Germany. 
	Email address:
	harald.garcke@ur.de
}

\date{Received: date / Accepted: date}

\allowdisplaybreaks

\maketitle
%

\begin{abstract}\noindent
{ 
In this paper, we propose a Barrett--Garcke--Nürnberg (BGN) method for evolving curves under a prescribed background velocity field in $\R^2$ and present the corresponding convergence analysis. 
Unlike mean curvature flow and surface diffusion, where the evolution velocities inherently exhibit parabolicity, this case is dominated by transport which poses a significant difficulty in establishing convergence proofs. To address the challenges imposed by this transport-dominant nature, we derive several discrete energy estimates of the transport type on discretized polynomial curves within the framework of the projection error. The use of the projection error is indispensable as it provides crucial additional stability through its orthogonality structure. We prove that the proposed method converges sub-optimally in the $L^2$ norm, and this is the first convergence proof for a fully discrete numerical method solving the evolution of curves driven by general flows.
\\
}
\end{abstract} 

\setlength\abovedisplayskip{3.5pt}
\setlength\belowdisplayskip{3.5pt}

\section{Introduction}\label{sec:intro}

 In this paper, we focus on the stability and convergence behavior of the BGN method in the scenarios where flows are dominated by transport. To be more precise, we are interested in the case where a closed curve $\Gamma(t)$ in $\R^2$ is evolving under an arbitrarily prescribed background velocity field $u(x,t)$ in $\R^2\times [0, T]$. We denote the parameterized flow map along $u$ by $X(t):\Gamma(0)\rightarrow \Gamma(t)$, which satisfies the velocity equation
\begin{align}\label{eq:evol}
	\partial_t X(\cdot, t) = u(X(\cdot, t), t)  \mbox{ on } \Gamma(0) \mbox{ for } t \in  [0,T],
\end{align}
with the initial condition $X(x, 0) = x$ for $x\in \Gamma(0)$.

The evolution equation \eqref{eq:evol} plays a crucial role in numerous applications, especially the sharp interface models, where the evolution of the interface is determined by a nonlinear coupling of bulk and interface dynamics. This includes problems of solving: the evolution of a surface or bulk domain with a moving boundary using the arbitrary Lagrangian–Eulerian (ALE) methods \cite{ER15, ER21, NF99, LS21, LXY23, RWX23}; moving interface problem driven by curvature quantities \cite{Fu20, Hysing09, BGN15a, BGN15b, BGN16, DLY22}; PDE-constrained shape optimization \cite{GLR24}; fluid-structure interactions \cite{Richter17}.
Therefore, developing robust and convergent discretization methods for \eqref{eq:evol} is highly desirable.

The parametric finite element approach to geometric flows was first introduced in Dziuk's seminal paper \cite{Dziuk1990a} and has been further developed over the years \cite{Deckelnick2005, Dziuk2013b, KLL19, Dem09, BGN2007JCP, BGN2008JCP}. The idea of the parametric finite element method is to use vector-valued finite element functions to track the graph of a surface within its ambient geometry. Since we are primarily concerned about the graph rather than the trajectory, there is an additional degree of freedom to choose the tangential velocity. By choosing a suitable tangential velocity, the finite element mesh can maintain high quality during long-term simulations. Such tangential smoothing velocities can be constructed by minimizing the deformation rate functional $\int_\Gamma |\nabla_\Gamma v|^2$ (see \cite{BGN2007JCP, BGN2008JCP, BL24, BGN_survey}), minimizing the deformation functional $\int_\Gamma |\nabla_\Gamma X|^2$ (see \cite{DL24}), and reparametrization \cite{Elliott-Fritz-2017, Bartels13, KMS24, MMNI16, MRI21, TRM18, VRBZ11}.
Improvements in nodal distribution can also be achieved by prescribing the tangential velocity \cite{ES12, EV15, Morigi10} or by considering the equilibrium of a spring model \cite{Kovacs19}.

The parametric finite element methods of the Barrett--Garcke--N\"urnberg (BGN) type have been successful in approximating the evolution of curves and surfaces under various geometric flows, including the flows of parabolic type, e.g. mean curvature flow \cite{BGN2008JCP} and surface diffusion \cite{BGN2007JCP}, and as well as the flows of transport type including two-phase flow for bubbles \cite{BGN15a, BGN15b, DLY22} and biomembranes \cite{BGN16}. 
A key feature of the BGN method is its implicitly defined tangential velocity, which helps maintain an equal distribution of mesh nodes. It has been rigorously shown in \cite{BL24} that the BGN velocity for curve shortening flow converges to the minimizer of the deformation rate functional $\int_\Gamma |\nabla_\Gamma v|^2$.

For the discretization of \eqref{eq:evol}, we propose the following fully discrete BGN system of transport type:  
Given a approximate polynomial curve $\Gamma_h^{m}$ at time level $t=t_m$, find a polynomial parametrization $X^{m+1}_h : \Gamma_h^{m}\rightarrow \R^2$ and a scalar finite element function $\eta_h^{m+1} : \Gamma_h^{m}\rightarrow \R$ satisfying the following weak formulation for all $(\chi_h,\phi_h)\in S_h(\Gamma_h^m)\times S_h(\Gamma_h^m)^2$: 
\begin{align}
	\int_{\Gamma_h^m}^h \frac{X_h^{m+1}-{\rm id}}{\tau} \cdot \bar n_h^m \chi_h
	&=
	\int_{\Gamma_h^m}^h u(t_m)|_{\Ghm} \cdot \bar n_h^m  \chi_h,
	\label{eq:BGN_tr1}\\
	\int_{\Gamma_h^m} \nabla_\Ghm X_h^{m+1} \cdot  \nabla_\Ghm \phi_h
	&=
	\int_{\Gamma_h^m}^h \eta_h^{m+1} \nbhm \cdot \phi_h
	\notag\\
	&\quad+
	\int_{\Gamma_h^m} \nabla_{\Gamma_h^m} {\rm id} \cdot  \nabla_{\Gamma_h^m} I_h[\phi_h - (\phi_h \cdot \bar n_h^m)\bar n_h^m],
	\label{eq:BGN_tr2}
\end{align}
where the superscript $h$ denotes the mass lumping integral (see \eqref{eq:lumped_mass} for the definition), $\nbhm$ is the averaged normal vector on $\Ghm$ (see \eqref{eq:bar_n} for the definition), and $S_h(\Gamma_h^m)$ is a space of scalar-valued finite element functions on $\Ghm$. These concepts will be defined in detail in Section \ref{section:results} later. Please also refer to Appendix \eqref{sec:disc-norm}--\eqref{sec:super1} for their basic properties. The second term on the right-hand-side of \eqref{eq:BGN_tr2} serves as the source of stabilization inspired by \cite[Eq. (1.5)]{BL24}. This term is crucial when deriving the tangential stability estimates in Section \ref{sec:tan_stab}. It is straightforward to show this additional term vanishes approximately:
\begin{align}
	\int_{\Gamma_h^m} \nabla_{\Gamma_h^m} {\rm id} \cdot  \nabla_{\Gamma_h^m} I_h[\phi_h - (\phi_h \cdot \bar n_h^m)\bar n_h^m]
	&\approx
	-\int_{\Gamma_h^m} \Delta_{\Gamma_h^m} {\rm id} \cdot  I_h[\phi_h - (\phi_h \cdot \bar n_h^m)\bar n_h^m] \notag\\
	&=
	\int_{\Gamma_h^m} H_h^m n_h^m \cdot  I_h[\phi_h - (\phi_h \cdot \bar n_h^m)\bar n_h^m]
	\approx
	0
	 \notag,
\end{align}
thus ensuring its consistency, where $H_h^m$ and $n_h^m$ are the mean curvature and unit normal vector on $\Ghm$ and we have used the identity $-\Delta_{\Ghm} {\rm id} = H_h^m n_h^m$ (see \cite[Eq. (11.9)]{Gia13}).
This stabilization term is helpful when we test \eqref{eq:BGN_tr2} with some almost tangential test function $I_h(I - \nbhm\otimes\nbhm)\phi_h \in S_h(\Ghm)^2$ to get
\begin{align}
	\int_{\Gamma_h^m} \nabla_\Ghm (X_h^{m+1} - X_h^m) \cdot  \nabla_\Ghm I_h(I - \nbhm\otimes\nbhm)\phi_h
	=
	0
	\qquad \phi_h \in S_h(\Ghm)^2
	 . \notag
\end{align}
Taking $\phi_h = X_h^{m+1} - X_h^m$,
the continuous analogue of the identity above writes
\begin{align}
	\int_{\Gamma} \nabla_\Gamma v \cdot  \nabla_\Gamma [(I - n\otimes n)v]
	=
	0
	, \notag
\end{align}
which furthermore implies the crucial tangential stability estimate (cf. the derivations of \cite[Eqs. (1.11)--(1.14)]{BL24}):
\begin{align}
	\int_{\Gamma} | \nabla_\Gamma [(I - n\otimes n)v] |^2
	\leq
	C
	\int_\Gamma |v\cdot n|^2
	. \notag
\end{align} 

The primary challenge in analyzing \eqref{eq:BGN_tr1}--\eqref{eq:BGN_tr2} lies in the absence of the $H^1$-positive definite bilinear form, i.e. $\int_\Ghm \nabla_\Ghm X_h^{m+1} \cdot \nabla_\Ghm \phi_h$, on the left-hand-side of \eqref{eq:BGN_tr1}. This is in sharp contrast to Dziuk's method \cite[Eq. (7)]{Dziuk1990a}, as well as the BGN method for mean curvature flow \cite[Eq. (2.24)]{BGN2008JCP} and surface diffusion \cite[Eqs. (2.2a), (2.2b)]{BGN2007JCP}. This loss of discrete $H^1$ parabolicity is a result of the transport-dominant nature of the underlying flow \eqref{eq:evol}. Nevertheless, the absence of the stiffness bilinear form in \eqref{eq:BGN_tr1} saves us from using the inverse inequality to control gradients of errors.
Additionally, by employing the projection error framework (cf. \cite[Section 3]{BL24FOCM}), we gain extra stability through the use of super-approximation estimates.
In summary, at the discrete level, several competing factors influence both stability and instability:
\begin{itemize}
	\item Instability: The lack of $H^1$ parabolicity at the continuous level prevents us from controlling $L_t^2 H_x^1$ norms in the energy estimates. As a result, it is essential to ensure that each component of the error remains $L_t^\infty L_x^2$ stable.
	This is particularly challenging because the errors associated to normal vectors (see Lemma \ref{lemma:ud}, Item 7) and surface discrepancies (see Lemma \ref{lemma:ud}, Item 6) involve gradients. In Section \ref{sec:refined}, we will eliminate this gradient dependence by utilizing the orthogonality of the projection error.
	\item Instability: There is no guarantee of the a priori boundedness of the shape regularity constants.
	To address this, in Section \ref{sec:bbd}, we carefully track the leading-order norm dependence of the parameterization map and apply a Gr\"onwall-type argument. 
	\item Stability: Fewer uses of the inverse inequality lead to an improvement in the velocity estimates (cf. Section \ref{sec:bbd_vel}).
	\item Stability: Improved stability estimates are achieved through the orthogonality structure within the framework of projection error (cf. Section \ref{sec:refined}). 
	Particularly, we uncover a crucial local integration-by-parts formula for the surface distortion factor
	\begin{align}\label{eq:NT-div}
		\nabla_\Gamma\cdot f
		=
		\nabla_\Gamma\cdot (N f)
		+
		\nabla_\Gamma\cdot (T f)
		=
		(\nabla_\Gamma\cdot N) f
		\qquad\mbox{for $Nf=f,f\in H^1(\Gamma)^2$},
	\end{align}
	where $N$ and $T$ are the normal and tangential projections on $\Gamma$. When $f$ is chosen to be the projection error, which is almost normal by construction, \eqref{eq:NT-div} helps us regain $L^2$-stability in the error equation.
\end{itemize}
In the paper, after careful analysis, we are able to show the instability does not overwhelm the stability, making it possible to get a convergence proof.
Another notable finding is that, although the framework of projection error was initially introduced to recover the $H^1$ parabolicity structure (cf. \cite[Section 5.2]{BL24FOCM}), it also proves highly effective for addressing transport equations (cf. \eqref{eq:NT-div} and Section \ref{sec:refined}). This highlights that projection error remains a canonical notion of error for evolving curves and surfaces dominated by transport.

Another important contribution of the paper is a high-order tangential stability estimate (cf. Lemma \ref{lemma:tan_stab_H2}), proved by using an intrinsic $H^2$ stability result of the discrete Laplacian (cf. Lemma \ref{lemma:H2norm}). This result plays an important role in the induction argument of the shape regularity (cf. Section \ref{sec:bbd}) for the critical finite element degree $k=3$.

Along the convergence proof, we also show that the transport BGN velocity $(X_h^{m+1} - X_h^m)/\tau$ is consistent to the following elliptic velocity system on the exact curve $\Gamma$:
\begin{align}\label{eq:ell}
	\begin{split}
		v\cdot n&=  u\cdot n, \\
		-\Delta_\Gamma v&=\lambda n ,
	\end{split}
\end{align}
which is the Euler-Lagrange equation of the following minimization problem:
$$
\min_{v\in H^1(\Gamma)} \int_{\Gamma}|\nabla_\Gamma v|^2 
\quad\mbox{under the pointwise constraint $v\cdot n = u\cdot n$},
$$
confirming the intrinsic tangential smoothing effect.

To the best of our knowledge, our proof is the first to show the convergence of a fully discrete scheme for transport equations, standing out for the significance of our results. 
Meanwhile, this is the first convergence result for a BGN-type system -- a notable advancement beyond the one-line BGN scheme for curve shortening flow analyzed in \cite{BL24}.
Moreover, the proposed method \eqref{eq:BGN_tr1}--\eqref{eq:BGN_tr2} only relies on pointwise evaluations of the velocity field $u$, making our method robust in the low-regularity regime.
The results developed in this paper can hopefully be applied to moving interface problems driven by mean curvature \cite{BGN15a, BGN15b, DLY22} and surface diffusion \cite{BGN16,VRBZ11}.
Exploring how to extend these results to higher dimensions will also be a key focus of future research.

Regarding other convergence results of semidiscrete and fully discrete parametric finite element methods, we refer readers to \cite{DeckelnickDziuk, Deckelnick-Dziuk-2009,Li-2020-SINUM,Ye-Cui-SINUM, JSZ23} for curve shortening flow with $k=1$; \cite{BL24,Li-2020-SINUM} for curve shortening flow with $k\geq 2$; \cite{BL24FOCM,KLL19,BL22A,Li21} for mean curvature flow with $k\geq 2, 3$; \cite{KLL-Willmore} for Willmore flow with $k\geq 2$, and \cite{Bao-Zhao-2021-SINUM, Bao-Jiang-Wang-Zhao-2017, JL21, DLY22, BGN_survey} for unconditional stability results.

The rest of this paper is organized as follows: Section \ref{section:results} presents the proposed numerical scheme and the main convergence theorem. Section \ref{section:preparation} introduces the preliminaries, including notations, basic approximation results, the induction hypothesis, and geometric relations. Section \ref{sec:stab-est} and \ref{Proof-THM-1} provide the primary stability and error estimates for the proposed transport BGN method respectively. Numerical examples are presented in Section \ref{section:numerical}. We have moved some well-known results concerning parametric finite element methods and the projection error to Appendices \ref{sec:surf_calc} through \ref{sec:super}, and the proof of Lemma \ref{lemma:NT_stab_ref} and \ref{lemma:e-convert} are provided in Appendix \ref{sec:appndix_tan_stab} and \ref{sec:e-convert} respectively.

\section{Numerical scheme and the main theorem}\label{section:results}


We begin by introducing several standard concepts related to the parametric finite element method (cf. \cite{Dem09}). Let $\Ghm$ be a closed curve that is globally continuous and can be parameterized piecewise by polynomial functions. The curve $\Ghm$ serves as an approximation to the smooth curve $\Gamma^m:=\Gamma(t_m)$, which evolves under a prescribed velocity field $u$. Each curved element $K$ of $\Ghm$ is the image of a curved element $K^0\subset\Gamma_h^0$ under the discrete flow map. We denote by $K_{\rm f}^0$ the unique flat segment sharing endpoints with $K^0$ and by $F_{K}:K_{\rm f}^0\rightarrow K$ the parametrization of $K$. Here, $F_K$ is the unique polynomial of degree $k$ that maps $K_{\rm f}^0$ onto $K$. The finite element space on the approximate curve $\Ghm$ is defined through the push-forward map $F_K$ as follows:
$$
S_h(\Gamma_h^m) = \{v_h\in C(\Gamma_h^m): v_h\circ F_K \in \mathbb{P}^k(K_{\rm f}^0)^2 \,\,\mbox{for every element}\,\, K\subset\Gamma_h^m\} ,
$$
where $\mathbb{P}^k(K_{\rm f}^0)$ denotes the space of polynomials of degree $k\ge 1$ on the flat segment $K_{\rm f}^0$.

For any piecewise continuous function $f\in C_{\rm pw}(\Ghm)$, we define its high-order mass lumping integral as follows, indicated by the superscript $h$, which assists in handling nodal-wise operations in the analysis:
\begin{align}\label{eq:lumped_mass}
	\int_{\Gamma_h^m}^h f := \sum\limits_{K\subset\Gamma_h^m} \int_{K_{\rm f}^0} I_h^{GL} \big[(f \circ F_K) |\nabla_{K_{\rm f}^0} F_K|\big] ,
\end{align}
where the summation includes all elements of the curve $\Ghm$, and $I_h^{GL}$ represents the interpolation operator at the Gauss--Lobatto points of the flat element $K_{\rm f}^0$ (cf. \cite[Eq. (10.2.3)]{Brenner08}). 
When the finite element degree $k=1$, the definition in \eqref{eq:lumped_mass} coincides with the definition in \cite[Eq. (2.2)]{BGN2008JCP}. 
In the rest of the paper, we will use the notation $I_h = I_h^{GL}$.

The averaged normal vector $\bar n_h^m \in S_h(\Ghm)^2$ is defined as the discrete $L^2$ projection of the piecewise continuous unit normal vector $n_h^m$ onto the finite element space $S_h(\Ghm)^2$, i.e., 
\begin{align}\label{eq:bar_n}
	\int_\Ghm^h \bar n_h^m\cdot \phi_h = \int_\Ghm^h n_h^m\cdot \phi_h 
	\quad\forall\, \phi_h\in S_h(\Ghm)^2 . 
\end{align}

Now we are in a good position to state the proposed numerical scheme for \eqref{eq:evol}.
Consider the sequence of grid points in time $t_m=m\tau$, where $m=0,1,\dots, [T/\tau]$, with a step size $\tau>0$, and $[T/\tau]$ is the greatest integer not exceeding $T/\tau$. We propose a transport-type BGN method as follows: Given a prescribed background velocity $u$ in $\R^2\times [0, T]$ and a polynomial curve $\Gamma_h^{m}$ at time level $t=t_m$ whose parameterization map is $X^{m}_h$, find the polynomial parameterization $X^{m+1}_h : \Gamma_h^{m}\rightarrow \R^2$ and a scalar finite element function $\eta_h^{m+1}: \Gamma_h^{m}\rightarrow \R$ for the next time level, satisfying the following weak formulation for all $(\chi_h,\phi_h)\in S_h(\Gamma_h^m)\times S_h(\Gamma_h^m)^2$:
\begin{align}
	\int_{\Gamma_h^m}^h \frac{X_h^{m+1}-{\rm id}}{\tau} \cdot \bar n_h^m \chi_h
	&=
	\int_{\Gamma_h^m}^h u(t_m)|_{\Ghm} \cdot \bar n_h^m  \chi_h,
	\label{eq:BGN_tr11}\\
	\int_{\Gamma_h^m} \nabla_\Ghm X_h^{m+1} \cdot  \nabla_\Ghm \phi_h
	&=
	\int_{\Gamma_h^m}^h \eta_h^{m+1} \nbhm \cdot \phi_h
	\notag\\
	&\quad+
	\int_{\Gamma_h^m} \nabla_{\Gamma_h^m} {\rm id} \cdot  \nabla_{\Gamma_h^m} I_h[\phi_h - (\phi_h \cdot \bar n_h^m)\bar n_h^m] .
	\label{eq:BGN_tr22}
\end{align}
\begin{remark}\upshape
	Note that the mass lumping integral is applied only to the $L^2$ inner product terms, and not to the stiffness and stabilization terms. This formulation follows the original BGN schemes \cite{BGN2007JCP,BGN2008JCP}. From an analytical perspective, employing mass lumping in the $L^2$ inner product terms is essential, as the cancellation structure $f_h g_h=0$ at all nodes $\mathcal N(\Ghm)$ frequently occurs in the analysis. This, in turn, yields
	\begin{align}
		\int_\Ghm^h f_h g_h = 0 . \notag
	\end{align}
	 On the other hand, for the stiffness terms, the super-approximation estimate does not hold. As a result, we can only establish 
	\begin{align}
		\bigg|\bigg(\int_\Ghm^h - \int_\Ghm \bigg)
		 \nabla_\Ghm f_h \cdot \nabla_\Ghm g_h \bigg| \leq
		 C_\Ghm
		  \| f_h \|_{H^1(\Ghm)} \| g_h \|_{H^1(\Ghm)} , \notag
	\end{align}
	which leads to the failure of Lemma \ref{lemma:NT_stab_ref} if the stiffness term on the left-hand side of \eqref{eq:NT_stab} is treated with the mass lumping integral. 
\end{remark}

Let $\delta>0$ be a sufficiently small constant such that every point $x$ in the $\delta$-neighborhood of the exact curve $\Gamma^m=\Gamma(t_m)$, denoted by $D_\delta(\Gamma^m)=\{x\in\R^2: {\rm dist}(x,\Gamma^m)\le \delta\}$, 
has a unique smooth projection of distance retraction onto $\Gamma^m$, denoted by $a^m(x)$, satisfying the following relation: 
\begin{align*}
	x - a^m(x) = \pm |x-a^m(x)| n^m(a^m(x)) ,
\end{align*}
where $n^m$ is the unit normal vector on $\Gamma^m$. It is known that such a constant $\delta$ exists and only depends on the curvature of $\Gamma^m$ (thus $\delta$ is independent of $m$, but possibly dependent on $T$); see \cite[Lemma~14.17]{GT2001} and \cite[Theorem 6.40]{Lee18}. 

We assume at $t=0$ that the polynomial parametrization map $F_{K^0}:K_{\rm f}^0\rightarrow K^0 \subset\Gamma_h^0$ has the following shape regularity property: 
\begin{align}\label{P0}
& \max_{K^0\subset\Gamma_h^0} 
\Big(\|F_{K^0}\|_{W^{k,\infty}(K_{\rm f}^0)} + \|\nabla_{K^0} F_{K^0}^{-1}\|_{L^\infty(K^0)} \Big) 
\le \kappa_0 ,
\end{align}
where $\kappa_0$ is some constant that is independent of $h$. This property holds for standard parametric finite elements which interpolate the smooth curve $\Gamma^0$ and guarantees the following optimal-order approximation to $\Gamma^0$ by $\Gamma_h^0$ (cf. \cite[Section 2.3]{Dem09}): 
\begin{align}\label{P1}
& \max_{K^0\subset\Gamma_h^0}  \| a^0 \circ F_{K^0} - F_{K^0} \|_{L^{\infty}(K_{\rm f}^0)} 
\le Ch^{k+1} .
\end{align}
The projection $a^0(x)$ is well defined for points $x$ in a neighborhood of $\Gamma^0$ and therefore well defined on $\Gamma_h^0$ for sufficiently small mesh size $h$. 

Then we define the nodal projection.
Let $x_j^m$, $j=1,\ldots,J$, be the nodes of the approximate curve $\Gamma_h^m$ at the time $t_m$ given by the transport BGN method in \eqref{eq:BGN_tr11}--\eqref{eq:BGN_tr22}. 
Following the framework of the projection error \cite{BL24FOCM}, we use ``hat $\string^$" to denote the quantities which are related to the nodal-wise projection. 
The projected piecewise polynomial curve $\Ghsm$ is uniquely determined by the projected nodes $\{a^m(x_j^m)\}_{j=1}^J$.
The error estimate \eqref{eq:err_est_1} ensures the projection $a^m(x_j^m)$ is well defined if the stepsize and mesh size are sufficiently small. Similar to $S_h(\Ghm)$, the finite element function space on $\Ghsm$, denoted by $S_h(\Ghsm)$, can be canonically defined in a parametric way via a push-forward polynomial map from $\Gamma_{h,\rm f}^0 := \cup_{K^0\subset\Gamma_h^0} K_{\rm f}^0$ to $\Ghsm$.

Following the notations in \cite[Section 1]{BL24FOCM}, we will always identify a finite element function by a vector of its nodal values. Such representation is unique if we have specified the underlying domain. For example, the two integrands of $$\int_{\Ghsm} \phi_h \quad\mbox{and}\quad \int_{\Ghm} \phi_h$$ have the same vector of nodal values, denoted by $\vb$, but are defined on different domains $\Ghsm$ and $\Ghm$. When the underlying domain is specified, $\vb$ is automatically substantialized to a finite element function $\phi_h$ on that domain. Since all of the quantitative computations in this paper involve either integrals or norms, our notations for finite element functions will always have a unique and clear meaning. For another example, $\| \phi_h \|_{\Ghsm}$ and $\| \phi_h \|_{\Ghm}$ denote the norms of a finite element function (a nodal vector) on the two different curves $\Ghsm$ and $\Ghm$, respectively.


We then define the finite element error function
$$ \hat e_h^{m}  \in H^1(\Ghsm) $$
which is uniquely determined by the nodal error $\{x_j^m - a^m(x_j^m)\}_{j=1}^J$. 
The error estimate for $\hat e_h^{m}$ is given in the following main theorem.

\begin{theorem}[Convergence of the transport BGN method]\label{thm:main}
	We assume that the flow map $ X: \Gamma^0\times[0,T]\rightarrow \R^2$ generated by \eqref{eq:evol} of a closed curve and its inverse map $X(\cdot,t)^{-1}:\Gamma(t)\rightarrow\Gamma^0$ are both sufficiently smooth, and the initial polynomial $\Gamma_h^0$ satisfies the approximation properties \eqref{P0}--\eqref{P1}. Let $\{X_h^m\}_{m=0}^{ [T/\tau] }$ be the finite element solutions given by the transport BGN method in \eqref{eq:BGN_tr11}--\eqref{eq:BGN_tr22}, subject to the initial condition $X_h^0 = {\rm id}$ on $\Gamma_h^0$. Then, for any given constant $c$ (independent of $\tau$ and $h$), there exists a positive constant $h_0$ such that for $\tau \le c h^{k}$ and $h\le h_0$ the following error estimate holds for finite elements of degree $k\ge 3${\rm:} 
	\begin{align}
		\label{eq:err_est_1}
		&\max_{1\le m\le [T/\tau]}
		\| \hat e_h^{m} \|_{L^2(\hat\Gamma_{h,*}^m)}
		\le C(\tau + h^{k}) , 
	\end{align}
	where the constant $C$ is independent of $\tau$ and $h$ {\rm(}but may depend on $c$ and $T${\rm)}.
\end{theorem}


\begin{remark}\label{rmk:init_app}\upshape 
	The condition \eqref{P1} implies the initial approximation error satisfies the estimate $\| \hat e_h^0 \|_{L^2(\Gamma_h^0)}\leq c_0 h^{k+1}$ for some constant $c_0$ which is independent of $h$.
\end{remark}

\begin{remark}\label{rmk:CFL}\upshape 
	The stepsize condition $\tau \le c h^{k}$ is required only in Section \ref{sec:bbd} to prove the shape regularity of the interpolated curve $\Ghsm$.
	In fact, this constraint is not essentially necessary (see Remark \ref{rmk:shape-reg}, where we show this condition can be removed up to the leading order) and is not observed in the numerical experiments (cf. Figure \ref{fig:b}).
\end{remark}

\begin{remark}\label{rmk:subopt}\upshape 
	Sub-optimality of $L^2$ convergence stems from the consistency error analysis in Section \ref{sec:cons_err} and is ubiquitous in other discretizations for transport equations (see \cite{CS98,JP86,Peterson91} and references therein). 
\end{remark}

\begin{remark}\upshape 
	The finite element degree condition $k\geq 3$ stems from the sub-optimal $L^2$ convergence and is required to ensure the boundedness of normal vectors in \eqref{eq:n-bbd}. Condition $k\geq 3$ is also sharp in the derivation of the tangential stability estimates in Section \ref{sec:tan_stab} and \ref{sec:bbd_vel} where we need a suitable induction hypothesis to ensure the smallness of nonlinear terms (also see Remark \ref{rmk:k3}). Such degree condition is common in finite element analysis for nonlinear equations, see \cite{Li-2019,BL24FOCM,BL24,KLL19}.
\end{remark}

\begin{remark}\upshape 
	Our current proof relies on the one-dimensional super-approximation estimate (Lemma \ref{lemma:super_conv2}), which is not applicable to general surface meshes, in order to establish the $L^2$ stability of the $J_{111}$ term in \eqref{eq:J11-decomp} and the tangential stability estimates in Section \ref{sec:tan_stab}.
	Nevertheless, extension to surfaces remains feasible if tensorial parametric finite elements are employed, as the construction of tensorial Gauss--Lobatto quadrature in such cases is straightforward.
\end{remark}

\section{Preliminaries}
\label{section:preparation}

In this section, we introduce some preliminaries of parametric finite element methods (cf. \cite{Dem09,KLL19,Kov18,Dziuk2013b}) and the projection error (cf. \cite{BL24FOCM,BL24}), including notations, basic approximation properties, norm equivalence results and geometric relations. To keep the presentation clean, we put some known results of surface calculus, discrete norms, averaged normal vectors and super-approximation estimates in the appendices.

\subsection{Notations}
\label{section:notation}
The following notations associated to the framework of the projection error will be frequently used in this article. They are similar to the notations in \cite[Section 3.1]{BL24FOCM} and are listed below for the convenience of the readers. 

\begin{longtable}{p{1.1cm}p{11cm}}
	$\Gamma^m$:
	& 
	The exact smooth curve at time level $t=t_m$.\\
	
	$\Gamma_h^m$:
	&
	The numerically computed curve at time level $t=t_m$.\\
	
	$\bfx^{m}$: 
	&
	The nodal vector $\bfx^m=(x_1^m,\dots,x_J^m)^\top$ consisting of the positions of nodes on $\Gamma_h^m$.\\
	
	$\hat\bfx_*^{m}$: 
	&
	The distance projection of $\bfx^{m}$ onto the exact curve $\Gamma^m$, i.e., $\hat\bfx_*^{m}=(\hat x_{1,*}^m,\dots,\hat x_{J,*}^m)^\top$ with $\hat x_{j,*}^m=a^m(x_j^m)$. \\
	
	$\bfx_*^{m+1}$: 
	&
	The new position of $\hat\bfx_*^{m}$ evolving {under the normal component of the prescribed velocity, i.e. $u(t)|_{\Gamma(t)}\cdot n(t) n(t)$}, (without additional tangential motion) from $t_m$ to $t_{m+1}$. \\
	
	$\Ghsm$: 
	&
	The piecewise polynomial curve which interpolates $\Gamma^m$ at the nodes in $\hat\bfx_*^{m}$.\\
	
	$\Gamma_{h,*}^{m+1}$: 
	&
	The piecewise polynomial curve which interpolates $\Gamma^{m+1}$ at the nodes in $\bfx_*^{m+1}$.\\
	
	$X_{h}^{m}$: 
	&
	The finite element function with nodal vector $\bfx^m$. It coincides with the identity map, i.e., ${\rm id}(x)=x$, when it is considered as a function on $\Gamma_{h}^m$. \\
	
	$X_{h}^{m+1}$: 
	&
	The finite element function with nodal vector $\bfx^{m+1}$. 
	When it is considered as a function on $\Gamma_{h}^m$, it represents the local flow map from $\Gamma_{h}^m$ to $\Gamma_{h}^{m+1}$.\\
	
	$\hat X_{h,*}^{m}$: 
	&
	The finite element function with nodal vector $\hat\bfx_*^m$. It coincides with the identity map, i.e., ${\rm id}(x)=x$, when it is considered as a function on $\Ghsm$. It coincides with the discrete flow map from $\hat\Gamma_{h,*}^0$ to $\Ghsm$ when it is considered as a function on $\hat\Gamma_{h,*}^0$.\\
	
	$X_{h,*}^{m+1}$: 
	&
	The finite element function with nodal vector $\bfx_*^{m+1}$. 
	When it is considered as a function on $\Ghsm$, it represents the local flow map from $\Ghsm$ to $\Gamma_{h,*}^{m+1}$.\\
	
	$X^{m+1}$: 
	&
	The local flow map from $\Gamma^m$ to $\Gamma^{m+1}$ under the flow $u(t)|_{\Gamma(t)}\cdot n(t) n(t)$. \\
	
	$\hat e_{h}^m$: 
	&
	The finite element error function with nodal vector $\hat\bfe^m=\bfx^m-\hat\bfx_*^m$.\\ 
	
	$e_{h}^{m+1}$: 
	&
	The auxiliary error function with nodal vector $\bfe^{m+1}=\bfx^{m+1}-\bfx_*^{m+1}$.\\ 
	
	$n^m$: 
	&
	The unit normal vector on $\Gamma^m$. \\
	
	$n^m_*$: 
	&
	The unit normal vector of $\Gamma^m$ inversely lifted to a neighborhood of $\Gamma^m$ (including $\Ghsm$), i.e., $n^m_*=n^m\circ a^m$. \\
	
	$\hat n_{h,*}^m$: 
	&
	The normal vector on $\Ghsm$. \\
	
	$\bar n_{h,*}^m$: 
	&
	The averaged normal vector on $\Ghsm$ is defined in \eqref{eq:bar_n*}, which is not necessarily unit. \\
	
	$n_h^m$: 
	&
	The normal vector on $\Gamma_{h}^m$. \\
	
	$\bar n_h^m$: 
	&
	The averaged normal vector on $\Gamma_{h}^m$ is defined in \eqref{eq:bar_n}, which is not necessarily unit. \\
	
	$\Nsm$: 
	&
	The normal projection operator $\Nsm=n^m_* (n^m_*)^\top$ on $\Ghsm$. \\
	
	$N^m$: 
	&
	The normal projection operator $N^m=n^m (n^m)^\top$ on $\Gamma^m$. 
	Thus $N^m$ is the lift of $\Nsm$ onto $\Gamma^m$, and $N_*^m$ is the extension of $N^m$ to a neighborhood of $\Gm$. \\
	
	$\Nhsm$: 
	&
	The normal projection operator $\Nhsm=\hat n^m_{h,*} (\hat n^m_{h,*})^\top$ on $\Ghsm$. \\
	
	$\Nbhsm$: 
	&
	The averaged normal projection operator $\Nbhsm= \frac{\bar n^m_{h,*}}{|\bar n^m_{h,*}|} (\frac{\bar n^m_{h,*}}{|\bar n^m_{h,*}|})^\top$ on $\Ghsm$. \\
	
	$T_*^m$: 
	&
	The tangential projection operator $T_*^m=I - n^m_* (n^m_*)^\top$ on $\Ghsm$. \\
	
	$T^m$: 
	&
	The tangential projection operator $T^m=I - n^m (n^m)^\top$ on $\Gamma^m$. 
	Thus $T^m$ is the lift of $\Tsm$ onto $\Gamma^m$. \\
	
	$\Thsm$: 
	&
	The tangential projection operator $\Thsm=I - \nhsm (\nhsm)^\top$ on $\Ghsm$. \\
	
	$\Tbhsm$: 
	&
	The averaged tangential projection operator $\Tbhsm=I - \frac{\bar n^m_{h,*}}{|\bar n^m_{h,*}|} (\frac{\bar n^m_{h,*}}{|\bar n^m_{h,*}|})^\top$ on $\Ghsm$. 
%
\end{longtable}

For the simplicity of notation, we shall denote by $I_h\Nbhsm\phi_h$ and $I_h\Tbhsm\phi_h$ the abbreviations of $I_h(\Nbhsm\phi_h)$ and $I_h(\Tbhsm\phi_h)$, respectively. Similar notations are also adopted for $I_h\Nhsm\phi_h$, $I_h\Nsm\phi_h$, $I_h\Thsm\phi_h$, $I_h\Tsm\phi_h$, and so on. 

Let's briefly revisit some basic notations introduced in Section \ref{section:results}. For a curved element $K$ of $\Ghsm$, we denote by $K^0\subset\Gamma_{h}^0$ the element mapped to $K$ through the discrete flow map $\hat X_{h,*}^m:\Gamma_{h}^0\rightarrow\Ghsm$. The parametrization of the element $K^0\subset\Gamma_h^0$ is given by $F_{K^0}: K_{\rm f}^0 \rightarrow K^0$, where $K_{\rm f}^0$ is the flat line segment with the same endpoints as $K^0$. These flat line segments $K_{\rm f}^0$ together form a piecewise linear curve
$$ 
\Gamma_{h,{\rm f}}^0=\bigcup_{K^0\subset\Gamma_h^0} K_{\rm f}^0 . 
$$
Up to the identification of nodal values introduced in Section \ref{section:results}, $\hat X_{h,*}^m:\Gamma_{h,{\rm f}}^0\rightarrow\Ghsm$ represents the unique piecewise polynomial parametrization of $\Ghsm$. We denote by $\| \hat X_{h,*}^m\|_{W^{j,\infty}_h(\Gamma_{h,\rm f}^0)}$ the piecewise Sobolev norms on $\Gamma_{h,\rm f}^0$, i.e., 
\begin{align*}
\| \hat X_{h,*}^m\|_{W^{j,\infty}_h(\Gamma_{h,\rm f}^0)}  
:=  \max_{K_{\rm f}^0\subset \Gamma_{h,\rm f}^0} \| \hat X_{h,*}^m\|_{W^{j,\infty}(K_{\rm f}^0)} . 
\end{align*}
Since each piece $K\subset \Ghsm$ can be endowed with a canonical smooth structure, the piecewise Sobolev norms can be also defined on $\Ghsm$.

\subsection{Lifts and interpolations}
\label{section:interp-lift}

Next, we introduce some standard concepts related to (inverse) lifts and interpolations on curves (cf. \cite{Dem09,KLL19}).

The lift of a function $f$ defined on $\Ghsm$ onto the smooth curve $\Gamma^m$ is defined as 
$$f^l=f\circ (a^m |_\Ghsm)^{-1}.$$
If $f=f_h$ is a finite element function whose domain is not necessarily $\Ghsm$, we first identify $f_h$ on the interpolated curve $\hat \Gamma_{h,*}^m$ and then apply the lifting operation defined above. The inverse lift of $f\in L^2(\Gamma^m)$ onto $\Ghsm$ is defined as $f^{-l} = v\circ a^m$. 

We denote by $I_{K}$ the interpolation operator on the flat segment $K_{\rm f}^0$. Since $F_{K}=a^m\circ F_{K}$ at the nodes of $K_{\rm f}^0$, it follows that $I_{K}[ a^m\circ F_{K} ]=F_{K}$. 
The interpolation of the distance projection $a^m |_\Ghsm: \Ghsm\rightarrow\Gamma^m$ onto the curved curve $\Ghsm$ is defined as 
$$
I_h a^m:=I_{K} [a^m\circ F_{K}] \circ F_{K}^{-1} = {\rm id }
\quad\mbox{on an element}\,\, K\subset \Ghsm . 
$$
For a smooth function $f$ on the smooth curve $\Gamma^m$, we denote by $I_hf$ the interpolation of the inversely lifted function $f^{-l}=f\circ a^m$ onto $\Ghsm$, i.e.,
$$
I_h f:=I_{K} [f\circ a^m\circ F_{K}] \circ F_{K}^{-1} 
\quad\mbox{on an element}\,\, K\subset \Ghsm . 
$$
We denote by $(I_hf)^l = (I_hf)\circ (a^m |_\Ghsm)^{-1}$ the lift of $I_hf$ onto $\Gamma^m$. 
For a piecewise smooth function $f$ on $\Ghsm$ (instead of $\Gm$), we use the same notation $I_hf$ to denote the following interpolated function on $\Ghsm$:   
$$
I_h f:=I_{K} [f\circ F_{K}] \circ F_{K}^{-1} 
\quad\mbox{on an element}\,\, K\subset \Ghsm . 
$$

\subsection{Shape regularity constants and basic approximation properties}
\label{section:interpolated}

Given the discrete flow map $\hat X_{h,*}^m:\Gamma_{h,{\rm f}}^0\rightarrow\Ghsm$, we define the shape regularity constants 
\begin{align}\label{P}
	\begin{aligned}
		\kappa_l
		:=&\max_{0\le m\le l} (
		\| \hat X_{h,*}^m\|_{W^{k-1,\infty}_h(\Gamma_{h,\rm f}^0)} 
		+ 
		\| (\hat X_{h,*}^m)^{-1} \|_{W^{1,\infty}_h(\Ghsm)})  ,\\
		\kappa_{*,l} 
		:=
		&\max_{0\le m\le l} 
		\| \hat X_{h,*}^m\|_{H^k_h(\Gamma_{h,\rm f}^0)} 
		.
	\end{aligned}
\end{align}
The a priori boundedness of $\kl$ and $\ksl$ (independent of $\tau$, $h$ and $l$) shall be proved in Section \ref{sec:bbd}. 

With our notation of identifying a finite element function by its vector of nodal values, the equivalence of $W^{1,p}, p\in [1,\infty]$, norm  follows immediately (cf. \cite[Lemma 4.3]{KLL17})
$$
C_{\kappa_m} ^{-1} \|f_h\|_{W^{1,p}(\Ghsm)} \le \|f_h\|_{W^{1,p}(\Gamma_{h,\rm f}^0)} \le C_{\kappa_m} \|f_h\|_{W^{1,p}(\Ghsm)} ,
$$
for $0\leq m\leq l$.
Since $\hat X_{h,*}^m:\Gamma_{h,\rm f}^0\rightarrow \Ghsm$ is the Lagrange interpolation of $a^m\circ \hat X_{h,*}^m:\Gamma_{h,\rm f}^0\rightarrow\Gm$ on the piecewise flat curve $\Gamma_{h,\rm f}^0$, it follows that (cf. \cite[Eq. (3.3)]{BL24FOCM})
\begin{align}\label{interpol-error-a-1}
\|a^m\circ \hat X_{h,*}^m - \hat X_{h,*}^m \|_{L^2(\Gamma_{h,\rm f}^0)} 
+ h \|a^m\circ \hat X_{h,*}^m - \hat X_{h,*}^m \|_{H^{1}(\Gamma_{h,\rm f}^0)}  
&\le C_\km (1 + \ksm) h^{k+1} .
\end{align} 
Inequality \eqref{interpol-error-a-1} can be equivalently written in the following form by using $I_ha^m={\rm id}$ on $\Ghsm$ and the norm equivalence on $\Gamma_{h,\rm f}^0$ and $\Ghsm$: 
\begin{align}
\|a^m - I_ha^m \|_{L^2(\Ghsm)} 
+ h \|a^m - I_ha^m \|_{H^{1}(\Ghsm)}  
&\le C_\km (1 + \ksm) h^{k+1} ,
\end{align} 
and consequently (cf. \cite[Eqs. (3.6)--(3.7)]{BL24FOCM}),
\begin{align}\label{normal-intpl}
	\begin{split}
		\|\hat n_{h,*}^m - n_*^m\|_{L^2(\Ghsm)} 
		&\le C_\km (1 + \ksm) h^{k},
		\\
		\|\hat n_{h,*}^m - n_*^m\|_{L^\infty(\Ghsm)} 
		&\le C_\km (1 + \ksm) h^{k-1/2}  ,
	\end{split}
\end{align}
where $\hat n_{h,*}^m$ is the piecewise unit normal vector on $\Ghsm$ and $n_*^m$ is the smooth extension of $n^m$ into the neighborhood $D_\delta(\Gm)$ via the retraction map $a^m$.
Moreover, from the chain rule, the following interpolation error estimates hold for any smooth function $f$ on $\Gm$: 
\begin{align}\label{interpol-error-f}
\|f^{-l} - I_hf \|_{L^2(\Ghsm)} 
+ h\|f^{-l} - I_hf \|_{H^1(\Ghsm)} 
&\le 
C_\km (1 + \ksm)
\|f\|_{H^{k+1}(\Gm)}
h^{k+1} ,\\
\|f - (I_hf)^l \|_{L^2(\Gamma^m)} 
+ h\|f - (I_hf)^l \|_{H^1(\Gamma^m)} 
&\le 
C_\km (1 + \ksm)
\|f\|_{H^{k+1}(\Gm)}
h^{k+1} .
\end{align} 

From the norm equivalence,  the following two elementary lemmas quantifying the errors of the mass and stiffness bilinear forms are standard (cf. \cite[Lemma 4.2]{BL24FOCM}, \cite[Lemma 5.6]{Kov18} and \cite[Lemma 4.1]{KLL17}). 
\begin{lemma}\label{lemma:geo-perturb}
	The following geometric perturbation estimates hold for $f_1,f_2\in H^{1}(\Ghsm)$ and their lifts $f_1^l,f_2^l\in H^{1}(\Gm)${\rm:}
	\begin{align*}
		\Big|\int_{\Ghsm} f_1f_2 - \int_{\Gm} f_1^l f_2^l \Big| 
		&\leq 
		C_\km (1 + \ksm)
		h^{k+1} \|f_1\|_{ L^\infty(\Ghsm)}\|f_2\|_{L^2(\Ghsm)}
	\end{align*}
	and
	\begin{align*}
		&\Big|\int_{\Ghsm}\nabla_{\Ghsm} f_1\cdot\nabla_{\Ghsm} f_2 - \int_{\Gm}\nabla_{\Gm} f_1^l\cdot \nabla_{\Gm} f_2^l \Big| \notag\\
		&\qquad\qquad\qquad\qquad 
		\leq 
		C_\km (1 + \ksm) h^{k+1}\|\nabla_{\Ghsm}f_1\|_{ L^\infty(\Ghsm)}\|\nabla_{\Ghsm}f_2\|_{L^2(\Ghsm)} .
	\end{align*}
\end{lemma} 
\begin{lemma}\label{lemma:e-blinear}
	For all finite element functions $f_h,g_h\in S_h(\Ghsm)$, it holds that
	\begin{align*}
		\Big|\int_{\Ghm} f_h g_h - \int_{\Ghsm} f_h g_h \Big| 
		&\leq 
		C_\km
		\| \nabla_\Ghsm \ehm \|_{L^2(\Ghsm)} \|f_h\|_{ L^\infty(\Ghsm)}\|g_h\|_{L^2(\Ghsm)}
	\end{align*}
	and
	\begin{align*}
		&\Big|\int_{\Ghm}\nabla_{\Gm} f_h\cdot\nabla_{\Gm} g_h - \int_{\Ghsm}\nabla_{\Ghsm} f_h \cdot \nabla_{\Ghsm} g_h \Big| \notag\\
		&\qquad\qquad\qquad\qquad 
		\leq 
		C_\km \| \nabla_\Ghsm \ehm \|_{L^2(\Ghsm)} \|\nabla_{\Ghsm}f_h\|_{ L^\infty(\Ghsm)}\|\nabla_{\Ghsm}g_h\|_{L^2(\Ghsm)} .
	\end{align*}
\end{lemma} 

In the rest of this paper, we use $C$ as a generic positive constant which may be different at different occurrences, possibly dependent on $T$ and $\kappa_m$ ($m$ is the current time level and can be easily read off from the inequality), but is independent of $\tau$, $h$, and $\ksm$. 
We use the notation $A \lesssim B$ to denote the relation ``$A\le CB$ for some constant $C$''. If $A\lesssim B$ and $B\lesssim A$ at the same time, then we use the notation $A\sim B$.
Besides, we denote by $C_0$ another generic positive constant which is independent of $\kappa_m$ and $\ksm$.


\subsection{Induction hypothesis}
\label{sec:ind-hypo}

We assume that the following conditions hold for $m=0,\dots,l$ (and then prove that these conditions could be recovered for $m=l+1$): 
\begin{enumerate}
\item[(1)]
The numerically computed curve $\Gamma_h^m$ is in a $\delta$-neighborhood of the exact curve $\Gamma^m$. Therefore, the distance projection of the nodes of $\Gamma_h^m$ onto $\Gamma^m$ are well defined (thus the interpolated curve $\Ghsm$ is well defined). 

\item[(2)]
The error $\ehm=X_h^m-\hat X_{h,*}^m$ satisfies the following estimates: 
\begin{align}
	\| \ehm \|_{L^2(\Ghsm)} + h\| \ehm \|_{H^1(\Ghsm)} &\leq h^{2.5} . \label{eq:ind_hypo1} 
\end{align}

\end{enumerate}

\begin{remark}\upshape
	The exponent $2.5$ is sharp in the derivation of the last inequality in \eqref{eq:H2-tan-stab2}.
\end{remark}

Based on these induction assumptions, the following results can be obtained from \eqref{eq:ind_hypo1} by applying the inverse inequality of finite element functions: 
\begin{align}\label{Linfty-W1infty-hat-em}
\| \nabla_{\Ghsm} \ehm \|_{L^2(\Ghsm)}\leq h^{1.5},
\quad
\|  \ehm \|_{L^\infty(\Ghsm)} \lesssim h^{2}
\quad\mbox{and}\quad 
\| \nabla_\Ghsm \ehm \|_{L^\infty(\Ghsm)} \lesssim h ,
\end{align} 
which, according to \cite[Lemma 4.3]{KLL17}, guarantee the equivalence of $L^p$ and $W^{1,p}$ norms, $1\le p\le \infty$, of finite element functions with a common nodal vector on the family of curves 
$$\hat\Gamma_{h,\theta}^m=(1 - \theta)\Ghsm + \theta\Ghm ,\quad  \theta\in [0, 1] .$$
As a consequence of \eqref{Linfty-W1infty-hat-em} and Lemma \ref{lemma:n_bar_app}, we have the following boundedness of normal vectors:
\begin{align}\label{eq:n-bbd}
	\| \nhm \|_{W_h^{1,\infty}(\Ghsm)}
	+
	\| \nbhm \|_{W_h^{1,\infty}(\Ghsm)}
	+
		\| \nhsm \|_{W_h^{1,\infty}(\Ghsm)}
	+
	\| \nbhsm \|_{W_h^{1,\infty}(\Ghsm)}
		\lesssim
		1.
\end{align}

\subsection{Geometric relations}\label{section:geometry}

The projection error at the time level $m$ is defined as $\ehm := X_h^m -  \hat X_{h,*}^m$ where $\hat X_{h,*}^m$ is the finite element function whose nodal values coincide with the projected nodes of $\Gamma_h^m$ onto $\Gm$.
By definition, at the time level $m+1$ we have the following nodal relation
\begin{align}\label{eq:geo_rel_1}
	\ehM = I_h\big[ (\eM\cdot \nsM)\nsM \big] +f_h ,
\end{align}
with
\begin{align}\label{eq:geo_rel_2}
	|f_h| \lesssim |[ I-n_*^{m + 1} (n_*^{m + 1})^\top ] e_h^{m+1} |^2
	\quad\mbox{at the nodes of $\GhsM$},
\end{align}
where $f_h$ can be interpreted as a quadratic remainder of the nodal-wise orthogonal projection due to the presence of curvature.

Let $X_{h,*}^{m+1}:\Ghsm\rightarrow \Gamma_{h,*}^{m+1}$ be the unique polynomial local flow map where the nodes of $\Ghsm$ move along $u(t)\cdot n(t) n(t)$, and we denote by $X^{m+1}:\Gamma^m\rightarrow\Gamma^{m+1}$ the exact local flow map along $u(t)\cdot n(t) n(t)$.
Since $X_{h,*}^{m+1} - \hat X_{h,*}^m = X^{m+1} - {\rm id}$ at the finite element nodes on $\Gm$, it follows that 
\begin{align}
		&X_{h,*}^{m+1} - \hat X_{h,*}^m =I_h(X^{m+1} - {\rm id}) &&\mbox{on}\,\,\, \Ghsm , \label{X-id-Hn0} \\
		&X^{m+1} - {\rm id} = \tau ( {{u^m \cdot n^m n^m}}  +  g^m )  &&\mbox{on}\,\,\,  \Gamma^m, \label{X-id-Hn}
\end{align}
where $u^m \cdot n^m n^m$ is the normal part of the prescribed flow $u(t)$ without the tangential motion at time level $t=t_m$, and $g^m$ is the smooth correction from the Taylor expansion, satisfying the following estimate: 
\begin{align}\label{W1infty-g}
	\|g^m\|_{W^{1,\infty}(\Gamma^m)}\le C\tau . 
\end{align}
Therefore, we obtain 
\begin{align}\label{eq:geo_rel_3}
	\begin{aligned}
		X_h^{m+1} - X_h^m 
		&= \eM - \ehm + X_{h,*}^{m+1} - \hat X_{h,*}^m \\
		&= \eM - \ehm + \tau I_h(u^m \cdot n^m n^m + g^m) .
	\end{aligned}
\end{align}
This relation helps us convert the numerical displacement $X_h^{m+1} - X_h^m$ to the error displacement $\eM - \ehm$.

The following geometric identities have been proved in \cite[Eqs. (A.15)--(A.17)]{BL24FOCM} which will help us handle $\hat X_{h,*}^{m + 1}$ in Section \ref{sec:err-est} and \ref{sec:bbd}:
\begin{align}
	\Nsm (\hat X_{h,*}^{m + 1} -\hat X_{h,*}^{m})
	&= (X^{m + 1} - {\rm id})\circ a^m + \rho_h && \mbox{at the nodes} , \label{eq:geo_rel_4}\\
	\mbox{where}\,\,\, |\rho_h| 
	&\le C_0 \tau^2 + C_0 |T_*^m (\hat X_{h,*}^{m + 1} -\hat X_{h,*}^{m})|^2
	&& \mbox{at the nodes} , \label{eq:geo_rel_5}\\
	T_*^m (\hat X_{h,*}^{m + 1} -\hat X_{h,*}^{m})
	&= T_*^m (X_{h}^{m + 1} - X_{h}^{m})
	+ T_*^m (N_*^{m+1}-N_*^{m})  \hat e_{h}^{m + 1} && \mbox{at the nodes} , \label{eq:geo_rel_6}
\end{align}
where $C_0$ is a constant that is independent of $\km$ and $\ksm$.

\section{Stability estimates}
\label{sec:stab-est}

\subsection{Consistency error}\label{sec:cons_err}

We define the consistency error for the first equation in the transport BGN system, i.e. Eq. \eqref{eq:BGN_tr11}, at the time level $t_m$ to be the following linear functional on $S_h(\Ghsm)$: 
\begin{align}\label{eq:cons_eq}
	d^m(\chi_h) 
	&:= 
	\int_{\hat\Gamma_{h,*}^{m}}^h \frac{X_{h,*}^{m+1}-{\rm id}}{\tau} \cdot \nbhsm \, \chi_h
	- \int_{\hat\Gamma_{h,*}^{m}}^h u(t_m)|_\Ghsm \cdot \nbhsm \chi_h 
	\notag \\
	&= 
	\int_{\hat\Gamma_{h,*}^{m}}^h I_h\Big( \frac{X^{m+1} - {\rm id}}{\tau} - u(t_m)|_\Ghsm \Big) \cdot \nsm \chi_h
	\notag\\
	&\quad+
	\int_{\hat\Gamma_{h,*}^{m}}^h I_h\Big( \frac{X^{m+1} - {\rm id}}{\tau} - u(t_m)|_\Ghsm \Big) \cdot (\nbhsm - \nsm) \chi_h
	\notag\\
	&=: d_1^m(\chi_h) + d_2^m(\chi_h), \quad\forall \chi\in S_h(\Ghsm) ,
\end{align}
where the averaged normal vector $\nbhsm$ is defined in \eqref{eq:bar_n*}.
From the geometric relations \eqref{X-id-Hn}--\eqref{W1infty-g}, Lemma \ref{lemma:lump}, and the consistency estimates for normal vectors (Lemma \ref{lemma:n_bar_app}), it follows that
\begin{align*}
	| d_{1}^m(\chi_h)|  &\lesssim \tau \| \chi_h \|_{L^2(\Ghsm)} ,\\
	| d_{2}^m(\chi_h) | &\lesssim (1 + \ksm) \tau h^k \| \chi_h \|_{L^2(\Ghsm)} . 
\end{align*}
The results are summarized in the following lemma.
\begin{lemma}
	The consistency error defined in \eqref{eq:cons_eq} satisfies
	\begin{align}\label{eq:cons_err}
		| d^m(\chi_h)| &\lesssim (\tau + (1 + \ksm) \tau h^k) \| \chi_h \|_{L^2(\Ghsm)} 
		\quad\forall\,\chi_h\in S_h(\Ghsm) . 
	\end{align}
\end{lemma}

\subsection{Naive estimates for linear forms}\label{sec:J_stab}

Subtracting \eqref{eq:cons_eq} from \eqref{eq:BGN_tr11}, we get the following error equation: 
\begin{align}\label{eq:err_eq1}
	&\int_{\Gamma_{h}^{m}}^h \frac{X_h^{m+1}- X_h^m}{\tau} \cdot \nbhm \chi_h
	- \int_{\hat\Gamma_{h, *}^{m}}^h \frac{X_{h, *}^{m+1}- \hat X_{h, *}^m}{\tau}  \cdot \nbhsm \chi_h \notag\\
	&\quad- \int_{\Gamma_h^m}^h u(t_m)|_{\Ghm} \cdot \nbhm \chi_h 
	+ \int_{\hat\Gamma_{h,*}^m}^h u(t_m)|_{\Ghsm} \cdot \nbhsm\chi_h \notag\\
	&= -d^m(\chi_h) .
\end{align} 
The left-hand side of \eqref{eq:err_eq1} can be furthermore written as 
\begin{align}\label{eq:mass_diff}
	&\int_{\Gamma_{h}^{m}}^h \frac{X_h^{m+1}- X_h^m}{\tau} \cdot \nbhm \chi_h
	- \int_{\hat\Gamma_{h, *}^{m}}^h \frac{X_{h, *}^{m+1}- \hat X_{h, *}^m}{\tau}  \cdot \nbhsm \chi_h \notag\\
	&\quad- \int_{\Gamma_h^m}^h u(t_m)|_{\Ghm} \cdot \nbhm \chi_h 
	+ \int_{\hat\Gamma_{h,*}^m}^h u(t_m)|_{\Ghsm} \cdot \nbhsm\chi_h \notag\\
	&= \int_{\hat\Gamma_{h, *}^{m}}^h \frac{e_h^{m+1}-\hat e_h^m}{\tau} \cdot \nbhsm \chi_h + J^m(\chi_h) ,
\end{align} 
with
\begin{align}\label{def-Jm-phi}
	J^m(\chi_h) 
	&=  
	\int_{\Ghm}^h \frac{X_{h}^{m+1}- X_{h}^m }{\tau}  \cdot \nbhm \chi_h  
	-
	 \int_{\hat\Gamma_{h, *}^{m}}^h \frac{X_{h}^{m+1}- X_{h}^m }{\tau} \cdot \nbhsm \chi_h
	\notag\\
	&\quad- \int_{\Gamma_h^m}^h u(t_m)|_{\Ghm} \cdot \nbhm \chi_h 
	+ \int_{\hat\Gamma_{h,*}^m}^h u(t_m)|_{\Ghsm} \cdot \nbhsm\chi_h \notag\\
	&=  \int_{\Ghm}^h \frac{\delta \tilde X_h^m}{\tau}  \cdot n_h^m \chi_h  
	-  \int_{\hat\Gamma_{h,*}^m}^h \frac{\delta \tilde X_h^m}{\tau}  \cdot \hat n_{h,*}^m \chi_h  
	\notag\\
	&\quad
	-  \int_{\hat\Gamma_{h,*}^m}^h (I_h u(t_m)|_{\Ghm} - I_h u(t_m)|_{\Ghsm}) \cdot \hat n_{h,*}^m \chi_h \notag\\
	&=: J_1^m(\chi_h) + J_2^m(\chi_h),
\end{align}
where $\delta \tilde X_h^m := X_h^{m+1} - X_h^{m}  - \tau I_h u(t_m) |_\Ghm \in S_h(\Ghm)^2$ and in the second equality we have used \eqref{eq:bar_n} and \eqref{eq:bar_n*1} to convert the averaged normal vectors $\nbhm$ and $\nbhsm$ to the corresponding piecewise unit normal vectors $\nhm$ and $\hat n_{h,*}^m$.

Since the curve discrepancy and  $\nhm - \nhsm$ each contribute an error of order $\nabla_\Ghsm \ehm$ (cf. Lemma \ref{lemma:ud}, Items 6, 7), $J_1^m$ can be bounded by
\begin{align}
	| J_1^m(\chi_h) |
	\lesssim
	\| \nabla_\Ghsm \ehm \|_{L^2(\Ghsm)} \| \chi_h \|_{L^2(\Ghsm)} , \notag
\end{align}
and the Lipschitz continuity of $u(t_m)$ implies
\begin{align}
	| J_2^m(\chi_h) |
	\lesssim
	\| \ehm \|_{L^2(\Ghsm)} \| \chi_h \|_{L^2(\Ghsm)} . \notag
\end{align}
In summary, we have the following lemma.
\begin{lemma}\label{lemma:J_ele}
	The linear form $J^m$ defined in \eqref{def-Jm-phi} admits the upper bound:
	\begin{align}\label{eq:J_ele}
		| J^m(\chi_h) |
		\lesssim
		\| \ehm \|_{H^1(\Ghsm)} \| \chi_h \|_{L^2(\Ghsm)} .
	\end{align}
\end{lemma}
Substituting \eqref{eq:mass_diff} into \eqref{eq:err_eq1}, we can rewrite the error equation into the following form: 
\begin{align}\label{eq:err_eq2}
	\int_{\hat\Gamma_{h, *}^{m}}^h \frac{e_h^{m+1}-\hat e_h^m}{\tau} \cdot {{\bar n_{h, *}^m}} \chi_h + J^m(\chi_h)
	= -d^m(\chi_h) .
\end{align} 
\begin{remark}\upshape
	The upper bound in \eqref{eq:J_ele} is not stable since for transport equations we do not have $H^1$ parabolicity to control $\| \ehm \|_{H^1(\Ghsm)}$. Thanks to the intrinsic orthogonality in the projection error, a refined version of Lemma \ref{lemma:J_ele} shall be derived in Section \ref{sec:refined} (see \eqref{eq:Je}). The improved estimate for the linear form $J^m$ implies the transport $L^2$-stability of the error equation \eqref{eq:err_eq2}.
\end{remark}

\subsection{High-order tangential stability estimates}
\label{sec:tan_stab}
By calculating the first variation, the Euler-Lagrange equation of the deformation rate functional
$$
\min_{v\in H^1(\Gamma)} \int_{\Gamma}|\nabla_\Gamma v|^2 
\quad\mbox{under the pointwise constraint $v\cdot n = u\cdot n$}
$$
is the following elliptic velocity system
\begin{subequations}\label{subeq:system0}
	\begin{align}
		v\cdot n&=u\cdot n \label{eq:ell_sys_v10} , \\
		-\Delta_\Gamma v&=\lambda n  . \label{eq:ell_sys_v20}
	\end{align}
\end{subequations}
The function $\lambda$ is the Lagrange multiplier for the pointwise constraint $v\cdot n = u\cdot n$. By the direct method of calculus of variation, it can be shown that the deformation rate functional has a unique smooth minimizer.
In this subsection we are going to show the consistency between the discrete velocity generated by the transport BGN method \eqref{eq:BGN_tr11}--\eqref{eq:BGN_tr22} and the elliptic velocity system \eqref{eq:ell_sys_v10}--\eqref{eq:ell_sys_v20}.


Since \eqref{eq:ell_sys_v10} implies that $v^m= (u^m \cdot n^m) n^m + T^m v^m$, where $T^m =I-n^m(n^m)^\top$ is the tangential projection matrix on $\Gm$, the following relation follows from \eqref{eq:geo_rel_3} and the nodal relation $T^m = \Tsm$:
\begin{align}\label{eq:geo_rel_v}
	X_h^{m+1} - X_h^{m}  - \tau I_h v^m 
	&= X_h^{m+1} - X_h^{m}  - \tau I_h (u^m \cdot n^m n^m)  - \tau I_h \Tsm v^m \notag\\
	&= \eM - \ehm  - \tau I_h \Tsm v^m + \tau I_h g^m \quad\mbox{on}\,\,\,\Ghsm. 
\end{align}

The following relation can be obtained by subtracting integral $\tau \int_{\Gamma_h^m} \nabla_{\Gamma_h^m} I_h v^m \cdot  \nabla_{\Gamma_h^m}  \phi_h $ from the both sides of the numerical scheme in \eqref{eq:BGN_tr22}: 
\begin{align}\label{eq:stiff-iden}
	&\int_{\Gamma_h^m} \nabla_{\Gamma_h^m} (X_h^{m+1} - X_h^{m} - \tau I_h v^m) \cdot  \nabla_{\Gamma_h^m}  \phi_h \notag\\
	&= 
	\int_{\Gamma_h^m}^h \eta_h^{m+1} \bar n_h^m\cdot \phi_h
	- \int_{\Gamma_h^m} \nabla_{\Gamma_h^m} X_h^{m} \cdot  \nabla_{\Gamma_h^m} I_h [(\phi_h\cdot\bar n_h^m) \bar n_h^m ]
	\notag\\
	&\quad- \tau \int_{\Gamma_h^m} \nabla_{\Gamma_h^m} I_h v^m \cdot  \nabla_{\Gamma_h^m}  \phi_h =: \sum_{i=1}^3 L_i(\phi_h) .
\end{align}
If the test function is specifically chosen to be an almost tangential function of the form $I_h\bar T_h^m\phi_h$, then $L_1(I_h\bar T_h^m\phi_h) = L_2(I_h\bar T_h^m\phi_h) =0$ due to the nodal-wise orthogonality. Analogous to \cite[Eqs. (4.42)--(4.45)]{BL24}, $L_3(I_h\bar T_h^m\phi_h)$ can be bounded by using super-approximation results:
\begin{align}\label{eq:L}
	| L_3(I_h\bar T_h^m\phi_h) |
	&\lesssim
	\tau ((1 + \ksm) h^{k+1} +\| \nabla_\Ghsm \ehm \|_{L^2(\Ghsm)} ) \| I_h\bar T_h^m\phi_h \|_{H^1(\Ghsm)}  .
\end{align}

Utilizing the orthogonality between $\Nbhm$ and $\Tbhm$, the following $H^1$ tangential stability estimate was shown in \cite[Eq. (4.59)]{BL24}.
\begin{lemma}\label{lemma:tan_stab}
We have the $H^1$ tangential stability estimate
\begin{align}\label{eq:tan_stab}
	&
	\| X_h^{m+1} - X_h^{m}  - \tau I_h v^m \|_{L^2(\Ghsm)}
	+
	\| \nabla_\Ghsm I_h \Tbhm (X_h^{m+1} - X_h^{m}  - \tau I_h v^m) \|_{L^2(\Ghsm)}  \notag\\
	&\lesssim  (1 + \ksm) \tau h^{k + 1} + \tau \| \nabla_\Ghsm \ehm \|_{L^2(\Ghsm)}  \notag\\
	&\quad+ (1 + h^{-2}\| \nabla_\Ghsm \ehm \|_{L^2(\Ghsm)}) \| I_h \Nbhm (X_h^{m+1} - X_h^{m}  - \tau I_h v^m) \|_{L^2(\Ghsm)} .
\end{align}
	As a consequence of \eqref{eq:geo_rel_v} and \eqref{W1infty-g}, we also have
	\begin{align}\label{eq:tan_stab_e}
		&
		\| \eM - \ehm  - \tau I_h\Tsm v^m \|_{L^2(\Ghsm)}
		+
		\| \nabla_\Ghsm I_h \Tbhm (\eM - \ehm  - \tau I_h \Tsm v^m) \|_{L^2(\Ghsm)}  \notag\\
		&\lesssim \tau (\tau + (1 + \ksm) h^{k+1}) + \tau \| \nabla_\Ghsm \ehm \|_{L^2(\Ghsm)}  \notag\\
		&\quad+(1 + h^{-2}\| \nabla_\Ghsm \ehm \|_{L^2(\Ghsm)})\| I_h \Nbhm (\eM - \ehm  - \tau I_h \Tsm v^m) \|_{L^2(\Ghsm)} .
	\end{align}
\end{lemma}
Lemma \ref{lemma:tan_stab} basically indicates that the $H^1$ norm of the discrete tangential velocity can be controlled by the $L^2$ norm of the discrete normal velocity. However, the $H^1$ tangential stability estimate is not sufficient to conclude the convergence for transport equations:
Notice that the factor $h^{-2}\| \nabla_\Ghsm \ehm \|_{L^2(\Ghsm)}$ on the right-hand side of \eqref{eq:tan_stab}--\eqref{eq:tan_stab_e} is critical for the finite element degree $k=3$ if we are aiming to show $O(h^k)$ convergence rate.
For $O(h^k)$ convergence, we can at best assume $h^{-2}\| \nabla_\Ghsm \ehm \|_{L^2(\Ghsm)} \lesssim h^{k-3-\epsilon}$, for some positive exponent $\epsilon$, in the induction hypothesis (Section \ref{sec:ind-hypo}).
This is a blow-up factor for $k=3$, and it will prevent us from concluding the a priori high-order estimates for the shape regularity.

To eliminate the influence of the critical term $h^{-2}\| \nabla_\Ghsm \ehm \|_{L^2(\Ghsm)}$, our remedy is to derive an $H^2$-version of Lemma \ref{lemma:tan_stab} where we are able to get rid of this unstable factor.
The proof of the high-order tangential stability estimate (Lemma \ref{lemma:tan_stab_H2}) relies on the following intrinsic $H^2$ stability of discrete Laplacian.
\begin{lemma}\label{lemma:H2norm}
	For any function $f_h\in S_h(\Ghm)$, its $H_h^2$ norm (piecewise $H^2$ norm) can be bounded by the discrete $H^2$ norm associated to the discrete Laplacian $\Delta_{\Ghm, h}$, i.e.
	\begin{align}
		\| f_h \|_{H_h^2(\Ghm)}
		\lesssim
		\| f_h \|_{H^1(\Ghm)}
		+
		\| \Delta_{\Ghm, h} f_h \|_{L^2(\Ghm)} ,
	\end{align}
	where $\Delta_{\Ghm, h} f_h\in S_h(\Ghm)$ is defined as the unique finite element function such that 
	\begin{align}
		\int_\Ghm \Delta_{\Ghm, h} f_h g_h 
		=
		-\int_\Ghm \nabla_\Ghm f_h \nabla_\Ghm g_h \notag
	\end{align}
	for any $g_h\in S_h(\Ghm)$.
\end{lemma}
\begin{proof}
	Let $\Delta_{\Gm, h} f_h^l \in S_h(\Gm)$ be the unique finite element function such that 
	\begin{align}
		\int_\Gm \Delta_{\Gm, h} f_h^l g_h^l 
		=
		-\int_\Gm \nabla_\Gm f_h^l \nabla_\Gm g_h^l \notag
	\end{align}
	for any $g_h^l \in S_h(\Gm)$.
	On the smooth curve $\Gm$, we define an auxiliary function $f$ to be the solution to
	\begin{align}\label{eq:aux}
		-\Delta_\Gm f + f = -\Delta_{\Gm, h} f_h^l + f_h^l .
	\end{align}
	By construction, $f_h^l$ is the Ritz-type projection of $f$ and thus satisfies the standard elliptic error estimate
	\begin{align}\label{eq:Ritz-H1}
		\| f_h^l -f \|_{L^2(\Gm)}
		+
		h\| f_h^l -f \|_{H^1(\Gm)}
		\lesssim
		h^2
		\| f \|_{H^2(\Gm)}
		.
	\end{align}
	Taking $L^2$ norm on both sides of \eqref{eq:aux} and using $L^2$ elliptic regularity theory and \eqref{eq:Ritz-H1}, we obtain
	\begin{align}
		\| f \|_{H^2(\Gm)} 
		&\lesssim
		 \| f \|_{H^1(\Gm)} 
		 + 
		 \| \Delta_\Gm f \|_{L^2(\Gm)}
		 \notag\\
		 &\lesssim
		 \| f \|_{H^1(\Gm)} 
		 +
		 \| f_h^l -f \|_{L^2(\Gm)}
		 + 
		 \| \Delta_{\Gm, h} f_h^l \|_{L^2(\Gm)}
		 \notag\\
		 &\lesssim
		 \| f \|_{H^1(\Gm)} 
		 +
		 h^2
		 \| f \|_{H^2(\Gm)}
		 + 
		 \| \Delta_{\Gm, h} f_h^l \|_{L^2(\Gm)}
		 \notag .
	\end{align}
	Upon absorbing $h^2
	\| f \|_{H^2(\Gm)}$ into the left-hand side,
	\begin{align}\label{eq:u-H2}
		\| f \|_{H^2(\Gm)} 
		\lesssim
		\| f \|_{H^1(\Gm)} 
		+ 
		\| \Delta_{\Gm, h} f_h^l \|_{L^2(\Gm)} .
	\end{align}
	Then the piecewise $H^2$ norm of $f_h$ can be bounded as follows:
	\begin{align}
			&\| f_h \|_{H_h^2(\Ghm)}
			\sim
			\| f_h^l \|_{H_h^2(\Gm)}
			\quad\mbox{(norm equivalence)}
			\notag\\
			&\lesssim h^{-1} \| f_h^l -( I_h f)^l \|_{H^1(\Gm)}
			+
			\| (I_h f)^l \|_{H_h^2(\Gm)}
			\quad\mbox{(inverse inequality)}
			\notag\\
			&\lesssim
			\| f \|_{H^2(\Gm)} 
			\quad\mbox{(\eqref{eq:Ritz-H1} and stability of $I_h$)}
			\notag\\
			&\lesssim
			\| f \|_{H^1(\Gm)} + \| \Delta_{\Gm, h} f_h^l \|_{L^2(\Gm)}  \quad\mbox{(	\eqref{eq:u-H2} is used)} . \notag
	\end{align}
	Moreover, using bilinear estimates (cf. Lemma \ref{lemma:geo-perturb} and \ref{lemma:e-blinear}),
	\begin{align}
		&\| \Delta_{\Ghm, h} f_h -  (\Delta_{\Gm, h} f_h^l)^{-l} \|_{L^2(\Ghm)} \notag\\
		&=
		\sup_{\|\phi_h\|_{L^2(\Ghm)}=1}
		\bigg| \int_\Ghm ( \Delta_{\Ghm, h} f_h -  (\Delta_{\Gm, h} f_h^l)^{-l})\cdot \phi_h \bigg|
		\notag\\
		&\leq
		\sup_{\|\phi_h\|_{L^2(\Ghm)}=1}
		\bigg| 
		\int_\Ghm \nabla_{\Ghm} f_h \cdot \nabla_{\Ghm} \phi_h 
		-
		\int_\Gm \nabla_{\Gm} f_h^l \cdot \nabla_{\Gm} \phi_h^l
		\bigg|
		\notag\\
		&\quad+
		\sup_{\|\phi_h\|_{L^2(\Ghm)}=1}
		\bigg| 
		\int_\Ghm (\Delta_{\Gm, h} f_h^l)^{-l} \cdot \phi_h
		-
		\int_\Gm  \Delta_{\Gm, h} f_h^l \cdot \phi_h^l
		\bigg|
		\notag\\
		&\lesssim 
		\sup_{\|\phi_h\|_{L^2(\Ghm)}=1}
		h^{-\frac 12}\big((1 + \ksm)h^{k+1} + \| \nabla_\Ghsm \ehm \|_{L^2(\Ghsm)}\big)
		\notag\\
		&\quad\times
		\big(\| \nabla_\Ghsm f_h \|_{L^2(\Ghsm)} + \| (\Delta_{\Gm, h} f_h^l)^{-l} \|_{L^2(\Ghsm)}\big) \| \nabla_\Ghsm \phi_h \|_{L^2(\Ghsm)} \notag\\
		&\lesssim 
		((1 + \ksm)h^{k-\frac12} + h^{-\frac 32} \| \nabla_\Ghsm \ehm \|_{L^2(\Ghsm)}) \| \nabla_\Ghsm f_h \|_{L^2(\Ghsm)}
		\notag\\
		&\quad+ 
		((1 + \ksm)h^{k+\frac12} + h^{-\frac 12} \| \nabla_\Ghsm \ehm \|_{L^2(\Ghsm)}) \| (\Delta_{\Gm, h} f_h^l)^{-l} \|_{L^2(\Ghsm)}
		\notag\\
		&\lesssim 
		\| \nabla_\Ghsm f_h \|_{L^2(\Ghsm)} 
		+
		((1 + \ksm)h^{k+\frac12} + h) \| (\Delta_{\Gm, h} f_h^l)^{-l} \|_{L^2(\Ghsm)}
		, \notag
	\end{align}
	where in the last line, we have used the induction hypothesis \eqref{eq:ind_hypo1}.
	
	We complete the proof by combining the above two estimates and using the triangle inequality.
\end{proof}
The following orthogonality lemma will help us get the crucial $H^2$ tangential stability estimate.
\begin{lemma}\label{lemma:NT_stab_ref}
	For any $f_h,g_h \in S_h(\Ghm)^2$, we have
	\begin{align}\label{eq:NT_stab}
		&\Big| \int_{\Gamma_h^m} \nabla_{\Gamma_h^m} I_h \Nbhm f_h \cdot  \nabla_{\Gamma_h^m} I_h \Tbhm g_h \Big|
		\lesssim 
		\min\Big\{
		h^{-1} \| f_h \|_{L^2(\Ghsm)} \| g_h \|_{L^2(\Ghsm)}
		,
		\notag\\
		&\hspace{30pt}
		\| f_h \|_{L^2(\Ghsm)} \| g_h \|_{H^1(\Ghsm)} + (1 + h^{-2} \| \nabla_\Ghsm \ehm \|_{L^2(\Ghsm)}) 
		 \| f_h \|_{L^2(\Ghsm)} \| g_h \|_{L^\infty(\Ghsm)}
		\Big\} .
	\end{align}
\end{lemma}
\begin{proof}
	See Appendix \ref{sec:appndix_tan_stab}.
\end{proof}
\begin{lemma}\label{lemma:tan_stab_H2}
	The following $H^2$ tangential stability estimate holds:
	\begin{align}\label{eq:tan_stab_H2}
		&\| \nabla_\Ghsm^2 I_h \Tbhm (X_h^{m+1} - X_h^{m}  - \tau I_h v^m) \|_{L^2(\Ghsm)}  \notag\\
		&\lesssim (1 + \ksm) \tau h^{k} + h^{-1}\tau \| \nabla_\Ghsm \ehm \|_{L^2(\Ghsm)}  \notag\\
		&\quad+ h^{-1} \| I_h \Nbhm (X_h^{m+1} - X_h^{m}  - \tau I_h v^m) \|_{L^2(\Ghsm)}  .
	\end{align}
	Moreover, using the conversion formula \eqref{eq:geo_rel_v}, we get
	\begin{align}\label{eq:tan_stab_H2_e}
		&\| \nabla_\Ghsm^2 I_h \Tbhm (\eM - \ehm  - \tau I_h \Tsm v^m) \|_{L^2(\Ghsm)}  \notag\\
		&\lesssim h^{-1} \tau (\tau + (1 + \ksm)  h^{k+1}) + h^{-1}\tau \| \nabla_\Ghsm \ehm \|_{L^2(\Ghsm)}  \notag\\
		&\quad+ h^{-1} \| I_h \Nbhm (\eM - \ehm  - \tau I_h \Tsm v^m) \|_{L^2(\Ghsm)}  .
	\end{align}
\end{lemma}
\begin{proof}
	Let $P(\Ghm):L^2(\Ghm)\rightarrow S_h(\Ghm)$ be the $L^2$ projection and $\delta X_h^m := X_h^{m+1} - X_h^{m}  - \tau I_h v^m \in S_h(\Ghm)^2$.
	Applying Lemma \ref{lemma:NT_stab_ref} with $(f_h, g_h) = (I_h \bar N_h^m P(\Ghm)\phi, I_h \bar T_h^m \delta X_h^m)$ and 
	$(f_h, g_h) = (I_h \bar N_h^m \delta X_h^m, I_h \bar T_h^m P(\Ghm)\phi)$ respectively, we obtain
	\begin{align}\label{eq:H2-tan-stab}
		&\| \Delta_{\Ghm,h} I_h \Tbhm\delta X_h^m \|_{L^2(\Ghm)}
		\notag\\
		&=
		\sup_{\| \phi \|_{L^2(\Ghm)}=1}\bigg| \int_\Ghm \Delta_{\Ghm,h} I_h \Tbhm\delta X_h^m \cdot \phi \bigg|
		\notag\\
		&=
		\sup_{\| \phi \|_{L^2(\Ghm)}=1}\bigg| \int_\Ghm \nabla_\Ghm I_h \Tbhm\delta X_h^m \cdot \nabla_\Ghm P_h(\Ghm) \phi \bigg|
		\notag\\
		&\leq
		\sup_{\| \phi \|_{L^2(\Ghm)}=1}\bigg| \int_\Ghm \nabla_\Ghm I_h \Tbhm\delta X_h^m \cdot \nabla_\Ghm I_h\Nbhm P_h(\Ghm) \phi \bigg|
		\notag\\
		&\quad+
		\sup_{\| \phi \|_{L^2(\Ghm)}=1}\bigg| \int_\Ghm \nabla_\Ghm I_h \Tbhm\delta X_h^m \cdot \nabla_\Ghm I_h\Tbhm P_h(\Ghm) \phi \bigg|
		\notag\\
		&\leq
		\sup_{\| \phi \|_{L^2(\Ghm)}=1}\bigg| \int_\Ghm \nabla_\Ghm I_h \Tbhm\delta X_h^m \cdot \nabla_\Ghm I_h\Nbhm P_h(\Ghm) \phi \bigg|
		\notag\\
		&\quad+
		\sup_{\| \phi \|_{L^2(\Ghm)}=1}\bigg| \int_\Ghm \nabla_\Ghm I_h \Nbhm\delta X_h^m \cdot \nabla_\Ghm I_h\Tbhm P_h(\Ghm) \phi \bigg|
		\notag\\
		&\quad+
		\sup_{\| \phi \|_{L^2(\Ghm)}=1}
		\bigg| \sum_{i=1}^3 L_i\bigg(I_h\Tbhm P_h(\Ghm) \phi\bigg) \bigg|
		\qquad\mbox{(\eqref{eq:stiff-iden} is used)}
		\notag\\
		&\lesssim
		\| I_h \Tbhm \delta X_h^m \|_{H^1(\Ghsm)}
		+ (1 + h^{-2} \| \nabla_\Ghsm \ehm \|_{L^2(\Ghsm)}) \| I_h \Tbhm \delta X_h^m \|_{L^\infty(\Ghsm)}
		\notag\\
		&\quad
		+
		h^{-1} \| I_h \Nbhm \delta X_h^m \|_{L^2(\Ghsm)} 
		+
		h^{-1}\tau  ((1 + \ksm) h^{k +1} + \| \nabla_\Ghsm \ehm \|_{L^2(\Ghsm)})
		 \\
		&\hspace{139pt}\mbox{(Lemma \ref{lemma:NT_stab_ref} and \eqref{eq:L} are used)} .
		\notag
	\end{align}
	The $H_h^2$ norm of the tangential motion can be bounded as follows:
	\begin{align}\label{eq:H2-tan-stab2}
		&\| I_h \Tbhm\delta X_h^m \|_{H_h^2(\Ghm)}
		\lesssim
		\| I_h \Tbhm\delta X_h^m \|_{H^1(\Ghm)}
		+
		\| \Delta_{\Ghm, h} I_h \Tbhm\delta X_h^m \|_{L^2(\Ghm)}
		\notag\\
		&\hspace{239pt}\mbox{(Lemma \ref{lemma:H2norm} is used)}
		\notag\\
		&\lesssim
		\| I_h \Tbhm \delta X_h^m \|_{H^1(\Ghsm)}
		+ (1 + h^{-2} \| \nabla_\Ghsm \ehm \|_{L^2(\Ghsm)}) \| I_h \Tbhm \delta X_h^m \|_{L^\infty(\Ghsm)}
		\notag\\
		&\quad
		+
		h^{-1} \| I_h \Nbhm \delta X_h^m \|_{L^2(\Ghsm)} 
		+
		h^{-1}\tau  ((1 + \ksm) h^{k +1} + \| \nabla_\Ghsm \ehm \|_{L^2(\Ghsm)})
		\notag\\
		&\hspace{239pt}\mbox{(\eqref{eq:H2-tan-stab} is used)}
		\notag\\
		&\lesssim
		\tau (h^{-1} + h^{-2} \| \nabla_\Ghsm \ehm \|_{L^2(\Ghsm)}) ((1 + \ksm) h^{k +1} + \| \nabla_\Ghsm \ehm \|_{L^2(\Ghsm)})
		\notag\\
		&\quad+
		(h^{-1} + h^{-2} \| \nabla_\Ghsm \ehm \|_{L^2(\Ghsm)} + h^{-4} \| \nabla_\Ghsm \ehm \|_{L^2(\Ghsm)}^2) \| I_h \Nbhm \delta X_h^m \|_{L^2(\Ghsm)}
		\notag\\
		&\hspace{239pt}\mbox{(Lemma \ref{lemma:tan_stab} is used)}
		\notag\\
		&\lesssim
		h^{-1} \tau ((1 + \ksm) h^{k +1} + \| \nabla_\Ghsm \ehm \|_{L^2(\Ghsm)})
		+
		h^{-1} \| I_h \Nbhm \delta X_h^m \|_{L^2(\Ghsm)}
		,
	\end{align}
	where in the last line, we have used the induction hypothesis \eqref{eq:ind_hypo1} to ensure the boundedness $h^{-2} \| \nabla_\Ghsm \ehm \|_{L^2(\Ghsm)} + h^{-4} \| \nabla_\Ghsm \ehm \|_{L^2(\Ghsm)}^2 \lesssim h^{-1}$. 
	The proof is complete.
\end{proof}
\begin{remark}\upshape\label{rmk:k3}
	In Lemma \ref{lemma:tan_stab_H2}, we have got rid of the annoyingly unstable factor, i.e. $h^{-2} \| \nabla_\Ghsm \ehm \|_{L^2(\Ghsm)}$, appearing in the $H^1$ tangential stability estimate (Lemma \ref{lemma:tan_stab}). This allows us to prove the stability and convergence for the critical finite element degree $k=3$.
\end{remark}

\subsection{Estimates for $\delta \ehm$}\label{sec:bbd_vel} 


Let the error displacement $\delta \ehm := \eM - \ehm - \tau I_h \Tsm v^m \in S_h(\Ghsm)^2$ with $v^m = v(t_m)$ being the solution to \eqref{eq:ell_sys_v10}--\eqref{eq:ell_sys_v20}, and we derive
\begin{align}\label{eq:err1}
	& \int_{\hat\Gamma_{h, *}^{m}}^h \frac{\delta \ehm}{\tau} \cdot \bar n_{h, *}^m\, \frac{\delta \ehm}{\tau} \cdot \bar n_{h, *}^m \notag\\
	&= - \int_{\hat\Gamma_{h, *}^{m}}^h I_h\Tsm v^m \cdot \bar n_{h, *}^m \,   \frac{\delta \ehm}{\tau} \cdot \bar n_{h, *}^m 
	 + 
	 \int_{\hat\Gamma_{h, *}^{m}}^h \frac{e_h^{m+1}-\hat e_h^m}{\tau} \cdot \bar n_{h, *}^m\, \frac{\delta \ehm}{\tau} \cdot \bar n_{h, *}^m \notag\\ 
	&= - \int_{\hat\Gamma_{h, *}^{m}}^h I_h\Tsm v^m \cdot \bar n_{h, *}^m \,   \frac{\delta \ehm}{\tau}  \cdot \bar n_{h, *}^m 
	- d^m \Big(I_h \Big(\frac{\delta \ehm}{\tau} \cdot \nbhsm \Big) \Big) 
	- J^m \Big(I_h \Big(\frac{\delta \ehm}{\tau} \cdot \nbhsm \Big) \Big) ,
\end{align}
where in the second identity we have used the error equation \eqref{eq:err_eq2} with test function $\chi_h = I_h \Big(\frac{\delta \ehm}{\tau} \cdot \nbhsm \Big)\in S_h(\Ghsm)$.
Due to the nodal-wise orthogonality, Lemma \ref{lemma:lump} and the consistency estimates for normal vector (Lemma \ref{lemma:n_bar_app}),
\begin{align}
	&\Big| \int_{\hat\Gamma_{h, *}^{m}}^h \frac{\delta \ehm}{\tau} \cdot \bar n_{h, *}^m\,  I_h\Tsm v^m \cdot \bar n_{h, *}^m \Big|
	\notag\\
	&=
	\Big| \int_{\hat\Gamma_{h, *}^{m}}^h \frac{\delta \ehm}{\tau} \cdot \bar n_{h, *}^m\,  I_h(\Tsm - \Tbhsm) v^m \cdot \bar n_{h, *}^m \Big|
	 \notag\\
	&\lesssim \| I_h\Tsm - \bar T_{h,*}^m \|_{L^2_h(\Ghsm)} \Big\| \frac{\delta \ehm}{\tau} \cdot \bar n_{h, *}^m \Big\|_{L^2_h(\Ghsm)} \notag\\
	&\lesssim {{(1 + \ksm)}}h^k \Big\| \frac{\delta \ehm}{\tau} \cdot \bar n_{h, *}^m \Big\|_{L^2_h(\Ghsm)} ,
	\notag
\end{align}
where the discrete norm $L_h^2$ is defined in Appendix \ref{sec:disc-norm} and we have used Lemma \ref{lemma:lump}.
The estimates for linear forms (Eqs. \eqref{eq:cons_err} and \eqref{eq:J_ele}) give
\begin{align}
	\Big| d^m\Big(I_h \Big(\frac{\delta \ehm}{\tau} \cdot \nbhsm \Big) \Big) \Big|
	&\lesssim (\tau + (1 + \ksm) \tau h^{k}) \Big\| I_h \Big(\frac{\delta \ehm}{\tau} \cdot \nbhsm \Big) \Big\|_{L^2(\Ghsm)}
	, \notag\\
	\Big| J^m \Big(I_h \Big(\frac{\delta \ehm}{\tau} \cdot \nbhsm \Big) \Big) \Big|
	&\lesssim \| \ehm \|_{H^1(\Ghsm)} \Big\| I_h \Big(\frac{\delta \ehm}{\tau} \cdot \nbhsm \Big) \Big\|_{L^2(\Ghsm)} . \notag
\end{align}
Plugging the above estimates into the right-hand side of \eqref{eq:err1} and using the norm equivalence in Lemma \ref{lemma:lump}, we obtain
\begin{align}
	\| \delta \ehm \cdot \bar n_{h,*}^m \|_{L_{h}^2(\Ghsm)} \lesssim \tau (\tau + (1 + \ksm) h^{k}) + \tau \| \ehm \|_{H^1(\Ghsm)} . \notag
\end{align}
Since, by definition, $\Nbhsm= \frac{\bar n^m_{h,*}}{|\bar n^m_{h,*}|} (\frac{\bar n^m_{h,*}}{|\bar n^m_{h,*}|})^\top$ and at each finite element nodes $| |\nbhsm| -1 | \lesssim h^{2k}$ (Eq. \eqref{eq:nhs_bar_len}).
Then we have the norm equivalence relation $\| I_h \Nbhsm f_h \|_{L^2(\Ghsm)} \sim \| f_h\cdot \nbhsm \|_{L_h^2(\Ghsm)}$ for all $f_h\in S_h(\Ghsm)^2$. Consequently,
\begin{align}\label{eq:vel_N1}
	\| I_h \Nbhsm \delta \ehm \|_{L^2(\Ghsm)} 
	&\lesssim \tau (\tau + (1 + \ksm) h^{k}) + \tau \| \ehm \|_{H^1(\Ghsm)}  .
\end{align}
Combining \eqref{eq:tan_stab_e} and \eqref{eq:vel_N1}, we get
\begin{align}
	&
	\| \delta\ehm \|_{L^2(\Ghsm)} 
	+
	\| \nabla_\Ghsm I_h \Tsm \delta\ehm \|_{L^2(\Ghsm)}  \notag\\
	&\leq 
		\| \delta\ehm \|_{L^2(\Ghsm)} 
	+
	\| \nabla_\Ghsm I_h \Tbhm \delta\ehm \|_{L^2(\Ghsm)} 
	+
	\| \nabla_\Ghsm I_h (\Tbhm - \Tsm) \delta\ehm \|_{L^2(\Ghsm)} 
	\notag\\
	&\lesssim 
	\tau (\tau + (1 + \ksm) h^{k+1}) + \tau \| \nabla_\Ghsm \ehm \|_{L^2(\Ghsm)}
	\notag\\
	&\quad+
	(1 + h^{-2} \| \nabla_\Ghsm \ehm \|_{L^2(\Ghsm)}) (\| I_h \Nbhsm \delta\ehm \|_{L^2(\Ghsm)} + \| I_h (\Nbhm - \Nbhsm) \delta\ehm \|_{L^2(\Ghsm)} )
	\notag\\
	&\quad+
	h^{-1}
	\| \nbhm - I_h\nsm \|_{L^\infty(\Ghsm)} 
	\|  \delta\ehm \|_{L^2(\Ghsm)} 
	\notag\\
	&\lesssim 
	\tau (\tau + (1 + \ksm) h^{k+1}) + \tau \| \nabla_\Ghsm \ehm \|_{L^2(\Ghsm)}
	\notag\\
	&\quad+
	\tau (1 + h^{-2} \| \nabla_\Ghsm \ehm \|_{L^2(\Ghsm)}) \big((\tau + (1 + \ksm) h^{k}) + \| \ehm \|_{H^1(\Ghsm)} \big)
	\notag\\
	&\quad+
	(1 + h^{-1} + h^{-2} \| \nabla_\Ghsm \ehm \|_{L^2(\Ghsm)})
	\times \notag\\
	&\qquad
	h^{-1/2}\big((1 + \ksm) h^k + \|\nabla_\Ghsm \ehm \|_{L^2(\Ghsm)}\big)
	\|  \delta\ehm \|_{L^2(\Ghsm)} 
	, \notag
\end{align}
where, in the last inequality, we have applied the following estimates, obtained by using Lemma \eqref{lemma:n_bar_app}, the inverse inequality and the induction hypothesis \eqref{eq:ind_hypo1},
	\begin{align}
		\begin{aligned}
			\| \nbhm - I_h \nsm \|_{L^\infty(\Ghsm)} &\lesssim 
			h^{-1/2}\big((1 + \ksm) h^k + \|\nabla_\Ghsm \ehm \|_{L^2(\Ghsm)}\big) , 
			\\
			\| I_h (\Nbhm - \Nbhsm) \|_{L^\infty(\Ghsm)} &\lesssim
			\| \nbhm - I_h \nsm \|_{L^\infty(\Ghsm)}
			(\| \nbhm \|_{L^\infty(\Ghsm)} 
			+
			\| \nbhsm \|_{L^\infty(\Ghsm)}
			)
			\notag\\
			&\lesssim
			h^{-1/2}\big((1 + \ksm) h^k + \|\nabla_\Ghsm \ehm \|_{L^2(\Ghsm)}\big) 
			. 
		\end{aligned}
	\end{align}
In view of the induction hypothesis \eqref{eq:ind_hypo1}, the last term on the right-hand side can be absorbed into the left-hand side. Then, we have
\begin{align}\label{eq:vel_T}
	&
	\| \eM - \ehm  - \tau I_h \Tsm v^m \|_{L^2(\Ghsm)} 
	+
	\| \nabla_\Ghsm I_h \Tsm (\eM - \ehm  - \tau I_h \Tsm v^m) \|_{L^2(\Ghsm)}  \notag\\
	&\lesssim 
	{{ \tau (1 + h^{-2} \| \nabla_\Ghsm \ehm \|_{L^2(\Ghsm)}) \big((\tau + (1 + \ksm) h^{k}) + \| \ehm \|_{H^1(\Ghsm)} \big) }}
	.
\end{align}
Similarly, from the $H^2$ tangential stability estimate \eqref{eq:tan_stab_H2_e},
\begin{align}\label{eq:vel_T_H2}
	&\| I_h \Tsm (\eM - \ehm  - \tau I_h \Tsm v^m) \|_{H_h^2(\Ghsm)}  \notag\\
	&\lesssim 
	h^{-1} \tau  (\tau + (1 + \ksm) h^{k}) + h^{-1} \tau \| \ehm \|_{H^1(\Ghsm)} .
\end{align}
The induction hypothesis \eqref{eq:ind_hypo1}, together with \eqref{eq:vel_N1} and \eqref{eq:vel_T}, implies the boundedness of $\delta\ehm$ under $W^{1,\infty}$ norm:
\begin{align}\label{eq:de-bbd}
	&\| \delta\ehm \|_{W^{1,\infty}(\Ghsm)} 
	\notag\\
	&\lesssim
	h^{-3/2} \| I_h \Nsm \delta\ehm \|_{L^{2}(\Ghsm)} 
	+
	h^{-1/2} \| I_h \Tsm \delta\ehm \|_{H^{1}(\Ghsm)} 
	\notag\\
	&\lesssim
	h^{-3/2} \tau \big((\tau + (1 + \ksm) h^{k}) + \| \ehm \|_{H^1(\Ghsm)} \big)
	\notag\\
	&\quad+
	h^{-1/2}
	\tau (1 + h^{-2} \| \nabla_\Ghsm \ehm \|_{L^2(\Ghsm)}) \big((\tau + (1 + \ksm) h^{k}) + \| \ehm \|_{H^1(\Ghsm)} \big)
	\notag\\
	&\lesssim \tau . 
\end{align}
Since
$\eM = \delta \ehm + \ehm + \tau I_h \Tsm v^m$ and $\delta \ehm$ is basically a higher order error term, we can use \eqref{eq:de-bbd} to estimate $\eM$:
\begin{align}
	\| \eM \|_{L^2(\Ghsm)} 
	&\leq \| \ehm \|_{L^2(\Ghsm)} + \|  \delta \ehm \|_{L^2(\Ghsm)}
	+
	 \| \tau I_h \Tsm v^m \|_{L^2(\Ghsm)}
	  \notag\\
	&\lesssim \tau + \| \ehm \|_{L^2(\Ghsm)} 
	,
	\label{eq:e_NT}\\
	\| \eM \|_{H^1(\Ghsm)} 
	&\leq \| \ehm \|_{H^1(\Ghsm)} + \|  \delta \ehm \|_{H^1(\Ghsm)}
	+
	\| \tau I_h \Tsm v^m \|_{H^1(\Ghsm)}
	\notag\\
	&\lesssim \tau + \| \ehm \|_{H^1(\Ghsm)}  .
	\label{eq:e_NT2}
\end{align}
%
Dealing with $\ehM$, we have
\begin{align}\label{eq:ehM-L2}
	\| \ehM \|_{L^2(\Ghsm)}
	&\lesssim
	\| \eM \|_{L^2(\Ghsm)} + \| \eM \|_{L^2(\Ghsm)}\| \eM \|_{L^\infty(\Ghsm)}
	\notag\\
	&\hspace{100pt}\mbox{(\eqref{eq:geo_rel_1}--\eqref{eq:geo_rel_2} are used)}
	\notag\\
	&\lesssim
	\tau + \| \ehm \|_{L^2(\Ghsm)} 
	\quad\mbox{(\eqref{eq:e_NT} and \eqref{eq:ind_hypo1} are used)} .
\end{align}

\subsection{Refined estimates for linear forms}\label{sec:refined}

In this subsection, we are going to derive the core $L^2$ stability of the linear form $J^m$ defined in \eqref{def-Jm-phi}.
We first define the numerical displacement $\delta \tilde X_h^m := X_h^{m+1} - X_h^{m}  - \tau I_h u(t_m) |_\Ghm \in S_h(\Ghsm)^2$.
From \eqref{W1infty-g}, \eqref{eq:geo_rel_v} and \eqref{eq:de-bbd}, we get the boundedness
\begin{align}\label{eq:dtX-bbd}
	\| \delta \tilde X_h^m \|_{W^{1,\infty}{(\Ghsm)}}
	=
	\| \delta\hat e_h^m + \tau I_h g^m + \tau I_h v^m - \tau I_h (u^m|_\Ghm)
	\|_{W^{1,\infty}{(\Ghsm)}}
	\lesssim
	\tau .
\end{align}
Consider the family of linearly interpolated intermediate curves $\hat\Gamma_{h,\theta}^{m}=(1-\theta) \hat\Gamma_{h,*}^{m}+\theta \Gamma_{h}^{m}$, parametrized by $\theta\in[0,1]$, and we denote the piecewise unit normal vector on $\hat\Gamma_{h,\theta}^{m}$ by $\hat n_{h,\theta}^m$.
The parametrized curve $\hat\Gamma_{h,\theta}^{m}$ moves with a constant velocity $\hat e_{h}^m$ as the parameter $\theta$ increases, and any finite element function $v_h$ with a fixed nodal vector independent of $\theta\in[0,1]$ has the transport property $\partial_\theta^\bullet v_{h} = 0$ on $\hat\Gamma_{h,\theta}^{m}$. 

By the fundamental theorem of calculus, the linear form $J_1^m(\chi_h)$ defined in \eqref{def-Jm-phi} admits the following decomposition:
\begin{align}
	J_1^m(\chi_h)
	=&\, \int_{\hat\Gamma_{h,\theta}^m }^h \frac{\delta \tilde X_h^m}{\tau} \cdot \hat n_{h,\theta}^m \chi_h
	\bigg|_{\theta=0}^{\theta=1} \notag\\
	=&\, \int_0^1 \frac{\d}{\d\theta}\int_{\hat\Gamma_{h,\theta}^m }^h \frac{\delta \tilde X_h^m}{\tau} \cdot \hat n_{h,\theta}^m \chi_h \d\theta \notag\\
	=&\, \int_0^1 \int_{\hat\Gamma_{h,\theta}^m }^h \frac{\delta \tilde X_h^m}{\tau} \cdot \partial_\theta^\bullet \hat n_{h,\theta}^m \chi_h \d\theta \notag\\
	&\, + \int_0^1 \int_{\hat\Gamma_{h,\theta}^m }^h \frac{\delta \tilde X_h^m}{\tau} \cdot \hat n_{h,\theta}^m \chi_h (\nabla_{\hat\Gamma_{h,\theta}^m } \cdot \hat e_{h}^m) \d\theta 
	\quad\mbox{(Lemma \ref{lemma:ud}, item 6)} . \notag\\
	=&\, -\int_0^1 \int_{\hat\Gamma_{h,\theta}^m }^h \frac{\delta \tilde X_h^m}{\tau} \cdot (\nabla_{\hat \Gamma_{h,\theta}^m } \hat e_{h}^m \cdot \hat n_{h,\theta}^m) \chi_h \d\theta \quad\mbox{(Lemma \ref{lemma:ud}, item 7)} \notag\\
	&\, + \int_0^1 \int_{\hat\Gamma_{h,\theta}^m }^h \frac{\delta \tilde X_h^m}{\tau} \cdot \hat n_{h,\theta}^m\, \chi_h (\nabla_{\hat\Gamma_{h,\theta}^m } \cdot \hat e_{h}^m) \d\theta 
	\notag\\
	=&\, 
	- \int_{\Ghsm}^h \frac{\delta \tilde X_h^m}{\tau} \cdot (\nabla_{\hat \Gamma_{h,*}^m } \hat e_{h}^m \cdot \hat n_{h,*}^m) \chi_h \notag\\
	&\, + \int_{\hat\Gamma_{h,*}^m }^h \frac{\delta \tilde X_h^m}{\tau} \cdot \hat n_{h,*}^m\, \chi_h (\nabla_{\hat\Gamma_{h,*}^m } \cdot \hat e_{h}^m)
	\notag\\
	&\,-\int_0^1 \int_{\hat\Gamma_{h,\theta}^m }^h \frac{\delta \tilde X_h^m}{\tau} \cdot (\nabla_{\hat \Gamma_{h,\theta}^m } \hat e_{h}^m \cdot \hat n_{h,\theta}^m) \chi_h \d\theta
	+
	\int_{\Ghsm}^h \frac{\delta \tilde X_h^m}{\tau} \cdot (\nabla_{\hat \Gamma_{h,*}^m } \hat e_{h}^m \cdot \hat n_{h,*}^m) \chi_h
	 \notag\\
	&\, + \int_0^1 \int_{\hat\Gamma_{h,\theta}^m }^h \frac{\delta \tilde X_h^m}{\tau} \cdot \hat n_{h,\theta}^m\, \chi_h (\nabla_{\hat\Gamma_{h,\theta}^m } \cdot \hat e_{h}^m) \d\theta 
	-
	\int_{\hat\Gamma_{h,*}^m }^h \frac{\delta \tilde X_h^m}{\tau} \cdot \hat n_{h,*}^m\, \chi_h (\nabla_{\hat\Gamma_{h,*}^m } \cdot \hat e_{h}^m)
	\notag\\
	&=: J_{11}^m(\chi_h) + J_{12}^m(\chi_h) + J_{13}^m(\chi_h)
	+
	J_{14}^m(\chi_h) .
\end{align}
Importantly, if the scalar test function is chosen to be $\chi_h = I_h(\ehm \cdot \nbhsm) \in S_h(\Ghsm)$, then from integration by parts and the orthogonality relation $\nabla_{\hat \Gamma_{h,*}^m } \hat n_{h,*}^m \cdot \hat e_{h}^m = \nabla_{\hat \Gamma_{h,*}^m } \hat n_{h,*}^m \cdot \hat T_{h,*}^m \hat e_{h}^m$, we obtain
\begin{align}\label{eq:J11-decomp}
	J_{11}^m(I_h(\ehm \cdot \nbhsm))
	&=
	 - \int_{\hat\Gamma_{h,*}^m }^h \frac{\delta \tilde X_h^m}{\tau} \cdot (\nabla_{\hat \Gamma_{h,*}^m } \hat e_{h}^m \cdot \hat n_{h,*}^m) \, \ehm \cdot \hat n_{h,*}^m 
	 \notag\\
	 &\quad
	 - \int_{\hat\Gamma_{h,*}^m }^h \frac{\delta \tilde X_h^m}{\tau} \cdot (\nabla_{\hat \Gamma_{h,*}^m } \hat e_{h}^m \cdot \hat n_{h,*}^m) \, \ehm \cdot (\nbhsm - \hat n_{h,*}^m)
	 \notag\\
	 &=
	 -\frac12 \int_{\hat\Gamma_{h,*}^m }^h \frac{\delta \tilde X_h^m}{\tau} \cdot \nabla_{\hat \Gamma_{h,*}^m } (\hat e_{h}^m \cdot \hat n_{h,*}^m)^2 
	 \notag\\
	 &\quad+
	 \int_{\hat\Gamma_{h,*}^m }^h \frac{\delta \tilde X_h^m}{\tau} \cdot (\nabla_{\hat \Gamma_{h,*}^m } \hat n_{h,*}^m \cdot \hat e_{h}^m) \, \ehm \cdot \hat n_{h,*}^m 
	 \notag\\
	 &\quad
	 - \int_{\hat\Gamma_{h,*}^m }^h \frac{\delta \tilde X_h^m}{\tau} \cdot (\nabla_{\hat \Gamma_{h,*}^m } \hat e_{h}^m \cdot \hat n_{h,*}^m) \, \ehm \cdot (\nbhsm - \hat n_{h,*}^m)
	 \notag\\
	 &=
	 -\frac12 \bigg(\int_{\hat\Gamma_{h,*}^m }^h - \int_{\hat\Gamma_{h,*}^m }\bigg) \frac{\delta \tilde X_h^m}{\tau} \cdot \nabla_{\hat \Gamma_{h,*}^m } (\hat e_{h}^m \cdot \hat n_{h,*}^m)^2
	 \notag\\
	 &\quad
	 -\frac12 \int_{\hat\Gamma_{h,*}^m } \frac{\delta \tilde X_h^m}{\tau} \cdot \nabla_{\hat \Gamma_{h,*}^m } (\hat e_{h}^m \cdot \hat n_{h,*}^m)^2 
	 \notag\\
	 &\quad+
 \int_{\hat\Gamma_{h,*}^m }^h \frac{\delta \tilde X_h^m}{\tau} \cdot (\nabla_{\hat \Gamma_{h,*}^m } \hat n_{h,*}^m \cdot \hat T_{h,*}^m \hat e_{h}^m) \, \ehm \cdot \hat n_{h,*}^m
	 \notag\\
	 &\quad
	 - \int_{\hat\Gamma_{h,*}^m }^h \frac{\delta \tilde X_h^m}{\tau} \cdot (\nabla_{\hat \Gamma_{h,*}^m } \hat e_{h}^m \cdot \hat n_{h,*}^m) \, \ehm \cdot (\nbhsm - \hat n_{h,*}^m)
	 \notag\\
	 &=
	 -\frac12 \bigg(\int_{\hat\Gamma_{h,*}^m }^h - \int_{\hat\Gamma_{h,*}^m }\bigg) \frac{\delta \tilde X_h^m}{\tau} \cdot \nabla_{\hat \Gamma_{h,*}^m } (\hat e_{h}^m \cdot \hat n_{h,*}^m)^2
	 \notag\\
	 &\quad
	 -\frac12 \int_{\hat\Gamma_{h,*}^m } \nabla_{\hat \Gamma_{h,*}^m }\cdot \bigg(\frac{\delta \tilde X_h^m}{\tau}  (\hat e_{h}^m \cdot \hat n_{h,*}^m)^2 \bigg)
	 \notag\\
	 &\quad
	 +\frac12 \int_{\hat\Gamma_{h,*}^m } \nabla_{\hat \Gamma_{h,*}^m } \cdot \frac{\delta \tilde X_h^m}{\tau} (\hat e_{h}^m \cdot \hat n_{h,*}^m)^2
	 \notag\\
	 &\quad+
	  \int_{\hat\Gamma_{h,*}^m }^h \frac{\delta \tilde X_h^m}{\tau} \cdot (\nabla_{\hat \Gamma_{h,*}^m } \hat n_{h,*}^m \cdot \hat T_{h,*}^m \hat e_{h}^m) \, \ehm \cdot \hat n_{h,*}^m 
	 \notag\\
	 &\quad
	 - \int_{\hat\Gamma_{h,*}^m }^h \frac{\delta \tilde X_h^m}{\tau} \cdot (\nabla_{\hat \Gamma_{h,*}^m } \hat e_{h}^m \cdot \hat n_{h,*}^m) \, \ehm \cdot (\nbhsm - \hat n_{h,*}^m)
	 \notag\\
	 &= 
	 \sum_{i=1}^5 J_{11i}^m(I_h(\ehm \cdot \nbhsm))
	 .
\end{align}
Using the super-approximation estimate in Lemma \ref{lemma:super_conv2}, the boundedness of normal vectors and $\delta\tilde X_h^m$ (cf. \eqref{eq:n-bbd} and \eqref{eq:dtX-bbd} respectively), and the inverse inequality,
\begin{align}
	|J_{111}^m(I_h(\ehm \cdot \nbhsm))|
	\lesssim h^{2k} \|  \nabla_{\hat \Gamma_{h,*}^m } (\hat e_{h}^m \cdot \hat n_{h,*}^m)^2 \|_{H_h^{2k}(\Ghsm)}
	\lesssim \| \ehm \|_{L^2(\Ghsm)}^2 . \notag
\end{align}
From the divergence theorem and a standard geometric perturbation estimate (cf. \cite{Kov18}), we get
\begin{align}
	&|J_{112}^m(I_h(\ehm \cdot \nbhsm))| 
	\notag\\
	&\leq
	\frac12 \bigg| \int_{\Gm} \nabla_{\Gm}\cdot \bigg(\frac{\delta \tilde X_h^m}{\tau}  (\hat e_{h}^m \cdot \hat n_{h,*}^m)^2 \bigg)^l \bigg|
	\notag\\
	&\quad+
	\frac12 \bigg|
	 \int_{\hat\Gamma_{h,*}^m } \nabla_{\hat \Gamma_{h,*}^m }\cdot \bigg(\frac{\delta \tilde X_h^m}{\tau}  (\hat e_{h}^m \cdot \hat n_{h,*}^m)^2 \bigg) 
	 -
	 \int_{\Gm} \nabla_{\Gm}\cdot \bigg(\frac{\delta \tilde X_h^m}{\tau}  (\hat e_{h}^m \cdot \hat n_{h,*}^m)^2 \bigg)^l
	 \bigg|
	\notag\\
	&\lesssim
	(1 + \ksm) h^{k+1/2} \| \ehm \|_{L^2(\Ghsm)}^2 . \notag
\end{align}
The estimate for $J_{13}^m$ follows directly from H\"older's inequality
\begin{align}
	|J_{113}^m(I_h(\ehm \cdot \nbhsm))| \lesssim \| \ehm \|_{L^2(\Ghsm)}^2 . \notag
\end{align}
By orthogonality and \eqref{normal-intpl},
\begin{align}
	|J_{114}^m(I_h(\ehm \cdot \nbhsm))| 
	&\lesssim \| \hat T_{h,*}^m \Nsm \|_{L^\infty(\hat \Gamma_{h,*}^m)} \| \ehm \|_{L^2(\Ghsm)}^2 \notag\\
	&= \| \hat T_{h,*}^m (\hat N_{h,*}^m - \Nsm) \|_{L^\infty(\hat \Gamma_{h,*}^m)} \| \ehm \|_{L^2(\Ghsm)}^2 \notag\\
	&\lesssim \| \hat n_{h,*}^m - \nsm \|_{L^\infty(\hat \Gamma_{h,*}^m)} \| \ehm \|_{L^2(\Ghsm)}^2 \notag\\
	&\lesssim (1 + \ksm) h^{k-1/2} \| \ehm \|_{L^2(\Ghsm)}^2 . \notag
\end{align}
Using the inverse inequality and Lemma \ref{lemma:n_bar_app}, we get
\begin{align}
	|J_{115}^m(I_h(\ehm \cdot \nbhsm))| 
	&\lesssim h^{-1/2} \| \hat n_{h,*}^m - \nbhsm \|_{L^2(\hat \Gamma_{h,*}^m)} \| \nabla_\Ghsm \ehm \|_{L^2(\Ghsm)} \| \ehm \|_{L^2(\Ghsm)} \notag\\
	&\lesssim (1 + \ksm) h^{k-3/2} \| \ehm \|_{L^2(\Ghsm)}^2 . \notag
\end{align}
In summary, $J_{11}^m$ admits the following $L^2$-stable upper bound
\begin{align}\label{eq:J11}
	|J_{11}^m(I_h(\ehm \cdot \nbhsm))| \lesssim \| \ehm \|_{L^2(\Ghsm)}^2 .
\end{align}

Then we decompose $J_{12}^m$ in the following way:
\begin{align}
	J_{12}^m(\chi_h)
	&= \int_{\hat\Gamma_{h,*}^m }^h \frac{\delta \tilde X_h^m}{\tau} \cdot \bar n_{h,*}^m \chi_h  \nabla_{\hat\Gamma_{h,*}^m } \cdot (\hat T_{h,*}^m \hat e_{h}^m)
	\notag\\
	&\quad+\int_{\hat\Gamma_{h,*}^m }^h \frac{\delta \tilde X_h^m}{\tau} \cdot \bar n_{h,*}^m \chi_h  \nabla_{\hat\Gamma_{h,*}^m } \cdot (\hat N_{h,*}^m \hat e_{h}^m)
	\notag\\
	&= \int_{\hat\Gamma_{h,*}^m }^h \frac{\delta \tilde X_h^m}{\tau} \cdot \bar n_{h,*}^m \chi_h  \nabla_{\hat\Gamma_{h,*}^m } \cdot ((\hat T_{h,*}^m - \Tsm) \hat e_{h}^m)
	\notag\\
	&\quad+ \int_{\hat\Gamma_{h,*}^m }^h \frac{\delta \tilde X_h^m}{\tau} \cdot \bar n_{h,*}^m \chi_h  \hat e_{h}^m \cdot (\nabla_{\hat\Gamma_{h,*}^m } \cdot \hat N_{h,*}^m ) 
	\notag\\
	&\quad+ \int_{\hat\Gamma_{h,*}^m }^h \frac{\delta \tilde X_h^m}{\tau} \cdot \bar n_{h,*}^m \chi_h (\hat N_{h,*}^m : \nabla_{\hat\Gamma_{h,*}^m } \hat e_{h}^m)
	\notag\\
	&=: J_{121}^m(\chi_h) + J_{122}^m(\chi_h) + J_{123}^m(\chi_h) , \notag
\end{align}
where the colon ``$:$" denotes the contraction of two matrices into a scalar.
From the inverse inequality and \eqref{normal-intpl}, we know
\begin{align}
	| J_{121}^m(\chi_h) | 
	&\lesssim 
	h^{-1}
	\| \nhsm - \nsm \|_{L^\infty(\Ghsm)}
	\| \ehm \|_{L^2(\Ghsm)} \| \chi_h \|_{L^2(\Ghsm)}
	\notag\\
	&\lesssim
	(1 + \ksm) h^{k-3/2}
	\| \ehm \|_{L^2(\Ghsm)} \| \chi_h \|_{L^2(\Ghsm)} \notag .
\end{align}
H\"older's inequality gives the bound
\begin{align}
	| J_{122}^m(\chi_h) | \lesssim \| \ehm \|_{L^2(\Ghsm)} \| \chi_h \|_{L^2(\Ghsm)} \notag .
\end{align}
Due to the nodal-wise orthogonality, $J_{123}^m(\chi_h) = 0$ for all $\chi_h \in S_h$. Collecting the estimates above, we obtain the $L^2$-stable upper bound for $J_{12}^m$:
\begin{align}\label{eq:J12}
	|J_{12}^m(\chi_h)| \lesssim \| \ehm \|_{L^2(\Ghsm)}  \| \chi_h \|_{L^2(\Ghsm)} .
\end{align}

We apply the fundamental theorem of calculus and Lemma \ref{lemma:ud} to $J_{13}^m$
\begin{align}
	J_{13}^m(\chi_h)
	&=
	-\int_0^1 \int_0^\theta \frac{\d}{\d \alpha} \int_{\hat\Gamma_{h,\alpha}^m }^h \frac{\delta \tilde X_h^m}{\tau} \cdot (\nabla_{\hat \Gamma_{h,\alpha}^m } \hat e_{h}^m \cdot \hat n_{h,\alpha}^m) \chi_h \d\alpha \d\theta
	\notag\\
	&=
	-\int_0^1 \int_0^\theta \int_{\hat\Gamma_{h,\alpha}^m }^h \frac{\delta \tilde X_h^m}{\tau} \cdot (\partial_\alpha^\bullet\nabla_{\hat \Gamma_{h,\alpha}^m } \hat e_{h}^m \cdot \hat n_{h,\alpha}^m) \chi_h \d\alpha \d\theta
	\notag\\
	&\quad
	-\int_0^1 \int_0^\theta \int_{\hat\Gamma_{h,\alpha}^m }^h \frac{\delta \tilde X_h^m}{\tau} \cdot (\nabla_{\hat \Gamma_{h,\alpha}^m } \hat e_{h}^m \cdot \partial_\alpha^\bullet \hat n_{h,\alpha}^m) \chi_h \d\alpha \d\theta
	\notag\\
	&\quad
	-\int_0^1 \int_0^\theta \int_{\hat\Gamma_{h,\alpha}^m }^h \frac{\delta \tilde X_h^m}{\tau} \cdot (\nabla_{\hat \Gamma_{h,\alpha}^m } \hat e_{h}^m \cdot \hat n_{h,\alpha}^m) (\theta \nabla_{\hat \Gamma_{h,\alpha}^m } \cdot \hat e_{h}^m ) \chi_h \d\alpha \d\theta
	\notag\\
	&=
	2 \int_0^1 \int_0^\theta \theta \int_{\hat\Gamma_{h,\alpha}^m }^h \frac{\delta \tilde X_h^m}{\tau} \cdot (\nabla_{\hat \Gamma_{h,\alpha}^m } \hat e_{h}^m \nabla_{\hat \Gamma_{h,\alpha}^m } \hat e_{h}^m \, \hat n_{h,\alpha}^m) \chi_h \d\alpha \d\theta
	\notag\\
	&\quad
	-\int_0^1 \int_0^\theta \theta \int_{\hat\Gamma_{h,\alpha}^m }^h \frac{\delta \tilde X_h^m}{\tau} \cdot \hat n_{h,\alpha}^m \big| \nabla_{\hat \Gamma_{h,\alpha}^m } \hat e_{h}^m \cdot \hat n_{h,\alpha}^m \big|^2 \chi_h \d\alpha \d\theta
	\notag\\
	&\quad
	-\int_0^1 \int_0^\theta \theta \int_{\hat\Gamma_{h,\alpha}^m }^h \frac{\delta \tilde X_h^m}{\tau} \cdot (\nabla_{\hat \Gamma_{h,\alpha}^m } \hat e_{h}^m \cdot \hat n_{h,\alpha}^m) (\nabla_{\hat \Gamma_{h,\alpha}^m } \cdot \hat e_{h}^m ) \chi_h \d\alpha \d\theta
	\notag ,
\end{align}
and consequently from H\"older's inequality and the norm equivalence 
\begin{align}\label{eq:J13}
	| J_{13}^m(\chi_h) |
	\lesssim
	\| \nabla_\Ghsm \ehm \|_{L^\infty(\Ghsm)} \| \nabla_\Ghsm \ehm \|_{L^2(\Ghsm)} \| \chi_h \|_{L^2(\Ghsm)} .
\end{align}
The term $J_{14}$ can be bounded with a very similar fashion whose proof is therefore omitted
\begin{align}\label{eq:J14}
	| J_{14}^m(\chi_h) |
	\lesssim
	\| \nabla_\Ghsm \ehm \|_{L^\infty(\Ghsm)} \| \nabla_\Ghsm \ehm \|_{L^2(\Ghsm)} \| \chi_h \|_{L^2(\Ghsm)} .
\end{align}

The estimate for $J_1^m$ follows from \eqref{eq:J11}--\eqref{eq:J14}
\begin{align}
	|J_{1}^m(I_h(\ehm \cdot \nbhsm))| \lesssim \| \ehm \|_{L^2(\Ghsm)}^2 . \notag
\end{align}
The Lipschitz continuity of $u$ implies
\begin{align}
	|J_2^m(\chi_h)| &\lesssim \| \ehm \|_{L^2(\Ghsm)} \| \chi_h \|_{L^2(\Ghsm)} \notag .
\end{align}
Then we conclude
\begin{align}\label{eq:Je}
	|J^m(I_h(\ehm \cdot \nbhsm)))| \lesssim \| \ehm \|_{L^2(\Ghsm)}^2 .
\end{align}
For a more general test function $\chi_h$, there is no way to eliminate the gradients in $J_{11}^m$ and $J_{12}^m$. A similar integration-by-parts argument as in \eqref{eq:J11-decomp} will give
\begin{align}\label{eq:Jphi}
	|J^m(\chi_h)| \lesssim \min\bigg\{\| \ehm \|_{H^1(\Ghsm)} \| \chi_h \|_{L^2(\Ghsm)}, \| \ehm \|_{L^2(\Ghsm)} \| \chi_h \|_{H^1(\Ghsm)}\bigg\} .
\end{align}
As a result,
\begin{align}\label{eq:JeM}
	&|J^m(I_h (\eM \cdot\nbhsm))|
	\notag\\ 
	&\leq 	|J^m(I_h(\ehm \cdot\nbhsm))| 
	+ 
	|J^m(I_h(\tau I_h\Tsm v^m \cdot\nbhsm))| 
	+
	|J^m(I_h(\delta\ehm \cdot\nbhsm))| 
	\notag\\
	&\lesssim \| \ehm \|_{L^2(\Ghsm)}^2
	+
	\tau \| \ehm \|_{L^2(\Ghsm)}
	+
	\| \ehm \|_{L^2(\Ghsm)} \| I_h(\delta\ehm \cdot\nbhsm) \|_{H^1(\Ghsm)}
	\notag\\
	&\lesssim \| \ehm \|_{L^2(\Ghsm)}^2
	+
	\tau \| \ehm \|_{L^2(\Ghsm)}
	,
\end{align}
where we have used \eqref{eq:de-bbd} in the last line.
\begin{remark}\upshape
	We notice from \eqref{eq:Jphi} and \eqref{eq:JeM} that $J^m$ has an improved bound if the test function has some certain orthogonality property. This is a key observation in the convergence proof since now we have very good control over the right-hand side of \eqref{eq:JeM} by $L_t^\infty L_x^2$ norm solely, compensating the absence of $H^1$-parabolicity.
\end{remark}

\section{Proof of Theorem \ref{thm:main}}
\label{Proof-THM-1}

\subsection{$L_t^\infty L_x^2$ error estimates}
\label{sec:err-est}

Up to now, we have confined our discussions on the curves at $t_m$. But in order to apply Gr\"onwall's inequality to get the main error estimate, we need to work with $\GhsM$. To this end, it is desirable to quantify the consistency between $\hat \Gamma_{h,*}^{m+1}$ and $\Ghsm$.
In fact, $\GhsM$ can be identified as $\Ghsm$ perturbed by an infinitesimal displacement of order $O(\tau)$ under some suitable norms. This $O(\tau)$ displacement can be quantified by the discrete velocity estimates developed in Section \ref{sec:bbd_vel}.

First we use the geometric relation \eqref{X-id-Hn0} to get
\begin{align}\label{eq:hat_X_s_diff3}
	\|  X_{h,*}^{m+1} - \hat X_{h,*}^{m} \|_{W^{1,\infty}(\Ghsm)} \lesssim \tau . 
\end{align}
Form \eqref{eq:geo_rel_v} and \eqref{eq:de-bbd}, we also have
\begin{align}\label{eq:diff-Xhm}
	&\|  X_{h}^{m+1} - X_{h}^{m} \|_{W^{1,\infty}(\Ghsm)}
	\notag\\ 
	&\leq 	\|  \delta\ehm \|_{W^{1,\infty}(\Ghsm)} 
	+
	\|  \tau I_h v^m \|_{W^{1,\infty}(\Ghsm)}
	+
	\|  \tau I_h g^m \|_{W^{1,\infty}(\Ghsm)}
	\notag\\
	&\lesssim \tau . 
\end{align}
According to the geometric relations \eqref{eq:geo_rel_4}--\eqref{eq:geo_rel_6}, 
\begin{align}\label{eq:hat_X_s_diff}
	\| \hat X_{h,*}^{m+1} - \hat X_{h,*}^{m} \|_{L^{\infty}(\Ghsm)}
	\lesssim
	\tau .
\end{align}
The $W^{1,\infty}$ estimate 
\begin{align}\label{eq:hat_X_s_diff1}
	\| \hat X_{h,*}^{m+1} - \hat X_{h,*}^{m} \|_{W^{1,\infty}(\Ghsm)}
	\lesssim \tau 
\end{align}
follows from a similar argument to \cite[Eq. (4.92)]{BL24}.

According to \cite[Lemma 4.3]{KLL17}, the displacement estimates above and the induction hypothesis \eqref{eq:ind_hypo1} together imply that the $L^p$ and $W^{1,p}$ norms of finite element functions on $\Ghm, \GhM, \Ghsm, \GhsM$ and $\Gamma_{h,*}^{m+1}$ with a common nodal vector are all equivalent.

The displacement estimates for positions also imply that all different kinds of normal vectors are evolving uniformly in time.
The Lipschitz continuity and  \eqref{eq:hat_X_s_diff} imply
\begin{align}\label{nsM-nsm-nodes1}
	|\nsM-\nsm| 
	\lesssim |\hat X_{h,*}^{m+1}-\hat X_{h,*}^{m}|
	+ \tau \lesssim \tau \quad\mbox{at the nodes} ,
\end{align}
and from \eqref{eq:hat_X_s_diff1}, Lemma \ref{lemma:ud} (Item 7), the definition of averaged normal vectors (Eq. \eqref{bar-n-hat-n}) and the norm equivalence, we also have
\begin{align}\label{eq:n-diff}
	&\| \hat n_{h,*}^{m+1} - \hat n_{h,*}^{m} \|_{L^\infty(\Ghsm)}
	+
	\| \bar n_{h,*}^{m+1} - \bar n_{h,*}^{m} \|_{L^\infty(\Ghsm)} \notag\\
	&\lesssim \| \nabla_\Ghsm (\hat X_{h,*}^{m+1} - \hat X_{h,*}^{m}) \|_{L^\infty(\Ghsm)} 
	\lesssim \tau .
\end{align}

Additionally, we need the following lemma to perform a stable conversion between $\| \eM\cdot \nbhsm \|_{L^2_h(\Ghsm)}^2$ and $\| \ehM\cdot \nbhsM \|_{L^2_h(\hat\Gamma_{h,*}^{m+1})}^2$. The proof is standard and can be found in Appendix \ref{sec:e-convert}.
\begin{lemma}\label{lemma:e-convert}
	We have the following estimate:
	\begin{align}
		&\| \ehM\cdot \nbhsM \|_{L^2_h(\hat\Gamma_{h,*}^{m+1})}^2 - \| \eM\cdot \nbhsm \|_{L^2_h(\Ghsm)}^2 \notag\\
		&\lesssim
		\tau(\tau + (1 + \ksm) h^{k})^2 + \tau \big(\| \ehM \|_{L^2(\Ghsm)}^2 + \| \eM \|_{L^2(\Ghsm)}^2 + \| \ehm \|_{L^2(\hat\Gamma_{h, *}^{m})}^2  \big) .
		\notag
	\end{align}
\end{lemma}
\begin{proof}
	See Appendix \ref{sec:e-convert}.
\end{proof}

Now, we are in a good position to derive the main error estimate. Testing the error equation \eqref{eq:err_eq2} by $\chi_h = I_h(\eM \cdot \nbhsm)$ and using Young's inequality, we obtain
\begin{align}\label{e-est-mass}
	&\frac{1}{2\tau}(\| \eM \cdot \nbhsm \|_{L_h^2(\Ghsm)}^2 - \| \ehm \cdot \nbhsm \|_{L_h^2(\Ghsm)}^2) 
	\notag\\
	&\leq - J^m(I_h(\eM \cdot \nbhsm)) - d^m(I_h(\eM \cdot \nbhsm)) ,
\end{align} 
and consequently
\begin{align}\label{error-estimate-1}
	& \frac{1}{2\tau} (\| \ehM \cdot \nbhsM \|_{L_h^2(\hat\Gamma_{h,*}^{m+1})}^2 - \| \ehm \cdot \nbhsm \|_{L_h^2(\Ghsm)}^2)  \notag\\
	&= \frac{1}{2\tau}(\| \eM \cdot \nbhsm \|_{L_h^2(\Ghsm)}^2 - \| \ehm \cdot \nbhsm \|_{L_h^2(\Ghsm)}^2)  \notag\\
	&\quad+ \frac{1}{2\tau}(\| \ehM \cdot \nbhsM \|_{L_h^2(\hat\Gamma_{h,*}^{m+1})}^2 - \| \eM \cdot \nbhsm \|_{L_h^2(\Ghsm)}^2)
	\notag\\
	&{\lesssim (\tau + (1 + \ksm) h^{k})^2 + \big(\| \ehM \|_{L^2(\Ghsm)}^2 + \| \eM \|_{L^2(\Ghsm)}^2 + \| \ehm \|_{L^2(\hat\Gamma_{h, *}^{m})}^2  \big)} ,
\end{align} 
where the inequality follows from \eqref{e-est-mass}, \eqref{eq:cons_err}, \eqref{eq:JeM} and Lemma \ref{lemma:e-convert}.

Then, by using \eqref{eq:e_NT}, \eqref{eq:ehM-L2} and \eqref{eq:NT_stab_L22}, we are able to convert all norms appearing on the right-hand side of \eqref{error-estimate-1} into $\| \ehm\cdot \nbhsm \|_{L_h^2(\Ghsm)}$ up to a consistency error of order $O(\tau)$:
\begin{align}\label{error-estimate-2}
	& \frac{ \| \ehM\cdot \nbhsM \|_{L_h^2(\hat\Gamma_{h,*}^{m+1})}^2 - \| \ehm\cdot \nbhsm \|_{L_h^2(\Ghsm)}^2 }{2\tau}
	\notag\\
	&\lesssim
	( \tau + (1 + \ksm)  h^{k})^2
		+ \| \ehm\cdot \nbhsm \|_{L_h^2(\Ghsm)}^2 .
\end{align} 
Applying the discrete Gr\"onwall's inequality and \eqref{eq:NT_stab_L22}, we obtain the main error estimate in $L_t^\infty L_x^2$ norm:
\begin{align}\label{eq:err_fin0}
	&\max_{1\le m\le l} \| \hat e_h^{m} \|_{L^2(\Ghsm)}^2 
	\leq 4
	\max_{1\le m\le l} \| \hat e_h^{m} \cdot \nbhsm \|_{L_h^2(\Ghsm)}^2 
	\notag\\
	&\le 
	C
	\| \hat e_h^{0} \|_{L^2(\Gamma_h^0)}^2 
	+
	\sum_{m=0}^{l-1} C_\km \tau ( \tau + (1 + \ksm)  h^{k})^2 
	\notag\\
	&\le 
	C h^{2k+2}
	+
	C_{\kappa_{l-1}} ( \tau + (1 + \kappa_{*, l-1})  h^{k})^2 
	\qquad 1\leq l\leq [T/\tau]
	,
\end{align}
where we have used the assumption that the initial approximation satisfies \eqref{P0}--\eqref{P1}.
Therefore, it remains to show the boundedness of the shape regularity constants.

\subsection{Shape regularity of $\Ghsm$ in $H_h^k$ norm}
\label{sec:bbd}

In this subsection, we are going to prove the a priori boundedness of $\kappa_{*,[T/\tau]}$. We regard $\hat X_{h,*}^m$ and $X_h^m$ as the maps from the piecewise flat curve $\Gamma_{h,\rm f}^0$ to $\Ghsm$ and $\Gamma_h^m$, respectively. Let $v^m_{\rm f} = v^m \circ a^m \circ \hat X_{h,*}^m$ and $g^m_{\rm f} = g^m \circ a^m \circ \hat X_{h,*}^m$, which are functions defined on the piecewise flat curve $\Gamma_{h,\rm f}^0$. 
By using relations \eqref{eq:geo_rel_4}--\eqref{eq:geo_rel_6}, we can decompose the displacement as follows for any $0\leq m \leq [T/\tau]-1$:
\begin{align*}
	&\|\hat X_{h,*}^{m + 1} -\hat X_{h,*}^{m}\|_{W^{j,\infty}_h(\Gamma_{h,\rm f}^0)} \notag\\
	&\le \| I_h [(X^{m + 1} - {\rm id})\circ a^m \circ \hat X_{h,*}^{m} ]\|_{W^{j,\infty}_h(\Gamma_{h,\rm f}^0)}
	+ \| \rho_h\circ \hat X_{h,*}^{m} \|_{W^{j,\infty}_h(\Gamma_{h,\rm f}^0)} \\
	&\quad+\| I_h[ I_h(T_*^m\circ \hat X_{h,*}^{m}) I_h (N_*^{m+1}\circ \hat X_{h,*}^{m+1}-N_*^{m}\circ \hat X_{h,*}^{m})  (\hat e_{h}^{m + 1} \circ \hat X_{h,*}^{m})]\|_{W^{j,\infty}_h(\Gamma_{h,\rm f}^0)} \\
	&\quad+\| I_h[(T_*^m \circ \hat X_{h,*}^{m} ) (X_{h}^{m + 1} - X_{h}^{m})  ] \|_{W^{j,\infty}_h(\Gamma_{h,\rm f}^0)} \\
	&=: E_1^m + E_2^m + E_3^m. 
\end{align*} 
From the stability of $I_h$ on $C^0(\Gamma_{h,\rm f}^0)\cap W^{k,\infty}_h(\Gamma_{h,\rm f}^0)$, chain rule, the inverse inequality and \eqref{eq:geo_rel_5}, we derive 
\begin{align}\label{eq:E1}
	E_1^m 
	&\le C_0 \| (X^{m + 1} - {\rm id})\circ a^m \circ \hat X_{h,*}^{m} \|_{W^{j,\infty}_h(\Gamma_{h,\rm f}^0)}
	+
	\| \rho_h\circ \hat X_{h,*}^{m} \|_{W^{j,\infty}_h(\Gamma_{h,\rm f}^0)}
	\notag\\
	&\le 
	C_0  \|  X^{m+1} - {\rm id} \|_{W^{j,\infty}(\Gm)} 
	\Big(1+ \sum_{j_1+\cdots+j_i\le j\atop j_1,\dots,j_i\ge 1} 
	\|\hat X_{h, *}^{m}\|_{W_h^{j_1,\infty}(\Gamma_{h,\rm f}^0)}\cdots\|\hat X_{h, *}^{m}\|_{W_h^{j_i,\infty}(\Gamma_{h,\rm f}^0)} \Big) \notag\\
	&\quad
	+ C_0 h^{-j} \| \rho_h\circ \hat X_{h,*}^{m} \|_{L^{\infty}(\Gamma_{h,\rm f}^0)}
	\notag\\
	&\le 
	C_0 \tau [1 + j(j-1)\|\hat X_{h, *}^{m}\|_{W_h^{j-1,\infty}(\Gamma_{h,\rm f}^0)}^j]
	+
	C_0j \tau \|\hat X_{h, *}^{m}\|_{W_h^{j,\infty}(\Gamma_{h,\rm f}^0)} \notag\\
	&\quad
	+ C_0 h^{-j} \big(\tau^2 + \|I_h \Tsm (\hat X_{h,*}^{m+1} - \hat X_{h,*}^{m}) \|_{L^\infty(\Ghsm)}^2\big) \notag\\
	&\le 
	C_\km h^{-j} \tau^2 
	+
	C_0 \tau [1 + j(j-1)\|\hat X_{h, *}^{m}\|_{W_h^{j-1,\infty}(\Gamma_{h,\rm f}^0)}^j]
	+
	C_0j \tau \|\hat X_{h, *}^{m}\|_{W_h^{j,\infty}(\Gamma_{h,\rm f}^0)} ,
\end{align}
where we have applied \eqref{eq:geo_rel_5} and \eqref{eq:hat_X_s_diff1} in the third and fourth inequalities respectively.
Here we have added a factor $j(j-1)$ in front of $\|\hat X_{h, *}^{m}\|_{W_h^{j-1,\infty}(\Gamma_{h,\rm f}^0)}$ to indicate that this term should disappear in the case $j=1$, and we have added a factor $j$ in front of $\|\hat X_{h, *}^{m}\|_{W_h^{j,\infty}(\Gamma_{h,\rm f}^0)}$ to indicate that this term should disappear in the case $j=0$. 

Furthermore, using the inverse inequality, we have 
\begin{align}\label{eq:E2}
	E_2^m 
	&\le C_0h^{-j-1/2} \| I_h[ I_h(T_*^m\circ \hat X_{h,*}^{m}) I_h (N_*^{m+1}\circ \hat X_{h,*}^{m+1}
	\notag\\
	&\hspace{100pt}
	-N_*^{m}\circ \hat X_{h,*}^{m})  (\hat e_{h}^{m + 1} \circ \hat X_{h,*}^{m})]\|_{L^{2}_h(\Gamma_{h,\rm f}^0)} 
	\notag\\
	&\le C_\km h^{-j-1/2} \tau (\tau + (1 + \ksm) h^{k}) ,
\end{align} 
where we have used the estimate $\|N_*^{m+1}\circ \hat X_{h,*}^{m+1}-N_*^{m}\circ \hat X_{h,*}^{m}\|_{L^{\infty}_h(\Gamma_{h,\rm f}^0)}\le C_\km \tau$ which follows from \eqref{eq:hat_X_s_diff1} and \eqref{nsM-nsm-nodes1}, 
and the error estimate $\| \hat e_{h}^{m + 1} \circ \hat X_{h,*}^{m} \|_{L^{2}_h(\Gamma_{h,\rm f}^0)}\le C_\km (\tau + (1 + \ksm)h^{k})$ which follows from the error estimate \eqref{eq:err_fin0}.

Using relation \eqref{eq:geo_rel_v} we can estimate $E_3$ in the above inequality as follows:
\begin{align*} 
	&E_3
	=
	\| I_h[(T_*^m \circ \hat X_{h,*}^{m} ) (X_{h}^{m + 1} - X_{h}^{m})  ]\|_{W^{j, \infty}_h(\Gamma_{h,\rm f}^0)} \notag\\
	&\leq \| I_h[(T_*^m \circ \hat X_{h,*}^{m} ) (X_{h}^{m+1} - X_{h}^{m} - \tau I_h v^m_{\rm f}) ] \|_{W^{j, \infty}_h(\Gamma_{h,\rm f}^0)} 
	\notag\\
	&\quad
	+ \tau\| I_h[(T_*^m \circ \hat X_{h,*}^{m} ) v^m_{\rm f}] \|_{W^{j, \infty}_h(\Gamma_{h,\rm f}^0)} \notag\\
	&= \| I_h[(T_*^m \circ \hat X_{h,*}^{m} ) (\eM - \ehm  - \tau I_h \Tsm v^m_{\rm f} + \tau I_h g^m_{\rm f} )]  \|_{W^{j, \infty}_h(\Gamma_{h,\rm f}^0)} 
	\notag\\
	&\quad
	+ \tau
	\| (T^m v^m)\circ a^m \circ \hat X_{h,*}^m \|_{W^{j, \infty}_h(\Gamma_{h,\rm f}^0)} 
	\\
	&\leq \| I_h[(T_*^m \circ \hat X_{h,*}^{m} ) (\eM - \ehm  - \tau I_h \Tsm v^m_{\rm f} )]  \|_{W^{j, \infty}_h(\Gamma_{h,\rm f}^0)} 
	\notag\\
	&\quad
	+ C_\km h^{-j+1}\tau^2 
	+
	C_0 \tau [1 + j(j-1)\|\hat X_{h, *}^{m}\|_{W_h^{j-1,\infty}(\Gamma_{h,\rm f}^0)}^j]
	+
	C_0j \tau \|\hat X_{h, *}^{m}\|_{W_h^{j,\infty}(\Gamma_{h,\rm f}^0)}
	 ,
\end{align*}  
where the last inequality follows from the following estimates: 
\begin{align*}
	&\| I_h g^m_{\rm f} \|_{W^{j, \infty}_h(\Gamma_{h,\rm f}^0)} 
	\le h^{-j+1} \| I_h g^m_{\rm f} \|_{W^{1, \infty}(\Gamma_{h,\rm f}^0)} \le C_\km h^{-j+1}\tau 
	\qquad\mbox{(in view of \eqref{W1infty-g})} ,\\
	&
	\| (T^m v^m)\circ a^m \circ \hat X_{h,*}^m \|_{W^{j, \infty}_h(\Gamma_{h,\rm f}^0)} 
	\le 
	C_0 [1 + j(j-1)\|\hat X_{h, *}^{m}\|_{W_h^{j-1,\infty}(\Gamma_{h,\rm f}^0)}^j]
	\notag\\
	&\hspace{150pt}
	+
	C_0j \|\hat X_{h, *}^{m}\|_{W_h^{j,\infty}(\Gamma_{h,\rm f}^0)} 
	\notag\\
	&\hspace{150pt}\mbox{(chain rule of differentation, cf. \eqref{eq:E1})} . 
\end{align*}
We continue the estimate for $E_3$:
\begin{align}\label{eq:E3}
	E_3
	&\leq  C_0h^{-j+3/2}\| I_h[(T_*^m \circ \hat X_{h,*}^{m} ) (\eM - \ehm  - \tau I_h \Tsm v^m_{\rm f} )]  \|_{H_h^{2}(\Gamma_{h,\rm f}^0)} 
	\notag\\
	&\quad
	+ C_\km h^{-j+1}\tau^2 
	+
	C_0 \tau [1 + j(j-1)\|\hat X_{h, *}^{m}\|_{W_h^{j-1,\infty}(\Gamma_{h,\rm f}^0)}^j]
	+
	C_0j \tau \|\hat X_{h, *}^{m}\|_{W_h^{j,\infty}(\Gamma_{h,\rm f}^0)}
	\notag\\
	&\leq C_\km h^{-j+1/2} \tau  (\tau + (1 + \ksm) h^{k}) 
	+ 
	C_\km h^{-j+1/2} \tau  \| \ehm \|_{H^1(\Ghsm)}
	\notag\\
	&\quad
	+ C_\km h^{-j+1}\tau^2 
	+
	C_0 \tau [1 + j(j-1)\|\hat X_{h, *}^{m}\|_{W_h^{j-1,\infty}(\Gamma_{h,\rm f}^0)}^j]
	+
	C_0j \tau \|\hat X_{h, *}^{m}\|_{W_h^{j,\infty}(\Gamma_{h,\rm f}^0)}
	\notag\\
	&\leq C_\km h^{-j-1/2} \tau  (\tau + (1 + \ksm) h^{k})
	\notag\\
	&\quad
	+
	C_0 \tau [1 + j(j-1)\|\hat X_{h, *}^{m}\|_{W_h^{j-1,\infty}(\Gamma_{h,\rm f}^0)}^j]
	+
	C_0j \tau \|\hat X_{h, *}^{m}\|_{W_h^{j,\infty}(\Gamma_{h,\rm f}^0)}
	,
\end{align}
where we have used the velocity estimate \eqref{eq:vel_T_H2} and error estimate \eqref{eq:err_fin0} in the second and third inequalities respectively.

%

Collecting the estimates for $E_1$, $E_2$ and $E_3$ (Eqs. \eqref{eq:E1}--\eqref{eq:E3}),
\begin{align}\label{eq:Wj_inf}
&\|\hat X_{h,*}^{m + 1} -\hat X_{h,*}^{m}\|_{W^{j,\infty}_h(\Gamma_{h,\rm f}^0)} \notag\\ 
&\leq C_\km h^{-j-1/2} \tau  (\tau + (1 + \ksm) h^{k})
\notag\\
&\quad
+
C_0 \tau [1 + j(j-1)\|\hat X_{h, *}^{m}\|_{W_h^{j-1,\infty}(\Gamma_{h,\rm f}^0)}^j]
+
C_0j \tau \|\hat X_{h, *}^{m}\|_{W_h^{j,\infty}(\Gamma_{h,\rm f}^0)}
.
\end{align} 
Due to the stepsize condition $\tau\leq ch^k$, for sufficiently small mesh size $h\leq h_{\km,\ksm}$ ($h_{\km,\ksm}$ is some constant which depends on $\km$ and $\ksm$), we have
$$
C_\km h^{-j-1/2} \tau  (\tau + (1 + \ksm) h^{k})
\leq
C_{\km,\ksm} \tau h^{k-j-1/2} 
\leq C_0 \tau
\qquad 0\leq j\leq k-1 .
$$
Therefore, if we take $j=0$, by using the triangle inequality, 
\begin{align}
	&\|  \hat X_{h,*}^{m+1} \|_{L^\infty(\Ghso)} - \|  \hat X_{h,*}^0 \|_{L^\infty(\Ghso)} \notag\\
	&\le \sum_{r=0}^m\|\hat X_{h,*}^{r + 1} -\hat X_{h,*}^{r}\|_{L^\infty(\Gamma_{h,\rm f}^0)} \notag\\ 
	&\le 
	C_0 
	+ \sum_{r=0}^m C_0 \tau \| \hat X_{h,*}^r \|_{L^\infty(\Gamma_{h,\rm f}^0)}
	\notag
	 .
\end{align}
By applying the discrete Gr\"onwall's inequality, we obtain the following result under the mesh size condition $h\leq h_{\km,\ksm}$: 
\begin{align}
	\max_{0\le m\le [T/\tau]} \|  \hat X_{h,*}^{m} \|_{L^\infty(\Ghso)} 
	&\le 
	C_0 . \notag
\end{align}
Note that the right-hand side of \eqref{eq:Wj_inf} is linear up to the leading order $\|  \hat X_{h,*}^{m} \|_{W_h^{j,\infty}}$. Therefore, recursively increasing the regularity exponent $j$, we can prove
\begin{align}\label{eq:k_bdd}
	\max_{0\le m\le [T/\tau]} \|  \hat X_{h,*}^{m} \|_{W_h^{j,\infty}(\Ghso)} 
	&\le 
	C_0 \qquad 0\leq j \leq k-1. 
\end{align}
The proof of $\| (\hat X_{h,*}^{m+1})^{-1} \|_{W^{1,\infty}(\hat\Gamma_{h,*}^{m+1})}\le C_0$ with $0\leq m\leq [T/\tau - 1]$ is simpler, i.e., the same as
\cite[Eq. (4.95)]{BL24} and \cite[Eq. (6.5)]{BGV25b}, and therefore omitted. 
This boundedness implies the $W^{1,p}$, $1\leq p\leq \infty$, norm equivalence on the discrete curves.
This shows $\kappa_{[T/\tau]} \leq C_0$ for $h\leq h_{C_0,\kappa_{*, [T/\tau]}}$. 

It remains to show $\kappa_{*, [T/\tau]} \leq C_0$. To this end, we need the following $W^{j, 2}$ version of \eqref{eq:Wj_inf} whose proof is similar and hence omitted here: For any $0\leq j\leq k$ and $0\leq m \leq [T/\tau]-1$,
\begin{align}\label{eq:Wj_2}
	&\|\hat X_{h,*}^{m + 1} -\hat X_{h,*}^{m}\|_{W^{j,2}_h(\Gamma_{h,\rm f}^0)} \notag\\ 
	&\leq C_\km h^{-j} \tau  (\tau + (1 + \ksm) h^{k})
	\notag\\
	&\quad
	+
	C_0 \tau [1 + j(j-1)\|\hat X_{h, *}^{m}\|_{W_h^{j-1,\infty}(\Gamma_{h,\rm f}^0)}^j]
	+
	C_0j \tau \|\hat X_{h, *}^{m}\|_{W_h^{j,2}(\Gamma_{h,\rm f}^0)}
	.
\end{align} 
Taking $j=k$,
\begin{align}
	&\|\hat X_{h,*}^{m + 1} -\hat X_{h,*}^{m}\|_{H^{k}_h(\Gamma_{h,\rm f}^0)} \notag\\ 
	&\leq C_\km h^{-k} \tau (\tau + (1 + \ksm) h^{k})
	+
	C_0 k \tau \|\hat X_{h, *}^{m}\|_{H_h^{k}(\Gamma_{h,\rm f}^0)}
	\notag\\
	&\leq
	C_0 \tau (1 + h^{-k}\tau) 
	+
	C_0 \tau (\ksm + \|\hat X_{h, *}^{m}\|_{H_h^{k}(\Gamma_{h,\rm f}^0)})
	, \notag
\end{align} 
where in the last step we have used the uniform boundedness $\kappa_m \leq C_0$ from \eqref{eq:k_bdd}.
Therefore, by using the triangle inequality, 
\begin{align}\label{eq:ks_Gronwall}
	&\kappa_{*,m+1} - \|  \hat X_{h,*}^0 \|_{W_h^{k, 2}(\Ghso)} \notag\\
	&= \max_{0 \leq r\leq m+1} \{\|  \hat X_{h,*}^{r} \|_{W_h^{k, 2}(\Ghso)} - \|  \hat X_{h,*}^0 \|_{W_h^{k, 2}(\Ghso)}
	\} 
	\notag\\
	&\le \sum_{r=0}^m\|\hat X_{h,*}^{r + 1} -\hat X_{h,*}^{r}\|_{W^{k,2}_h(\Gamma_{h,\rm f}^0)} \notag\\ 
	&\le 
	\sum_{r=0}^m C_0 \tau (1 + h^{-k}\tau) 
	 + \sum_{r=0}^m C_0 \tau(\kappa_{*,r} + \| \hat X_{h,*}^r \|_{W^{k, 2}_h(\Gamma_{h,\rm f}^0)}) \notag\\
	&\le 
	C_0 + \sum_{r=0}^m C_0 \tau \kappa_{*,r}
	\quad\mbox{(under condition $\tau\leq ch^k$)} ,
\end{align}
where we have used the trivial size relation $\|\hat X_{h, *}^{m}\|_{H_h^{k}(\Gamma_{h,\rm f}^0)} \leq \ksm$.
By applying the discrete Gr\"onwall's inequality, we obtain the following boundedness results under the mesh size condition $h\leq h_{C_0,\kappa_{*,[T/\tau]}}$: 
\begin{align}\label{eq:ks_bdd}
	\begin{split}
	\kappa_{*,[T/\tau]}
	&\le 
	C_0 ,
	\\
	\max_{0\le m\le [T/\tau]} \|  \hat X_{h,*}^{m} \|_{W_h^{k, 2}(\Ghso)} 
	&\le 
	C_0 . 
\end{split}
\end{align}
Consequently, the mesh size condition can be improved to $h\leq h_{C_0,C_0}$ where $h_{C_0,C_0}$ is a constant which is independent of $h$ and $\tau$.


The induction hypothesis is recovered by combining \eqref{eq:err_fin0} and \eqref{eq:ks_bdd}. Therefore, the proof of Theorem \ref{thm:main} is complete.
\begin{remark}\upshape \label{rmk:shape-reg}
	The time stepsize condition $\tau\leq ch^k$ is not essentially necessary: We observe from \eqref{eq:ks_Gronwall} and \eqref{eq:ks_bdd} that without the condition $\tau\leq ch^k$ we are still able to conclude from Gr\"onwall inequality $\kappa_{*,[T/\tau]} \leq C_0 h^{-k}\tau$, assuming the boundedness of the lower-order Sobolev norms \eqref{eq:k_bdd}. On the right-hand side of the error equation \eqref{eq:err_fin0}, the product structure still leads to the desired convergence rate $\kappa_{*,[T/\tau]} h^k \leq C_0 \tau$.
	This argument suggests that the time stepsize constraint can be eliminated at the leading order. However, due to the appearance of lower-order Sobolev norms of the parametrization map $\hat X_{h,*}^m$ in a multilinear form --- arising from the chain rule (see, for instance, the second term on the right-hand side of \eqref{eq:E1}) --- it remains unclear how to remove the stepsize restriction for such terms. Resolving this issue will be the subject of future research.
\end{remark}


\section{Numerical experiments}
\label{section:numerical}

We consider the evolution of a parameterized ellipse
\begin{align}\label{eq:dumbbell}
	\begin{pmatrix} x(\xi) \\ y(\xi) \end{pmatrix} 
	=
	\begin{pmatrix} \cos(2\pi\xi) \\ \sin(2\pi\xi) / 3 \end{pmatrix}, \quad \xi \in [0,1] ,
\end{align}
under the prescribed velocity field (depicted in Figure \ref{fig:ell-sp1})
\begin{align}\label{eq:exp_v}
	v = (v_x, v_y) = \Big(1 - \Big(\frac{x^2 + y^2}{x^2 + 9 y^2}\Big)^{1/2} \Big)
	\frac{(x, y)}{(x^2 + y^2)^{1/2}} .
\end{align}
This velocity field represents a radial transport with constant speed which only depends on $\theta := \tan^{-1}\frac{y}{x}$. Moreover, $v$ is constructed in such a way that at the final time $T=1$ the ellipse will evolve into a unit sphere. By construction, $v$ also has a non-trivial tangential component which will distort the mesh if we do not add any tangential smoothing velocity.

We test the convergence of the proposed transport BGN method in  \eqref{eq:BGN_tr1}--\eqref{eq:BGN_tr2} on the time interval $[0, T]$, with $T=1$, under the time stepsize condition $\tau = O(h^k)$. For the mesh sizes $h=2^{-4}, 2^{-5}, 2^{-6}, 2^{-7}$, we measure the $L_t^\infty L_x^2$ norm of the error $\ehm$ with the time steps $N_t = 2^4, 2^5, 2^6, 2^7$ for $k=1$; $N_t = 2^4, 2^6, 2^8, 2^{10}$ for $k=2$; and $N_t = 2^3, 2^6, 2^9, 2^{12}$ for $k=3$ respectively. The corresponding errors are plotted in Figure \ref{fig:a}. We observe evident $O(h^k)$ convergence for all $k=1,2,3$, which are consistent with our main results (Theorem \ref{thm:main}). 
In order to test the sharpness of the time stepsize condition, we record the $L_t^\infty L_x^2$ norm of the error in Figure \ref{fig:b} under large time steps and fixed small mesh size $h=2^{-10}$. According to the results, the necessity of the stepsize condition $\tau\leq c h^k$ is not observed.
\begin{figure}[htbp!]
	\centering
	\subfigure[Spatial discretization errors]{\label{fig:a}\includegraphics[width=0.48\linewidth]{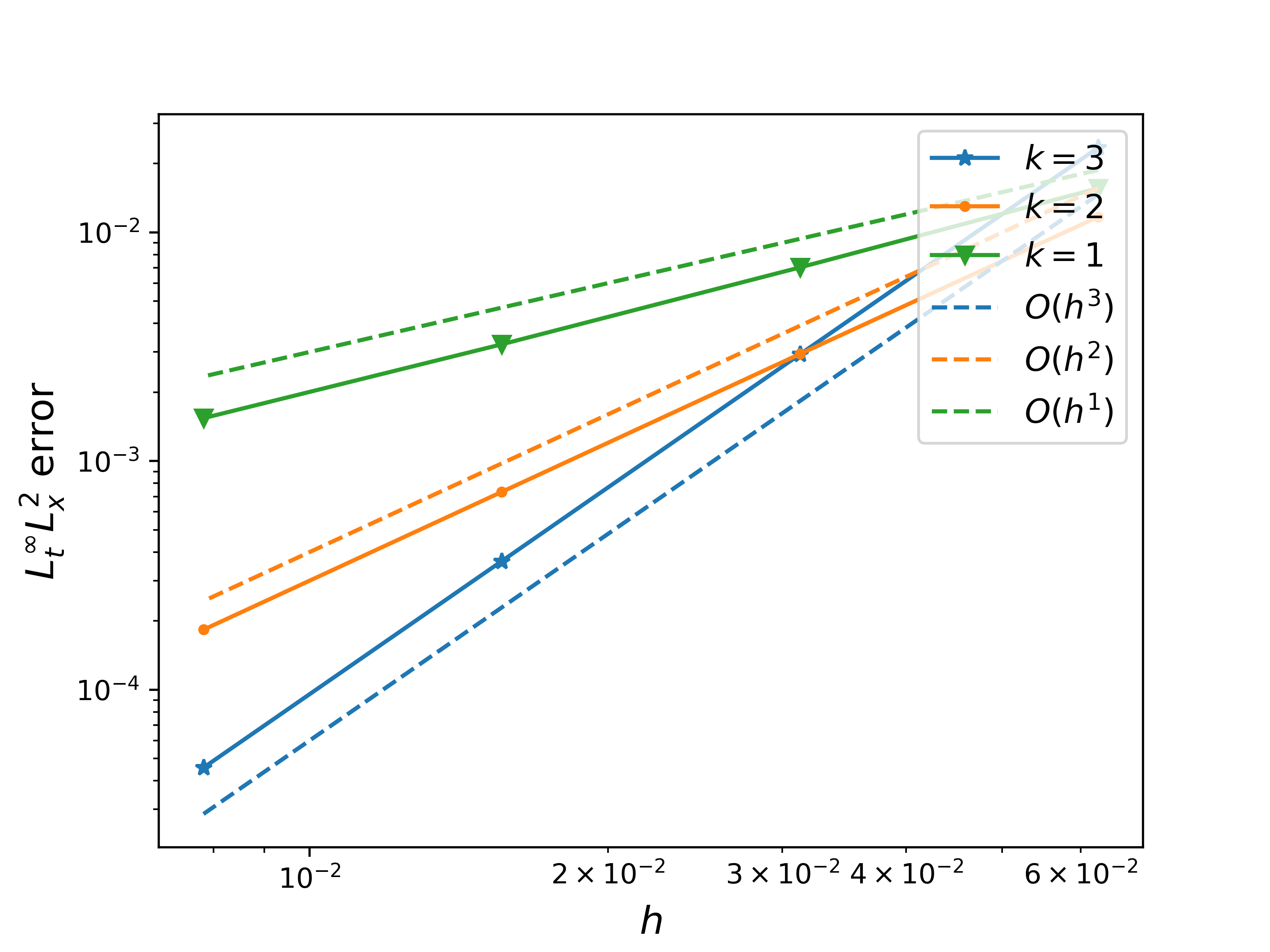}}
	\subfigure[Temporal discretization errors]{\label{fig:b}\includegraphics[width=0.48\linewidth]{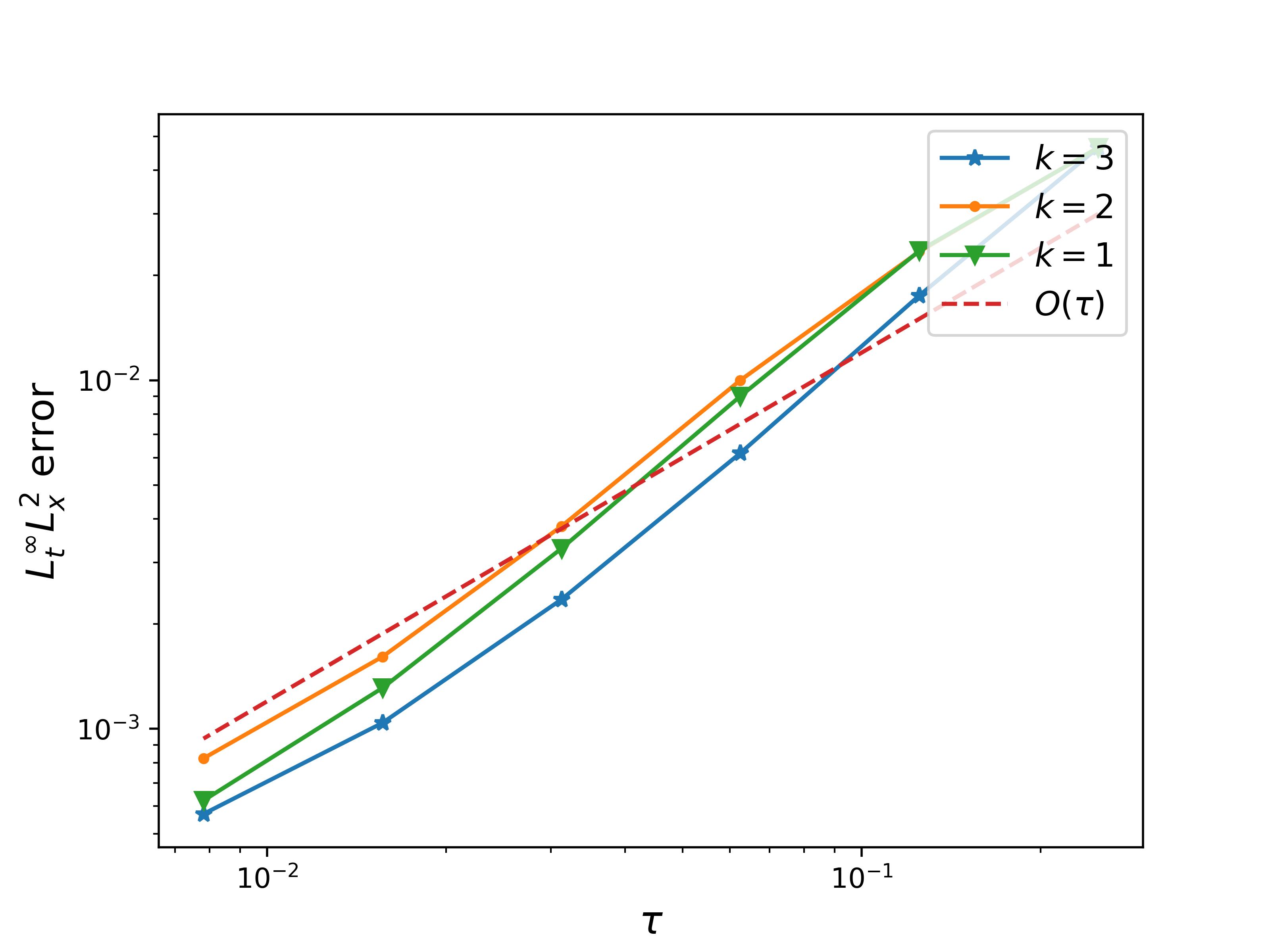}}
	\caption{Rate of convergence}
	\label{fig:e-rate}
\end{figure}

The tangential smoothing effect of the proposed method is shown in Figure \ref{fig:ratio1-ell2sp} where the mesh ratio $h_{\rm max} / h_{\rm min}$ decreases for the transport BGN method and increases for the evolution of \eqref{eq:exp_v} with no tangential redistribution.
The exact and discrete evolution of the curve \eqref{eq:dumbbell} under the velocity field defined in \eqref{eq:exp_v} is illustrated in Figure \ref{fig:ell-sp2} and Figure \ref{fig:ell-sp3}, respectively.

As a second example, we consider a velocity field that is rotation-dominated (see \cite[Example 4.3]{BHL24}):
\begin{align}\label{eq:exp_v1}
	v = (v_x, v_y) = 
	(1 - (x^2 + y^2)) (x,y)
	+
	\Big(1.2 - \frac{y}{(x^2 + y^2)^{1/2}} \Big) (-y, x)
	.
\end{align}
Figure \ref{fig:ATV} displays the velocity field \eqref{eq:exp_v1}, the mesh ratio, the exact evolution, and the discrete evolution for the same initial curve \eqref{eq:dumbbell}.
It can be observed that the proposed transport BGN method performs noticeably better than the plain exact evolution in terms of shape regularity and node distribution.
\begin{figure}[htbp!]
	\centering
	\subfigure[The velocity field $v$]{\label{fig:ell-sp1}\includegraphics[width=0.48\linewidth]{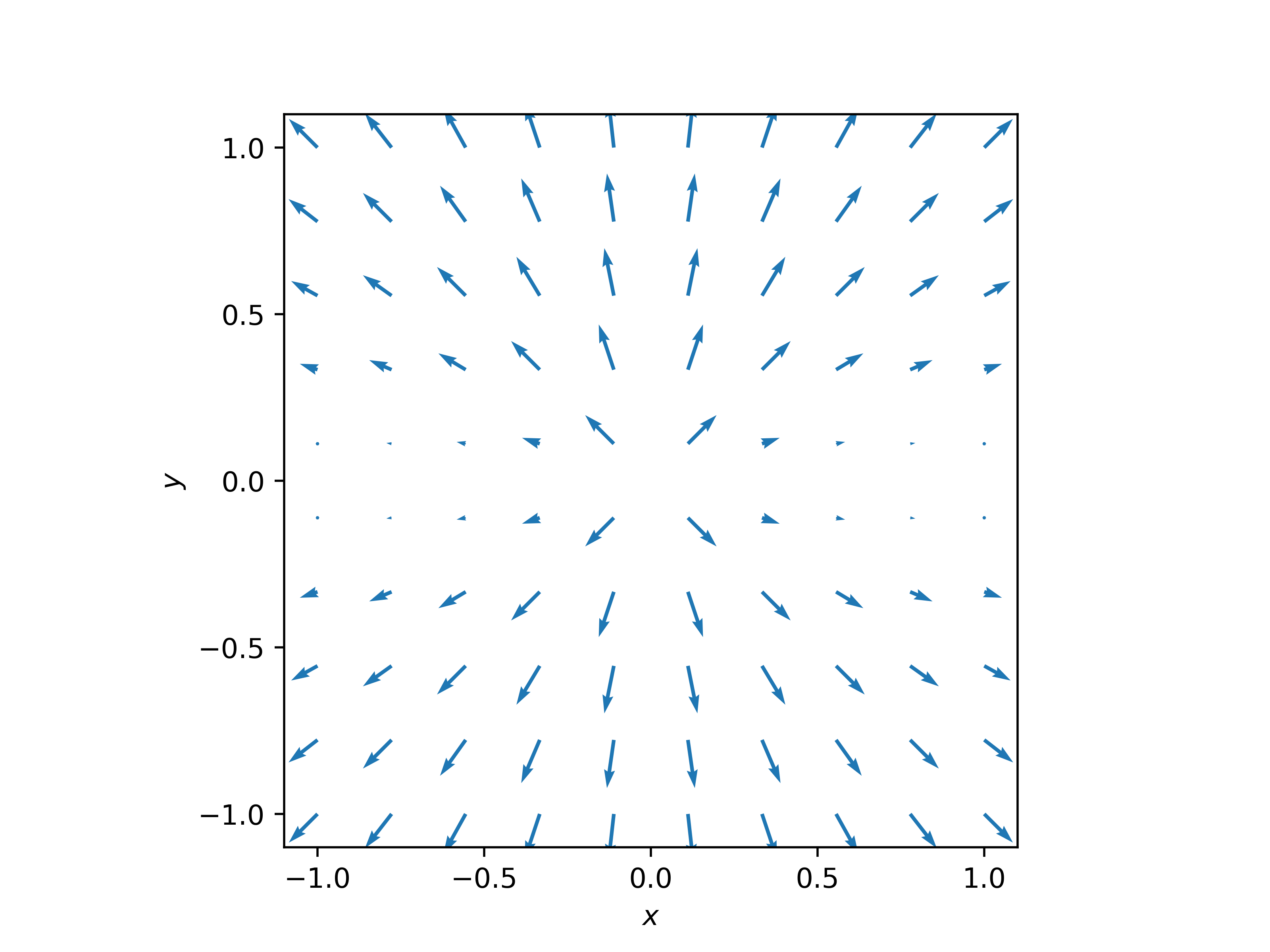}}
	\subfigure[Mesh ratio $h_{\rm max}/h_{\rm min}$]{\label{fig:ratio1-ell2sp}\includegraphics[width=0.48\linewidth]{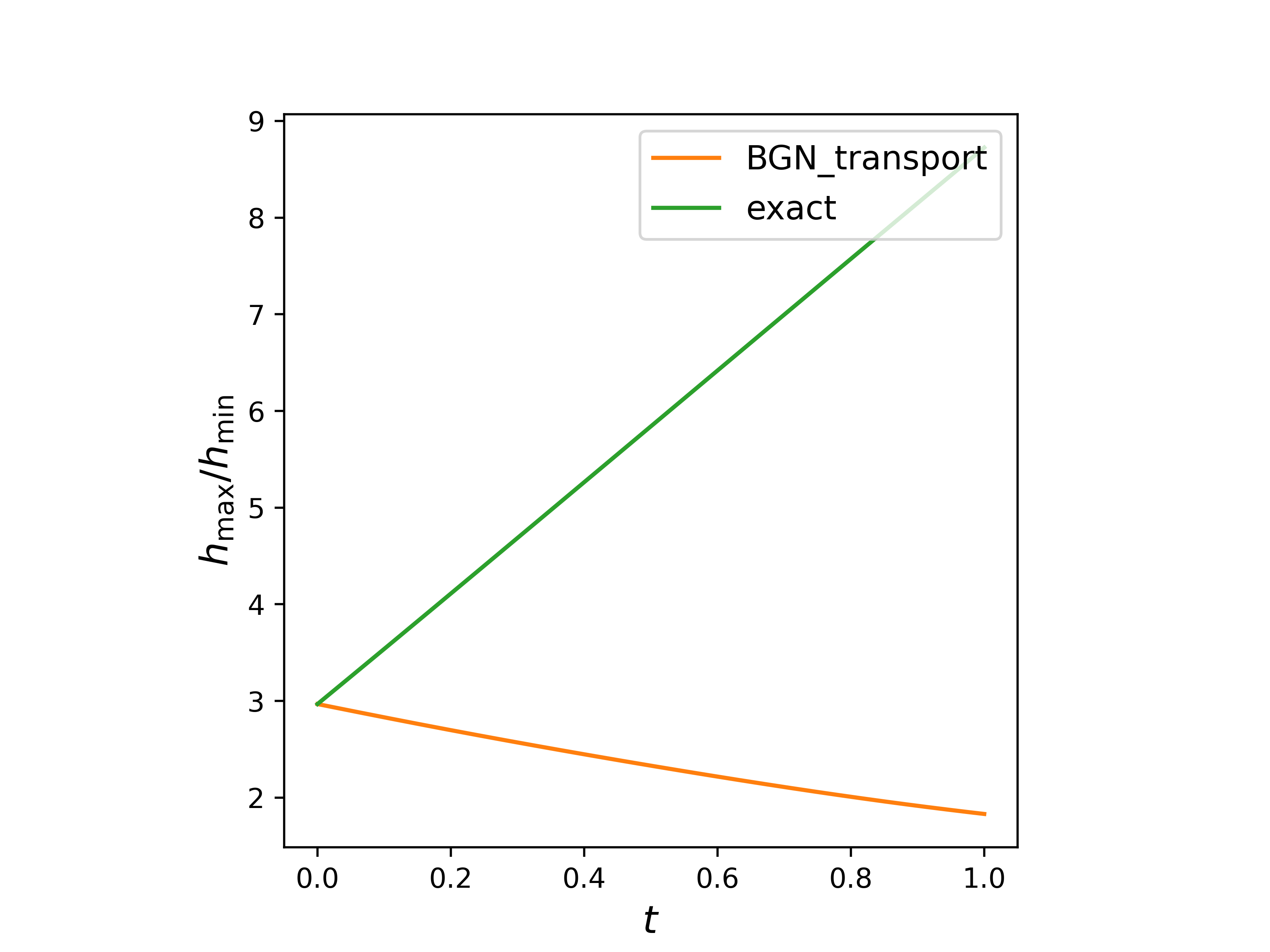}}
	\subfigure[Evolution under velocity $v$]{\label{fig:ell-sp2}\includegraphics[width=0.48\linewidth]{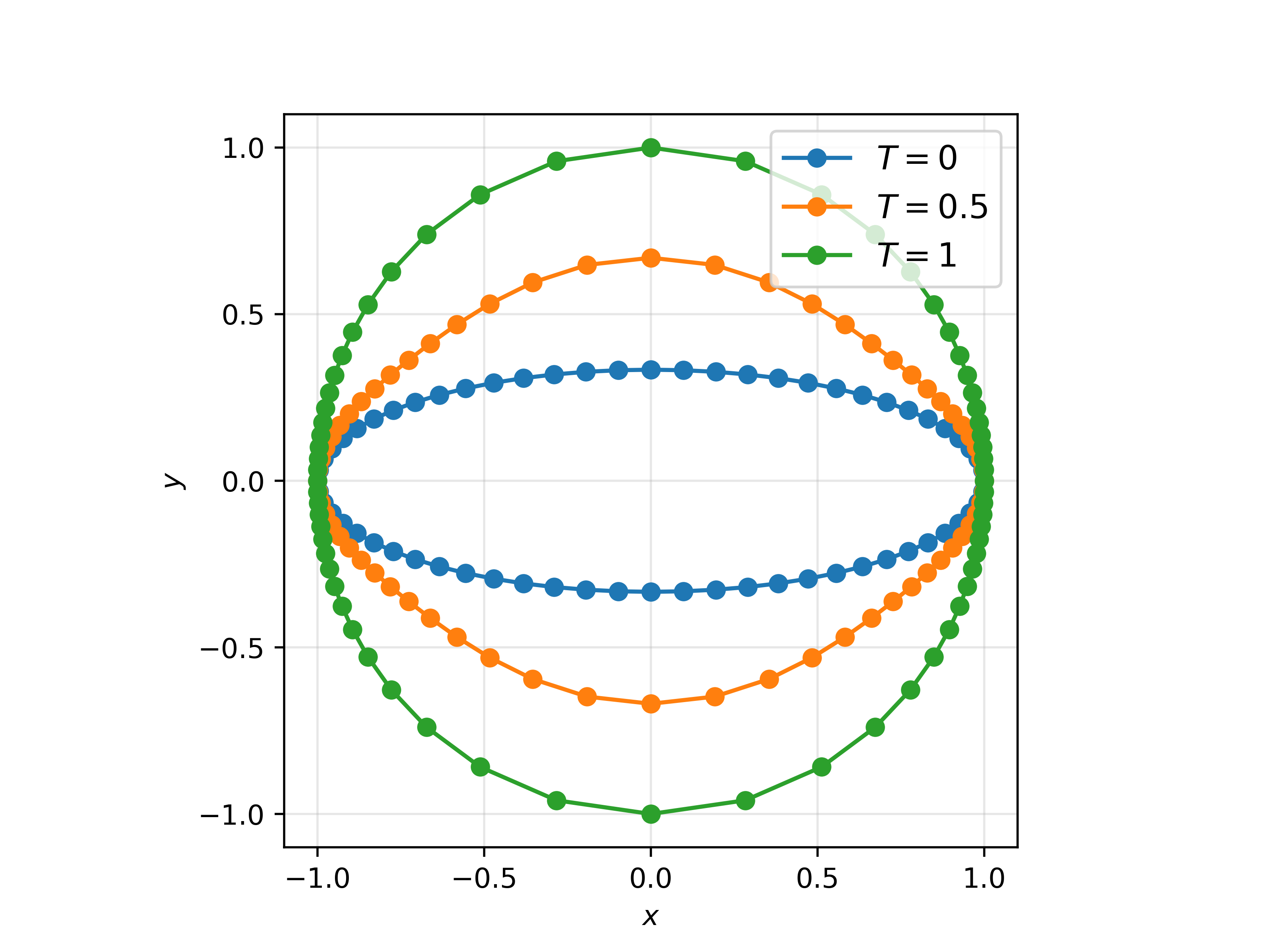}}
	\subfigure[Evolution in our method]{\label{fig:ell-sp3}\includegraphics[width=0.48\linewidth]{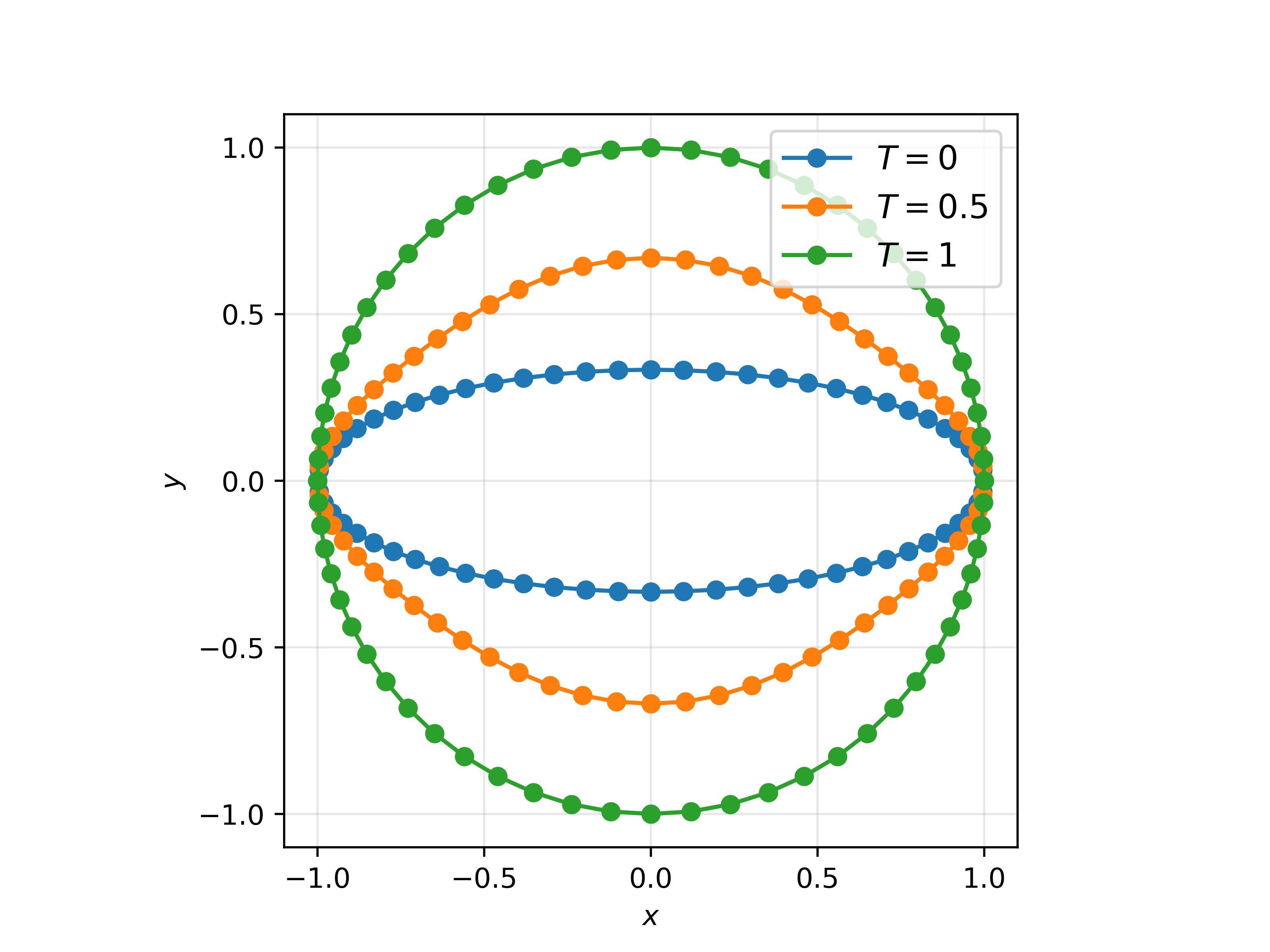}}
	\caption{Evolution of a curve $\Gamma(t)$ under velocity field $v$ in \eqref{eq:exp_v}}
	\label{fig:ell-sp}
\end{figure}

\begin{figure}[htbp!]
	\centering
	\subfigure[The velocity field $v$]{\label{fig:ATV1}\includegraphics[width=0.48\linewidth]{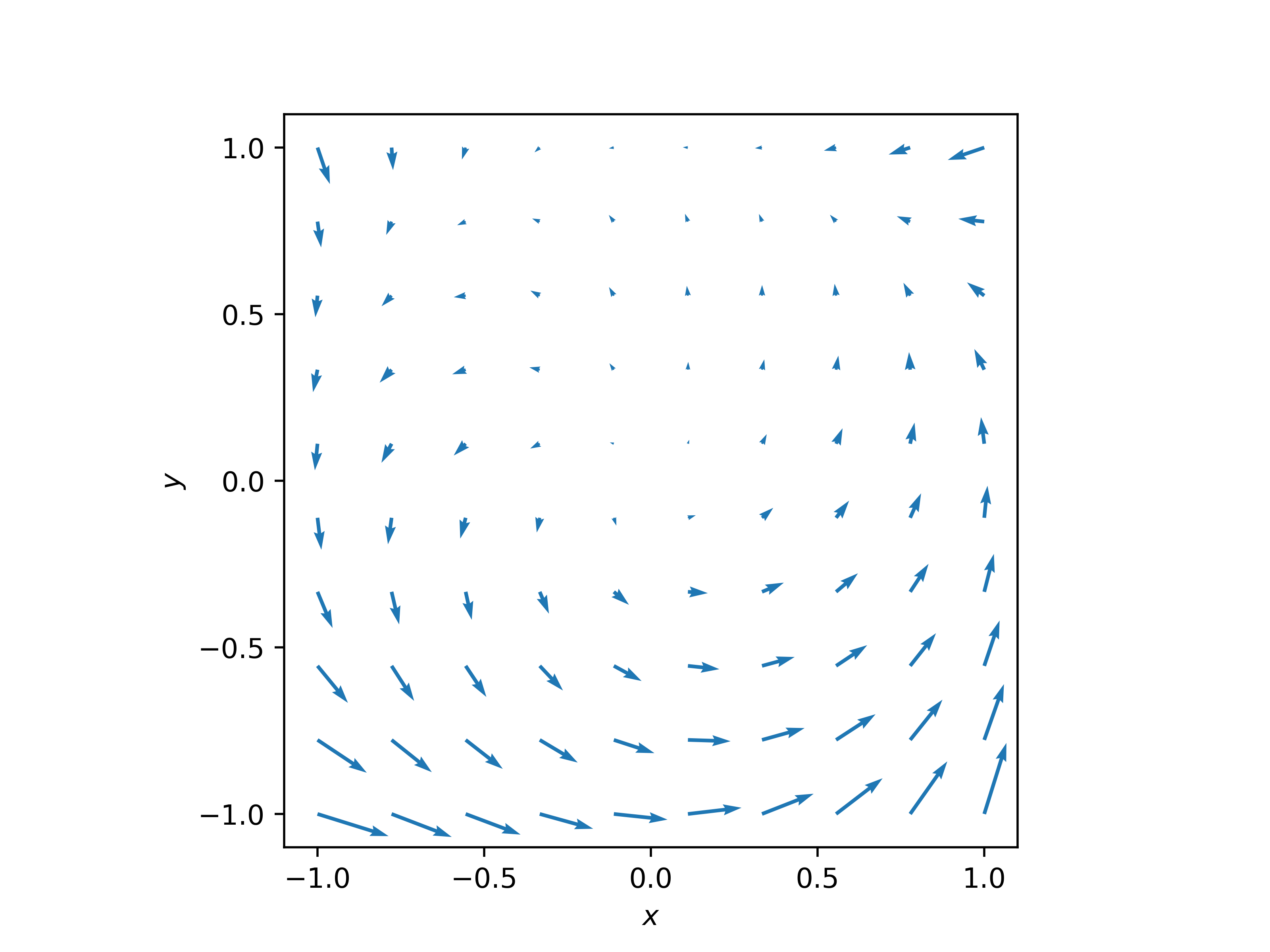}}
	\subfigure[Mesh ratio $h_{\rm max}/h_{\rm min}$]{\label{fig:ratio-ATV}\includegraphics[width=0.48\linewidth]{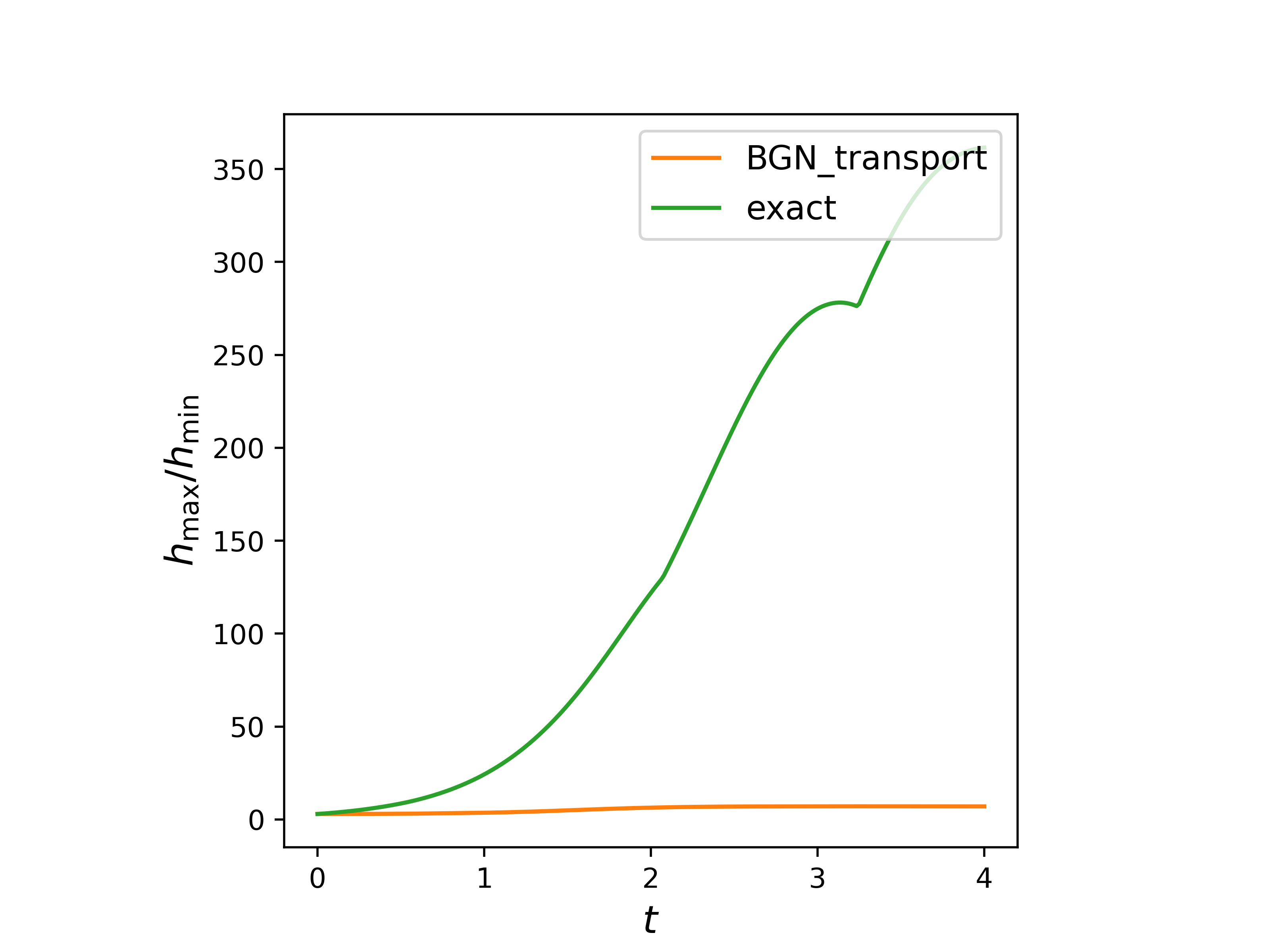}}
	\subfigure[Evolution under velocity $v$]{\label{fig:ATV2}\includegraphics[width=0.48\linewidth]{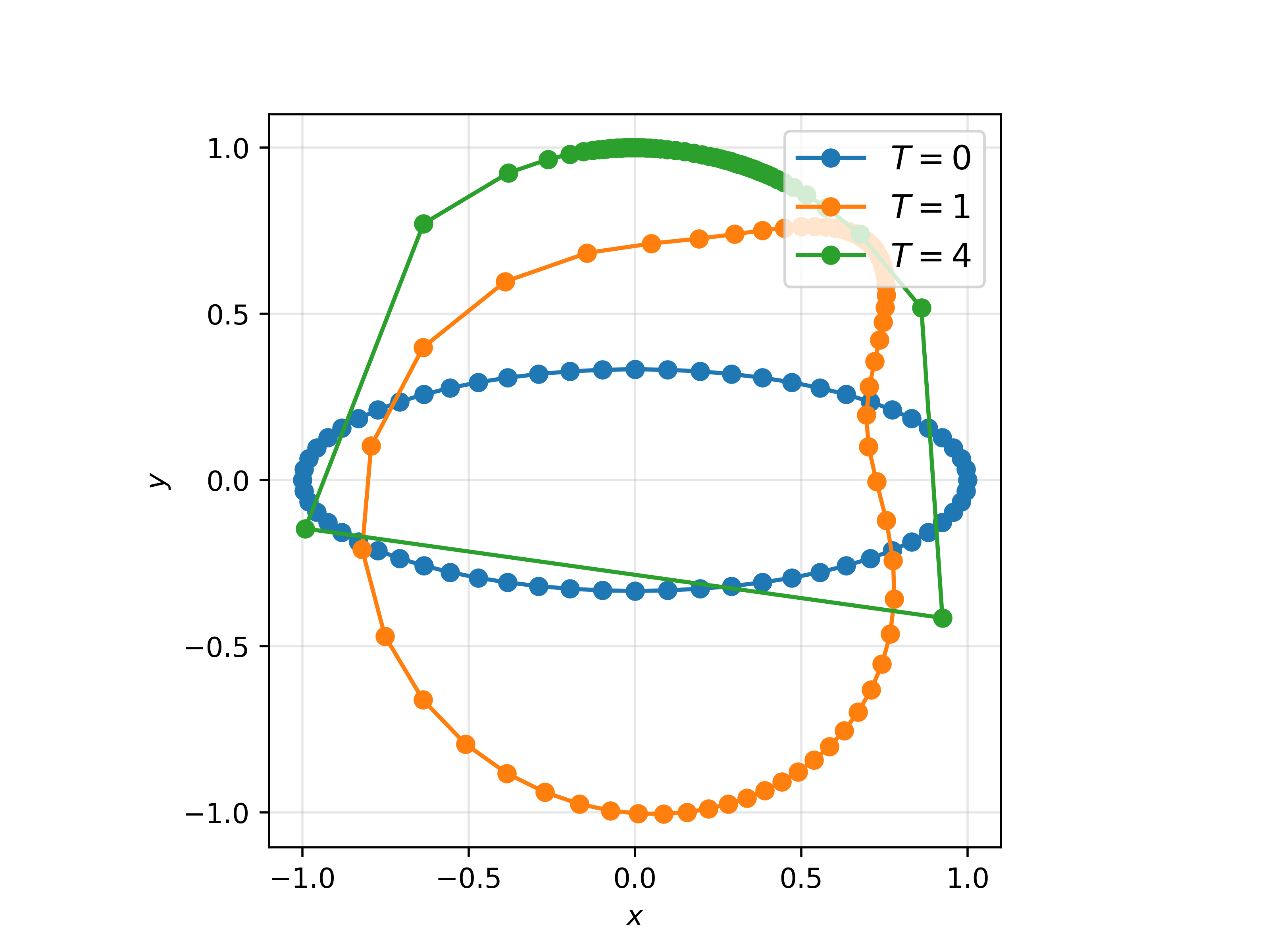}}
	\subfigure[Evolution in our method]{\label{fig:ATV3}\includegraphics[width=0.48\linewidth]{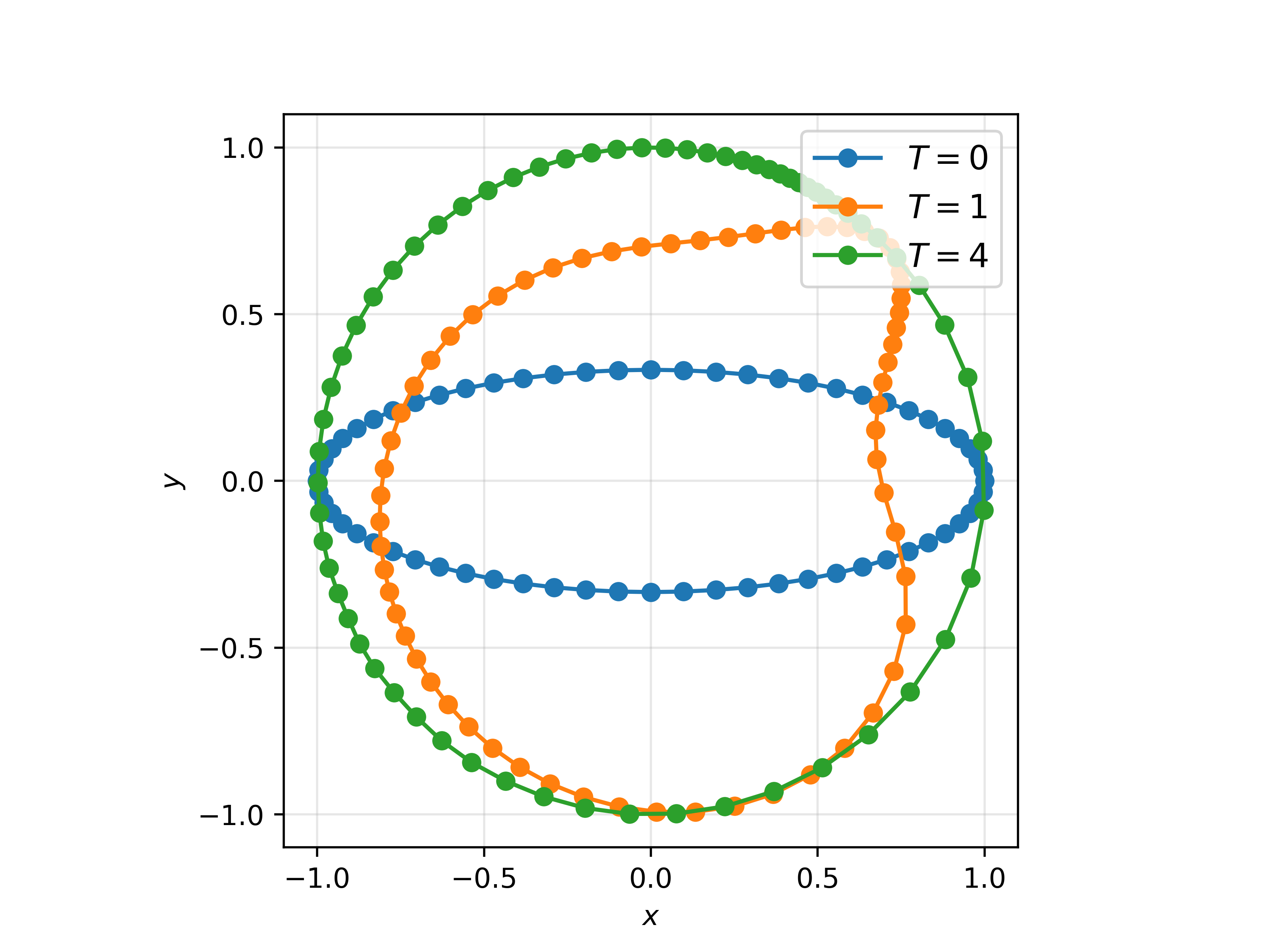}}
	\caption{Evolution of a curve $\Gamma(t)$ under velocity field $v$ in \eqref{eq:exp_v1}}
	\label{fig:ATV}
\end{figure}



\begin{appendices}
	
\section{Surface calculus formulas}
\label{sec:surf_calc}

Given a smooth curve $\Gamma$ (with or without boundary) in $\R^2$ and $u\in C^\infty(\Gamma)$, we denote by $\ud_i u, i = 1, 2,$ the $i$th component of the tangent vector $\nabla_\Gamma u$ in $\R^2$. The corresponding Leibniz rule, chain rule, integration-by-parts formula, commutators, and the evolution equation of normal vector, are summarized below.
\begin{lemma}\label{lemma:ud}
	Let $\Gamma$ and $ \Gamma^\prime$ be two smooth curves that are possibly open, such as smooth pieces of some finite element curves, and let $f, h \in C^\infty(\Gamma)$ and $g\in C^\infty(\Gamma^\prime; \Gamma)$ be given functions. Then the following results hold. 
	\begin{itemize}
		\item[1.]  $\ud_i(fh) = \ud_i f h + f\ud_i h$ on $\Gamma$.
		\item[2.] $\ud_i(g\circ f) = (\ud_j g\circ f)\, \ud_i f$ on $\Gamma'$.
		\item[3.] $\int_{\Gamma}f \ud_i h = -\int_{\Gamma}\ud_i f h + \int_{\Gamma}f h H n_i + \int_{\partial\Gamma}f h \mu_i$ where $n, \mu$ are the normal and co-normal (tangential) direction, respectively, and $H:=\ud_i n_i$ (with the Einstein notation) is the mean curvature, i.e. the trace of the second fundamental form.
		\item[4.] $\ud_i \ud_j f = \ud_j \ud_i f + n_i H_{jl} \ud_l f -  n_j H_{il} \ud_l f$, where $H_{ij} := \ud_i n_j = \ud_j n_i$.
		\item[5.] If $\Gamma$ evolves under the velocity field $v$, and $G_T := \bigcup_{t\in [0, T]}\Gamma(t) \times \{t\}$, then 
		$$\md(\ud_i f) = \ud_i (\md f )- (\ud_i v_j - n_i n_l \ud_j v_l)\ud_j f \quad\forall\, f \in C^2(G_T) ,$$
		where $\md$ denotes the material derivative with respect to $v$.
		\item[6.] If $f, h \in C^2(G_T)$ then 
		$$\frac{\d}{\d t}\int_{\Gamma} f h = \int_{\Gamma} \md f h + \int_{\Gamma} f\md h + \int_{\Gamma} f h (\nabla_\Gamma\cdot v).$$
		The divergence is defined as $\nabla_\Gamma\cdot v := \ud_i v_i$, which coincides with the intrinsic divergence on the curve if $v$ is a tangential vector field on $\Gamma$. Since the Lagrange interpolation commutes with the material time derivative, it is straightforward to check in the local coordinates that an analogous result also holds for the mass lumping integral, i.e., 
		$$
		\frac{\d}{\d t}\int_{\Gamma_h}^h \tilde f \tilde h = \int_{\Gamma_h}^h \md \tilde f \tilde h + \int_{\Gamma_h}^h \tilde f\md \tilde h + \int_{\Gamma_h}^h \tilde f \tilde h (\nabla_{\Gamma_h}\cdot v_h) ,
		$$
		where $\Gamma_h$ is a finite element curve moving with polynomial velocity $v_h\in S_h(\Gamma_h)^2$ (mass lumping is well defined on $\Gamma_h$), and $\tilde f,\tilde h$ are continuous functions defined on $\bigcup_{t\in [0, T]}\Gamma_h(t)\times\{t\}$.  
		\item[7.] The evolution of the unit normal vector $n$ of the curve $\Gamma$ with respect to the velocity field $v$ satisfies the following relation: 
		$$
		\md n_i = -\ud_i v_j n_j .
		$$
	\end{itemize}
\end{lemma}
\begin{proof}
	The first two relations are obvious from the local formula of $\ud$ (cf. \cite[Eq. (5.1)]{BL24FOCM}). The third relation is shown in \cite[Theorem 2.10]{Dziuk2013b}. The fourth and fifth equalities are proved in \cite[Lemma 2.4 and 2.6]{Dz13b}, and the proof of the sixth and last formulae can be found in \cite[Appendix A]{DE2007IJNA} and \cite[p. 33]{Mantegazza2011} respectively.
\end{proof}

\section{Discrete norms}
\label{sec:disc-norm}

Since the weights of the Gauss--Lobatto quadrature are positive, the discrete $L^p$ norm defined by 
$$
\| f \|_{L^p_h(\Ghsm)}  := \Big( \int_\Ghsm^h |f|^p \Big)^{\frac1p} 
= \Big( \sum\limits_{K \subset \Ghsm} \int_{K_{\rm f}^0} I_K \big( |f \circ F_K |^p |\nabla_{K_{\rm f}^0} F_K| \big) \Big)^{\frac1p}
$$ 
is indeed a norm on the finite element space $S_h(\Ghsm)$ because $\| f \|_{L^p_h(\Ghsm)}=0$ iff $f=0$ at all the nodes of $\Ghsm$. In addition, this discrete $L^p$ norm is also well defined for functions which are piecewise continuous on $\Ghsm$. Its basic properties are summarized below. 

\begin{lemma}\label{lemma:lump}
	The following relations hold for all finite element functions $f_h\in S_h(\Ghsm)$ and piecewise continuous functions $w_1, w_2, w_3$ on $\Ghsm$: 
	\begin{align}
		\| f_h \|_{L^p_h(\Ghsm)} &\sim \| f_h \|_{L^p(\Ghsm)} , \notag\\
		\| \nabla_\Ghsm f_h \|_{L^p_h(\Ghsm)} &\sim \| \nabla_\Ghsm f_h \|_{L^p(\Ghsm)} , \notag\\
		\Big|\int_\Ghsm^h w_1 w_2 w_3 \Big| &\lesssim \| w_1 \|_{L^\infty(\Ghsm)} \| w_2 \|_{L^2_h(\Ghsm)} \| w_3 \|_{L^2_h(\Ghsm)}  \notag .
	\end{align}
\end{lemma}

\section{Averaged normal vectors}
\label{sec:super1}

On the interpolated curve $\Ghsm$ we can define the averaged normal vector $\bar n_{h,*}^m$ similarly as $\bar n_h^m$ on $\Gamma_h^m$, which is defined in \eqref{eq:bar_n}. Namely, we define $\bar n_{h,*}^m\in S_h(\Ghsm)^2$ to be the unique finite element function satisfying the following relation: 
\begin{align}\label{eq:bar_n*}
	\int_\Ghsm^h \bar n_{h,*}^m\cdot \phi_h = \int_\Ghsm^h \hat n_{h,*}^m\cdot \phi_h 
	\quad\forall\, \phi_h \in S_h(\Ghsm)^2.
\end{align}
Since \eqref{eq:bar_n*} only involves nodal values, it follows that 
\begin{align}\label{eq:bar_n*1}
	\int_\Ghsm^h \bar n_{h,*}^m\cdot \phi = \int_\Ghsm^h \hat n_{h,*}^m\cdot \phi \quad\forall\, \phi\in C(\Ghsm)^2.
\end{align}
It is straightforward to verify the following relations: 
\begin{align}\label{bar-n-hat-n}
	\bar n_{h,*}^m(p) &= \hat n_{h,*}^m(p) \hspace{91pt} \mbox{if $p$ is an interior node of an element} , \notag \\
	\bar n_{h,*}^m(p) &= \frac{| w_{K_1}(p) | |K_{1\rm f}^0| \,\, \hat n_{h,*}^m(p-) }{| w_{K_1}(p)||K_{1\rm f}^0| + | w_{K_2}(p)||K_{2\rm f}^0|} 
	+  \frac{| w_{K_2}(p) ||K_{2\rm f}^0| \,\, \hat n_{h,*}^m(p+) }{| w_{K_1}(p)| |K_{1\rm f}^0|+ | w_{K_2}(p)||K_{2\rm f}^0|}   \\
	&\hspace{135pt}\mbox{if $p=K_1\cap K_2$ for two elements $K_1$ and $K_2$} , \notag
\end{align}
where $w_{K}(p)= \nabla_{K_{\rm f}^0} F_{K}\circ F_{K}^{-1} (p) $ for $p\in K$, with $\hat n_{h,*}^m(p-) $ and $\hat n_{h,*}^m(p+) $ denoting the left (from $K_1$) and right (from $K_2$) values of the piecewisely defined normal vector $\hat n_{h,*}^m$ on $\Ghsm$. Therefore, the amplitude of $\bar n_{h,*}^m$ at the nodes satisfies the following estimates:
\begin{align}\label{eq:nhs_bar_len}
	\begin{split}
		|\bar n_{h,*}^m(p)| &= 1 \hspace{89pt}\mbox{if $p$ is an interior node of an element} , \\
		| \bar n_{h,*}^m(p) | &\leq 1, \quad \big| | \bar n_{h,*}^m(p) | - 1\big| \le C | \hat n_{h,*}^m(p+) - \hat n_{h,*}^m(p-) |^2 \le C_{\kappa_m} h^{2k} \\
		&\hspace{107pt}\mbox{if $p=K_1\cap K_2$ for two elements $K_1$ and $K_2$}  .
	\end{split}
\end{align}

By treating the shape regularity constant explicitly, a slight modification to \cite[Lemma 3.8]{BL24} leads to the following consistency estimates for the averaged normal vectors.
\begin{lemma}\label{lemma:n_bar_app}
	The following approximation properties of $\nbhsm$ and $\nbhm$ hold:
	\begin{align}
		\begin{aligned}
			\| \nbhsm - I_h \nsm \|_{L^{{2}}(\Ghsm)} &\lesssim {(1 + \ksm) h^k} , \\
			\| \nbhm - I_h \nsm \|_{L^2(\Ghsm)} &\lesssim (1 + \ksm) h^k + \|\nabla_\Ghsm \ehm \|_{L^2(\Ghsm)} , \\ 
			\| \nbhsm - \hat n_{h,*}^m \|_{L^2(\Ghsm)} &\lesssim (1 + \ksm) h^k . 
		\end{aligned}
	\end{align}
\end{lemma}
	
\section{Super-approximation estimates}
\label{sec:super}

In the framework of the projection error and Gauss--Lobatto mass lumping, the following standard super-approximation results can be found in \cite[Section 3.5]{BL24FOCM} and \cite[Lemma 3.6]{GLW22}.
\begin{lemma}\label{lemma:super_conv}
	The following estimates hold for any piecewise smooth function $f$ and finite element functions $\phi_h,v_h, w_h\in S_h(\Ghsm)$: 
	\begin{align*}
		\| (1 - I_h)(f \phi_h) \|_{L^2(\Ghsm)} &\lesssim \| f \|_{W_h^{k+1,\infty}(\Ghsm)} h \| \phi_h \|_{L^{2}(\Ghsm)} , \notag\\
		\| \nabla_\Ghsm (1 - I_h)(f \phi_h) \|_{L^2(\Ghsm)} &\lesssim \| f \|_{W_h^{k+1,\infty}(\Ghsm)} h \| \phi_h \|_{H^{1}(\Ghsm)} ,\\
		\| (1 - I_h)(v_h w_h) \|_{L^2(\Ghsm)} &\lesssim  h^2 \| v_h \|_{W^{1,\infty}(\Ghsm)} \| w_h \|_{H^{1}(\Ghsm)} ,\\
		\| \nabla_{\Ghsm}(1 - I_h)(v_h w_h) \|_{L^2(\Ghsm)} &\lesssim  h \| v_h \|_{W^{1,\infty}(\Ghsm)} \| w_h \|_{H^{1}(\Ghsm)} . 
	\end{align*}
\end{lemma}
\begin{lemma}\label{lemma:super_conv2}
	Let $f$ be a function which is smooth on every element $K$ of $\Ghsm$, and assume that the pull-back function $f\circ F_K $ vanishes at all the Gauss--Lobatto points of the flat segment $K_{\rm f}^0$ for every element $K$ of $\Ghsm$. Then the following estimate holds: 
	\begin{align}
		\Big|\int_\Ghsm f \d\xi \Big| 
		\lesssim h^{2k} \| f \|_{W^{2k,1}_h(\Ghsm)} ,
	\end{align}
	where $\|\cdot\|_{W^{2k,1}_h(\Ghsm)}$ denotes the piecewise $W^{2k,1}$ norm . 
\end{lemma}
\begin{lemma}\label{lemma:T<=N2}
	For sufficiently small $h$, the following estimates hold:
	\begin{align} 
		\| I_h \Tbhsm \hat e_{h}^m \|_{L^2(\Ghsm)}
		&\lesssim
		h\| \hat e_{h}^m\cdot \bar n_{h,*}^m \|_{L_h^2(\Ghsm)}  , \\
		\|\hat e_{h}^m \|_{L^2(\Ghsm)}
		&\le
		2\| \hat e_{h}^m\cdot \bar n_{h,*}^m \|_{L_h^2(\Ghsm)}  . \label{eq:NT_stab_L22} 
	\end{align}
\end{lemma}
	
\section{Proof of Lemma \ref{lemma:NT_stab_ref}}
\label{sec:appndix_tan_stab}

	Given $f_h,g_h\in S_h(\Ghm)^2$, we derive from Lemma \ref{lemma:geo-perturb} and \ref{lemma:e-blinear} that
	\begin{align}\label{eq:tan_stab1}
		&\Big| \int_{\Gamma_h^m} \nabla_{\Gamma_h^m} I_h \Nbhm f_h \cdot  \nabla_{\Gamma_h^m}  I_h \Tbhm g_h \Big| 
		\notag\\ 
		&= \Big| \Big( \int_{\Gamma_h^m} \nabla_{\Gamma_h^m} I_h \Nbhm f_h \cdot  \nabla_{\Gamma_h^m}  I_h \Tbhm g_h
		- \int_{\Ghsm} \nabla_{\Ghsm} I_h \Nbhm f_h \cdot  \nabla_{\Ghsm}  I_h \Tbhm g_h \Big)
		\notag\\ 
		&\quad+ \Big( \int_{\Ghsm} \nabla_{\Ghsm} I_h \Nbhm f_h \cdot  \nabla_{\Ghsm}  I_h \Tbhm g_h
		- \int_{\Gm} \nabla_{\Gm} (I_h \Nbhm f_h)^l \cdot  \nabla_{\Gm}  (I_h \Tbhm g_h)^l \Big)
		\notag\\ 
		&\quad+ \int_{\Gm} \nabla_{\Gm} (I_h \Nbhm f_h)^l \cdot  \nabla_{\Gm}  (I_h \Tbhm g_h)^l \Big|
		\notag\\ 
		&\lesssim h^{-1/2}\big((1 + \ksm) h^{k+1} + \| \nabla_\Ghsm \ehm \|_{L^2(\Ghsm)} \big) 
		\notag\\
		&\qquad\qquad\qquad\qquad\qquad\qquad\times
		\| \nabla_{\Ghsm} I_h \Nbhm f_h \|_{L^2(\Ghsm)}  \|  \nabla_{\Ghsm}  I_h \Tbhm g_h \|_{L^2(\Ghsm)} 
		\notag\\ 
		&\quad +\Big| \int_{\Gm} \nabla_{\Gm} (I_h \Nbhm f_h)^l \cdot  \nabla_{\Gm} ( I_h \Tbhm g_h)^l \Big|
		,
	\end{align}
Besides,
	\begin{align}\label{eq:tan_stab212}
		&\Big| \int_{\Gm} \nabla_{\Gm} (I_h \Nbhm f_h)^l \cdot  \nabla_{\Gm}  (I_h \Tbhm  g_h)^l \Big|
		\notag\\ 
		&= \Big| \int_{\Gm} \nabla_{\Gm} (\Nbhm f_h)^l \cdot  \nabla_{\Gm} (\Tbhm g_h)^l
		\notag\\ 
		&\quad- \int_{\Gm} \nabla_{\Gm} ((1 - I_h)\Nbhm f_h)^l \cdot  \nabla_{\Gm} (I_h \Tbhm g_h)^l
		\notag\\ 
		&\quad- \int_{\Gm} \nabla_{\Gm} (I_h\Nbhm f_h)^l \cdot  \nabla_{\Gm} ((1 - I_h) \Tbhm g_h)^l
		\notag\\ 
		&\quad- \int_{\Gm} \nabla_{\Gm}  ((1 - I_h) \Nbhm f_h)^l \cdot  \nabla_{\Gm} ( (1 - I_h) \Tbhm  g_h)^l \Big|
		\notag\\ 
		&\lesssim \Big| \int_{\Gm} \nabla_{\Gm} (\Nbhm f_h)^l \cdot  \nabla_{\Gm} (\Tbhm g_h)^l \Big|
		\notag\\ 
		&\quad+ h \| f_h \|_{H^1(\Ghsm)}  \| \nabla I_h \Tbhm g_h \|_{L^2(\Ghsm)} \notag\\
		&\quad+ h \| \nabla_{\Ghsm} I_h\Nbhm f_h  \|_{L^2(\Ghsm)} \| g_h \|_{H^1(\Ghsm)} \notag\\
		&\quad
		+ h^2 \| f_h \|_{H^1(\Ghsm)} \|  g_h \|_{H^1(\Ghsm)}
		\notag\\
		&\lesssim \Big| \int_{\Gm} \nabla_{\Gm} (\Nbhm f_h)^l \cdot  \nabla_{\Gm} (\Tbhm g_h)^l \Big|
		\notag\\ 
		&\quad+ \| f_h \|_{L^2(\Ghsm)}  \| g_h \|_{H^1(\Ghsm)} ,
	\end{align}
	where we have used the super-approximation estimates (cf. Lemma \ref{lemma:super_conv}): 
	\begin{align}
		\begin{aligned}
			\| \nabla_{\Gm} ((1 - I_h)\Nbhm f_h)^l \|_{L^2(\Gm)}  &\lesssim h \|\nbhm\|_{W^{1,\infty}(\Ghsm)} \| f_h \|_{H^1(\Ghsm)} ,\\
			\| \nabla_{\Gm} ( (1 - I_h) \Tbhm g_h)^l \|_{L^2(\Gm)} & \lesssim h \|\nbhm\|_{W^{1,\infty}(\Ghsm)} \| g_h \|_{H^1(\Ghsm)} . \notag
		\end{aligned}
	\end{align}
	The boundedness of $\|\Nbhm\|_{W^{1,\infty}(\Ghsm)} $ and $\|\Tbhm\|_{W^{1,\infty}(\Ghsm)} $ follows from the $W^{1,\infty}$ estimate of $\nbhm$ in \eqref{eq:n-bbd}.

	The first term on the right-hand side of \eqref{eq:tan_stab212} can be further decomposed into 
	\begin{align}\label{eq:tan_stab3}
		& \Big| \int_{\Gm} \nabla_{\Gm} (\Nbhm f_h)^l \cdot  \nabla_{\Gm} (\Tbhm g_h)^l \Big|
		\notag\\ 
		&= \Big| \int_{\Gm} \nabla_{\Gm} (N_*^m\Nbhm f_h)^l \cdot  \nabla_{\Gm} (T_*^m\Tbhm g_h)^l
		\notag\\ 
		&\quad+ \int_{\Gm} \nabla_{\Gm} ((\bar N_h^{m} - N_*^m)\Nbhm f_h)^l \cdot  \nabla_{\Gm} (T_*^m \Tbhm g_h)^l 
		\notag\\ 
		&\quad+ \int_{\Gm} \nabla_{\Gm} (\bar N_h^{m}\Nbhm f_h)^l  \cdot  \nabla_{\Gm} ((\bar T_h^{m} - T_*^m)\Tbhm g_h)^l \Big|
		\notag\\ 
		&\lesssim \Big| \int_{\Gm} \nabla_{\Gm} (N_*^m\Nbhm  f_h)^l \cdot  \nabla_{\Gm} (T_*^m \Tbhm g_h)^l \Big|
		\notag\\ 
		&\quad+ \Big(\| \nabla_\Ghsm (\Nbhm - \Nsm) \|_{L^2(\Ghsm)} \|  \Nbhm f_h \|_{L^\infty(\Ghsm)} \notag\\
		&\quad+ \|  \Nbhm - \Nsm \|_{L^\infty(\Ghsm)} \| \nabla_\Ghsm \Nbhm f_h \|_{L^2(\Ghsm)} \Big) \|  \Tbhm  g_h \|_{H^1(\Ghsm)} \notag\\
		&\quad+ \Big(\min\big\{\| \nabla_\Ghsm ( \Tbhm - \Tsm) \|_{L^2(\Ghsm)} \|  \Tbhm g_h \|_{L^\infty(\Ghsm)}, 
		\notag\\
		&\qquad\qquad\qquad
		\| \nabla_\Ghsm ( \Tbhm - \Tsm) \|_{L^\infty(\Ghsm)} \|  \Tbhm g_h \|_{L^2(\Ghsm)}\big\} \notag\\
		&\quad+  \|  \Tbhm - \Tsm \|_{L^\infty(\Ghsm)} \| \nabla_\Ghsm \Tbhm g_h \|_{L^2(\Ghsm)} \Big) \|  \Nbhm f_h \|_{H^1(\Ghsm)} \notag\\
		&\hspace{120pt} \mbox{(product rule of differentiation is used)}\notag\\
		&\lesssim \Big| \int_{\Gm} \nabla_{\Gm} (N_*^m\Nbhm  f_h)^l \cdot  \nabla_{\Gm} (T_*^m \Tbhm g_h)^l \Big|
		\notag\\ 
		&\quad+ 
		((1+\ksm)h^{k-3/2} + h^{-3/2} \| \nabla_\Ghsm \ehm \|_{L^2(\Ghsm)}) 
		\|  f_h \|_{L^2(\Ghsm)}
		\|  g_h \|_{H^1(\Ghsm)} \notag\\
		&\quad+ 
		((1+\ksm)h^{k-3/2} + h^{-3/2} \| \nabla_\Ghsm \ehm \|_{L^2(\Ghsm)}) 
		\notag\\
		&\qquad
		\times
		\min\{ h^{-1/2} \| f_h \|_{L^2(\Ghsm)}
		\|  g_h \|_{L^\infty(\Ghsm)} ,
		h^{-1}\|  f_h \|_{L^2(\Ghsm)}
		\|  g_h \|_{L^2(\Ghsm)}
		\} 
		,
	\end{align} 
	where in the last inequality we have applied Lemma \ref{lemma:n_bar_app}.
	For the first term on the right-hand side of \eqref{eq:tan_stab3}, we proceed as follows:
	\begin{align*}
		& \Big| \int_{\Gm} \nabla_{\Gm} [N^m(\Nbhm f_h)^l ]\cdot  \nabla_{\Gm} [T^m(\Tbhm g_h)^l] \Big|
		\notag\\ 
		&= \Big| \int_{\Gm} (\nabla_{\Gm} N^m) N^m (\Nbhm f_h)^l \cdot  (\nabla_{\Gm} T^m) T^m (\Tbhm g_h)^l
		\notag\\ 
		&\quad+ \int_{\Gm} N^m \nabla_{\Gm} [N^m(\Nbhm f_h)^l] \cdot T^m \nabla_{\Gm} [T^m(\Tbhm g_h)^l]
		\notag\\ 
		&\quad+ \int_{\Gm} (\nabla_{\Gm} N^m) N^m (\Nbhm f_h)^l \cdot T^m  \nabla_{\Gm} [T^m(\Tbhm g_h)^l]
		\notag\\ 
		&\quad+ \int_{\Gm} N^m \nabla_{\Gm} [N^m(\Nbhm f_h)^l] \cdot  (\nabla_{\Gm} T^m) T^m (\Tbhm g_h)^l \Big| ,
	\end{align*}
	where the second term on the right-hand side is zero due to the orthogonality between the two projections $N^m$ and $T^m$. For the last term on the right-hand side, we can remove the gradient from $N^m(\Nbhm f_h)^l$ via integration by parts. This leads to the following estimate: 
	\begin{align}\label{eq:tan_stab51}
		&\Big| \int_{\Gm} \nabla_{\Gm} N^m(\Nbhm f_h)^l \cdot  \nabla_{\Gm} T^m(\Tbhm g_h)^l \Big| \notag\\ 
		&\lesssim \| \Nbhm f_h \|_{L^2(\Ghsm)}  \| \Tbhm g_h \|_{H^1(\Ghsm)}
		\notag\\
		&\lesssim \| f_h \|_{L^2(\Ghsm)}  \| g_h \|_{H^1(\Ghsm)} .
	\end{align}
	
	Then Lemma \ref{lemma:NT_stab_ref} follows from combining \eqref{eq:tan_stab1}--\eqref{eq:tan_stab51} and the induction hypothesis \eqref{eq:ind_hypo1}.

\section{Proof of Lemma \ref{lemma:e-convert}}
\label{sec:e-convert}

To show the stability of converting $\| \eM\cdot \nbhsm \|_{L^2_h(\Ghsm)}^2$ to $\| \ehM\cdot \nbhsM \|_{L^2_h(\hat\Gamma_{h,*}^{m+1})}^2$, we decompose their difference into the following three parts: 
\begin{align}
	&\| \ehM\cdot \nbhsM \|_{L^2_h(\hat\Gamma_{h,*}^{m+1})}^2 - \| \eM\cdot \nbhsm \|_{L^2_h(\Ghsm)}^2 \notag\\
	&= \| \ehM\cdot \bar n_{h,*}^{m+1} \|_{L^2_h(\hat\Gamma_{h,*}^{m+1})}^2 - \| \ehM\cdot \bar n_{h,*}^{m+1} \|_{L^2_h(\hat\Gamma_{h,*}^{m})}^2 
	&&\mbox{(change of $\hat\Gamma_{h,*}^{m+1}$ to $\hat\Gamma_{h,*}^{m}$)}\notag\\
	&\quad+ \| \ehM\cdot \bar n_{h,*}^{m+1} \|_{L^2_h(\hat\Gamma_{h,*}^{m})}^2  - \| \ehM\cdot \bar n_{h,*}^{m} \|_{L^2_h(\Ghsm)}^2 
	&&\mbox{(change of $\bar n_{h,*}^{m+1}$ to $\bar n_{h,*}^{m}$)} \notag\\
	&\quad+ \| \ehM\cdot \bar n_{h,*}^{m} \|_{L^2_h(\hat\Gamma_{h,*}^{m})}^2 - \| \eM\cdot \bar n_{h,*}^{m} \|_{L^2_h(\hat\Gamma_{h,*}^{m})}^2 
	&&\mbox{(change of $\ehM$ to $\eM$)} \notag\\
	&=: M_1^m + M_2^m + M_3^m .
	\notag
\end{align}
By the fundamental theorem of calculus, \eqref{eq:hat_X_s_diff1} and the norm equivalence of curves $\Ghsm$ and $\hat\Gamma_{h,*}^{m+1}$ in Section \ref{sec:err-est}, we know
\begin{align}\label{eq:M1}
	M_1^m
	&=  \| \ehM\cdot \bar n_{h,*}^{m+1} \|_{L^2_h(\hat\Gamma_{h,*}^{m+1})}^2 - \| \ehM\cdot \bar n_{h,*}^{m+1} \|_{L^2_h(\Ghsm)}^2  \notag\\
	&\lesssim \| \nabla_\Ghsm (\hat X_h^{m+1} - \hat X_h^m) \|_{L^\infty(\Ghsm)} \| \ehM \|_{L^2(\Ghsm)}^2 \notag\\
	&\lesssim \tau \| \ehM \|_{L^2(\Ghsm)}^2
	 \quad\mbox{(here \eqref{eq:hat_X_s_diff1} is used)}.
\end{align}
Eq. \eqref{eq:n-diff} implies that 
\begin{align}\label{eq:M2}
	M_2^m
	&=  \| \ehM\cdot \bar n_{h,*}^{m+1} \|_{L^2_h(\hat\Gamma_{h,*}^{m})}^2 - \| \ehM\cdot \bar n_{h,*}^{m} \|_{L^2_h(\Ghsm)}^2  \lesssim \tau \| \hat e_h^{m+1} \|_{L^2(\Ghsm)}^2 .
\end{align}
The term $M_3^m$ can be furthermore decomposed into several parts as follows: 
\begin{align}
	M_3^m 
	&= \| \ehM\cdot \bar n_{h,*}^{m} \|_{L^2_h(\Ghsm)}^2  - \| \eM\cdot \bar n_{h,*}^{m} \|_{L^2_h(\Ghsm)}^2 \notag\\
	&= \int_\Ghsm^h (\ehM - \eM)\cdot \bar n_{h,*}^{m}  (\ehM + \eM)\cdot \bar n_{h,*}^{m} \notag\\
	&= \int_\Ghsm^h \Big(I_h \TsM(\ehM - \eM) + f_h \Big) \cdot \hat n_{h,*}^{m}  (\ehM + \eM) \cdot \bar n_{h,*}^{m} \notag\\
	&= -\int_\Ghsm^h I_h (\TsM - \Tsm) \eM \cdot (\hat n_{h,*}^{m} -\nsm)  (\ehM + \eM) \cdot \bar n_{h,*}^{m} \notag\\
	&\quad -\int_\Ghsm^h I_h \Tsm (\eM - \ehm - \tau I_h \Tsm v^m) \cdot (\hat n_{h,*}^{m} -\nsm)  (\ehM + \eM) \cdot \bar n_{h,*}^{m} \notag\\
	&\quad - \int_\Ghsm^h \tau I_h \Tsm v^m \cdot (\hat n_{h,*}^{m} -\nsm)  (\ehM + \eM) \cdot \bar n_{h,*}^{m} \notag\\
	&\quad+ \int_\Ghsm^h I_h \TsM \eM \cdot (\nsM - \nsm)  (\ehM + \eM) \cdot \bar n_{h,*}^{m} \notag\\
	&\quad+ \int_\Ghsm^h f_h \cdot \hat n_{h,*}^{m}  (\ehM + \eM) \cdot \bar n_{h,*}^{m} \notag\\
	&=: \sum_{i=1}^5 M_{3i}^m .
	\notag
\end{align}
From \eqref{normal-intpl} and \eqref{nsM-nsm-nodes1}, we get
\begin{align}
	M_{31}^m &\lesssim (1 + \ksm) h^{k-1/2} \tau (\| \ehM \|_{L^2(\Ghsm)}^2 + \| \eM \|_{L^2(\Ghsm)}^2) , \notag\\
	M_{33}^m &\lesssim (1 + \kappa_{*,m}) h^k \tau (\| \ehM \|_{L^2(\Ghsm)} + \| \eM \|_{L^2(\Ghsm)}) , \notag\\
	M_{34}^m &\lesssim \tau (\| \ehM \|_{L^2(\Ghsm)}^2 + \| \eM \|_{L^2(\Ghsm)}^2) , \notag
\end{align}
and in view of \eqref{eq:vel_T}
\begin{align}
	M_{32}^m &\lesssim 
	\tau (1 + h^{-2} \| \nabla_\Ghsm \ehm \|_{L^2(\Ghsm)}) \big((\tau + (1 + \ksm) h^{k}) + \| \ehm \|_{H^1(\Ghsm)} \big)
	\notag\\
	&\quad\times
	(1 + \kappa_{*,m}) h^k (\| \ehM \|_{L^2(\Ghsm)} + \| \eM \|_{L^2(\Ghsm)}) . \notag
\end{align}
Finally, by using the geometric relation \eqref{eq:geo_rel_2} and \eqref{nsM-nsm-nodes1}, as well as the relation $(1 - \nsm(\nsm)^\top) \ehm =0$ at the nodes, we have
\begin{align*}
		M_{35}^m  
		&= \int_\Ghsm^h f_h\cdot \hat n_{h,*}^{m}  (\ehM + \eM)\cdot \bar n_{h,*}^{m} \notag\\
		&\lesssim \| (1 - \nsM(\nsM)^\top) \eM \|_{L_h^4(\Ghsm)}^2 (\| \ehM \|_{L^2(\Ghsm)} + \| \eM \|_{L^2(\Ghsm)}) \notag\\
		&\lesssim \| (1 - \nsm(\nsm)^\top) (\eM - \ehm) \|_{L_h^4(\Ghsm)}^2 (\| \ehM \|_{L^2(\Ghsm)} + \| \eM \|_{L^2(\Ghsm)}) \notag\\
		&\quad+ \| (\nsM(\nsM)^\top - \nsm(\nsm)^\top) \eM \|_{L_h^4(\Ghsm)}^2 (\| \ehM \|_{L^2(\Ghsm)} + \| \eM \|_{L^2(\Ghsm)}) \notag\\
		&\lesssim ( \| I_h \Tsm (\eM - \ehm - \tau I_h \Tsm v^m) \|_{L^4(\Ghsm)}^2 + \tau^2 ) (\| \ehM \|_{L^2(\Ghsm)} + \| \eM \|_{L^2(\Ghsm)}) \notag\\
		&\lesssim  \tau^2 (\| \ehM \|_{L^2(\Ghsm)} + \| \eM \|_{L^2(\Ghsm)})  ,
\end{align*}
where we have used \eqref{eq:de-bbd} in the last step. 
Obviously, $M_{33}^m$, $M_{34}^m$ and $M_{35}^m$ are leading order terms under the induction hypothesis, and by collecting the estimates of $M_{3j}^m$, $j=1,\dots,5$, we obtain the following estimate: 
\begin{align}\label{eq:M3}
		M_3^m 
		&\lesssim 
		\tau (\| \ehM \|_{L^2(\Ghsm)}^2 + \| \eM \|_{L^2(\Ghsm)}^2
		+
		\| \ehm \|_{L^2(\Ghsm)}^2
		)
		\notag\\
		&\quad
		+
		\big((1 + \kappa_{*,m}) h^k \tau + \tau^2 \big) (\| \ehM \|_{L^2(\Ghsm)} + \| \eM \|_{L^2(\Ghsm)}) 
		.
\end{align}
The proof is complete by combining \eqref{eq:M1}--\eqref{eq:M3}.

\end{appendices}

\renewcommand{\refname}{\bf References}

\bibliographystyle{abbrv}
\bibliography{MCF_Vt}

\end{document}